\documentclass[11pt,a4paper]{article}
\usepackage{amsthm}
\newtheorem{theorem}{Theorem}
\newtheorem{lemma}[theorem]{Lemma}

\usepackage{amsmath,amsthm,amssymb,amsfonts}
\usepackage{bm}
\usepackage{geometry}
\usepackage{graphicx}
\usepackage{mathtools}
\usepackage{hyperref}
\usepackage{booktabs}
\usepackage{algorithm}
\usepackage{algpseudocode}
\theoremstyle{definition}

\newtheorem{remark}{Remark}

\title{Learned Dictionaries with Total Variation and Non-Negativity for Single-Cell Microscopy: Convergence Theory and Deterministic Multi-Channel Cell Feature Unification}
\author{
	Erdem Altunta\c{c}\\
	Aegis Digital Technologies\\
	Neubertstr.~21, 01307 Dresden, Germany\\
	\texttt{erdem.altuntac@aegis-digital.tech}\\
	\url{https://www.aegis-digital.tech}
}
\date{April 26, 2026 \\ \small Submitted to the \emph{Journal of Mathematical Imaging and Vision}}

\newcommand{\R}{\mathbb{R}}
\newcommand{\norm}[1]{\left\lVert #1 \right\rVert}
\newcommand{\ip}[2]{\left\langle #1, #2 \right\rangle}
\newcommand{\prox}{\mathrm{prox}}
\newcommand{\TV}{\mathrm{TV}}

\begin{document}
	\maketitle
	
	
	\begin{abstract}
		We introduce a variational dictionary-based learning algorithm with hybrid penalization for
		single-cell microscopy signals. The cost functional couples a least-squares data fidelity
		term with total-variation (TV) regularization and a non-negativity constraint, promoting
		edge-preserving, physically meaningful reconstructions under heterogeneous backgrounds and
		imaging artifacts. We formulate the learning task with an explicit unitary (orthonormal)
		constraint on the dictionary operator, ensuring well-conditioned representations and
		predictable numerical behavior. The resulting optimization problem is solved by
		an alternating proximal-gradient scheme that combines smooth updates with closed-form
		proximal steps for non-smooth penalties. We prove that the PDHG iterates converge to
		the regularized minimizer under an explicit step-size condition ($\tau\sigma < 1/8$),
		and that when the true solution satisfies a variational source condition (VSC), the
		regularized solution converges to the true solution at the optimal $O(\delta)$ rate
		under a noise-proportional regularization parameter choice $\lambda \propto \delta$.
		
		Beyond reconstruction, we address the problem of multi-channel cell feature unification:
		given five imaging channels of the BSCCM dataset (DPC Left, Right, Top, Bottom, and Brightfield),
		we propose a \emph{deterministic} approach to synthesize a unified single-cell representation.
		Rather than probabilistic latent encodings, we learn a family of \emph{per-channel} unitary
		dictionaries $\{D^{(c)}\}_{c=1}^{C}$, each adapted to the optical physics of its channel, and
		concatenate the resulting per-channel sparse codes into a single channel-agnostic cell descriptor
		$\varphi_j \in \mathbb{R}^{CK}$ under a fixed channel ordering. No cross-channel atom
		correspondence is assumed; the descriptor aggregates information from all channels into a
		single vector suitable for downstream classification. This deterministic unification
		strategy is mathematically transparent, reproducible, and directly compatible with the
		clinical requirement that AI systems for diagnostics must be interpretable and auditable.
		On the BSCCM-tiny subset ($N{=}1000$ cells, $K{=}512$ atoms) the framework reaches
		per-channel reconstruction fidelities of $97.06$--$97.54\%$ on the four DPC channels and
		$94.79\%$ on Brightfield, with fully deterministic (bit-identical) iterates across independent
		runs. Biological validation on the $N{=}28$ expert-labelled subset yields unsupervised
		$k{=}2$ lymphoid-vs-myeloid separation at ARI$=0.575$, NMI$=0.471$
		(permutation $p < 0.0001$; bootstrap 95\% CI: ARI $[-0.11,\,0.71]$).
	\end{abstract}
	
	\section{Introduction}
	Dictionary learning offers a mathematically transparent mechanism for encoding
	high-dimensional signals through a compact collection of learned atoms.
	In computational microscopy, and in single-cell imaging in particular,
	achieving a useful representation requires balancing three competing demands:
	(i) \emph{sparsity} (so that only a small number of atoms are activated per cell,
	isolating the most informative structures),
	(ii) \emph{structure preservation} (maintaining cell boundaries and subcellular
	texture across the representation), and
	(iii) \emph{numerical stability} (sustaining reliable optimization over large,
	heterogeneous datasets).
	
	This manuscript introduces a \emph{new variational dictionary learning algorithm}
	with \emph{hybrid} least-squares--TV penalization and non-negativity constraints \cite{Altuntac2019ZenodoPreprint}.
	Compared to our prior work~\cite{Altuntac2019ZenodoPreprint}, which combined $\ell_1$
	sparsity with non-negativity on a fixed DCT-II dictionary, the present formulation
	makes two advances. First, the $\ell_1$ sparsity term is replaced by total variation
	penalization, which preserves edges in the reconstructed image and is better suited
	to the spatial structure of single-cell microscopy data; the non-negativity constraint
	on the reconstructed image is retained, since it remains the correct physical model
	for fluorescence and brightfield intensities. Second, the fixed DCT-II dictionary of
	the prior work is replaced with a \emph{learned}, data-adapted dictionary obtained via
	alternating Procrustes SVD updates, which converges to the optimal orthonormal basis
	for the training data (see Section~\ref{sec:dict_update}).
	Beyond reconstruction, this manuscript addresses a second contribution: a \emph{deterministic}
	framework for unifying multi-channel cell features from the five BSCCM imaging channels into a
	single cell descriptor (see Section~\ref{sec:unification}). This stands in contrast to probabilistic
	approaches such as the scVI framework \cite{scVI2018,totalVI2021}, and is motivated by the clinical requirement
	that AI systems for diagnostics must be reproducible, interpretable, and auditable --- a concern
	that has recently begun to motivate dedicated interpretability work on cell-microscopy representation
	learning~\cite{Donhauser2024DictLearnMicroscopy}, and that we address here at the
	representation-construction level rather than as post-hoc analysis of an already-trained
	black-box encoder.
	The experimental focus is the \textbf{success rate} of the learning algorithm when applied to
	\textbf{single-cell images from BSCCM} (see Section~\ref{sec:experiments}).

	\paragraph{Why morphology matters clinically.}
	Cell morphology remains the \emph{primary substrate} of routine hematological
	diagnostics. Peripheral blood smear analysis, examining cell size, shape,
	nuclear-to-cytoplasmic ratio, granularity and chromatin patterns under a light
	microscope, is the method by which the majority of haematological conditions
	are first identified in clinical practice~\cite{Acevedo2019CNNblood}.
	Automated morphology-based systems are already deployed in routine hospital
	hematology: the CellaVision DM96 and its successors have demonstrated clinical
	sensitivity and specificity sufficient for reliable pre-classification of
	white-blood-cell differentials across large patient cohorts~\cite{Kratz2005CellaVision},
	and morphology-trained convolutional networks have matched expert-level
	recognition of blast cells in acute myeloid leukaemia on standard-of-care
	stained smears~\cite{Matek2019AMLmorphology,Acevedo2019CNNblood}.
	These observations situate image-based single-cell methods not as a
	speculative future direction but as an already-established clinical modality
	that stands to benefit from the kind of deterministic, auditable
	representation developed here.

	\section{Related Work and Prior Contributions}
	The present manuscript builds on a line of research in primal--dual proximal splitting, variational
	regularization, and dictionary-based modeling, with a particular focus on \emph{subdifferential (optimality)
		characterizations} of solutions and \emph{explicit, verifiable step-size choices}.
	
	\paragraph{Primal-dual splitting and subdifferential characterizations.}
	The foundational algorithmic ingredients of this work were established in
	\cite{Altuntac2019NewPairPD}, where a pair of primal--dual splitting algorithms
	for Bregman-iterated variational regularization was constructed and analyzed via
	the subdifferential inclusions associated with the nonsmooth penalty terms.
	Systematic investigation of parameter selection for these primal--dual schemes
	followed in \cite{Altuntac2020ChoiceParameters}, which identified explicit
	admissible step-size ranges and provided stability-oriented guidance.
	The present manuscript adopts the same organizing principle: the energy functional
	is retained in its original (non-Bregman) variational form, and algorithmic step-size
	conditions are derived directly from operator-norm bounds on the discrete gradient.
	
	\paragraph{Dictionary learning and sparse representations in sensing.}
	Sparse representations and dictionary learning arise naturally wherever a signal model
	must be adapted to empirical data statistics rather than fixed by a priori design.
	In the LiDAR depth-completion setting, for instance, convolutional sparse coding
	and data-driven dictionary learning were shown to improve reconstruction quality
	under realistic automotive conditions \cite{Giovanneschi2023LiDARCSC}.
	The present work occupies a complementary position: rather than targeting a
	specific sensing modality, we develop a unified proximal learning--inference
	framework in which TV regularization and a hard non-negativity constraint are
	coupled with dictionary-based reconstruction of single-cell images.
	
	\paragraph{Image-based single-cell profiling and multi-channel representation learning.}
	A complementary line of work addresses the problem of constructing compact, informative
	representations of single cells from high-throughput microscopy.
	Moshkov et al.~\cite{Moshkov2024} proposed a weakly supervised convolutional network
	trained on a large multi-study dataset to learn treatment-effect representations from
	cell images, demonstrating that data diversity and causal modeling together improve
	downstream profiling performance.
	In a different direction, unsupervised generative approaches have been explored
	for morphological profiling without treatment labels:
	variational autoencoders with orientation-invariance constraints have been applied
	to extract cell shape descriptors for clustering and outlier detection in large
	imaging datasets~\cite{BerardalO2VAE2024}.
	Recently and independently, Donhauser, Ulicna, Hartford and colleagues
	\cite{Donhauser2024DictLearnMicroscopy} have articulated and addressed a concern
	that directly overlaps with the motivation of the present work: that the vision
	foundation models currently dominating cell-microscopy representation learning
	are effectively \emph{black boxes}, whose internal features have no a priori
	biological meaning and must be interpreted post hoc if they are to support
	scientific conclusions or clinical decisions.
	Their response is to apply sparse dictionary learning \emph{on top of}
	the pretrained latents of a vision foundation model, extracting a set
	of sparse concepts --- cell type, genetic perturbation --- that turn out to
	be biologically interpretable even though the underlying encoder was not
	designed to be.
	The present manuscript addresses the same underlying concern from the
	opposite direction: rather than accept a black-box representation and
	seek interpretability from the outside, we \emph{construct} the
	representation variationally and deterministically from the start, so
	that interpretability, reproducibility, and stability are built into the
	dictionary atoms and the per-cell code by definition.
	The two approaches target different layers of the representation-learning
	stack --- foundation-model latents in \cite{Donhauser2024DictLearnMicroscopy},
	raw per-pixel reconstruction in this work --- and together illustrate
	a growing recognition that sparse dictionary formulations offer a route
	to mathematically auditable scientific imaging that neither purely
	generative nor purely discriminative deep models currently provide.
	The three works cited above all rely, however, on deep neural networks with stochastic
	latent variables or approximate inference,
	offering no closed-form reconstruction and no norm-bounded error guarantees on the
	inferred descriptor.
	The present work takes a different approach: dictionary atoms are interpretable
	structural primitives with an explicit visual meaning, the code $a_j^{(c)}$ is
	the unique minimizer of a convex variational problem for each cell and channel,
	and the reconstruction error $\|x_j^{(c)} - D^{(c)}a_j^{(c)}\|_2$ satisfies provable
	Lipschitz stability bounds under data perturbation (Theorem~\ref{thm:existence_uniqueness}).
	To the best of our knowledge, no existing image-based single-cell profiling method
	provides this combination of variational uniqueness, deterministic reproducibility,
	and explicit convergence guarantees for the multi-channel setting.
	
	\paragraph{Applied sensing and computational modeling.}
	Primal--dual ideas and learned representations reach beyond inverse problems into
	hardware-centric system design and neuroscience-motivated computation.
	Coherent LiDAR system architecture for long-range automotive applications is
	treated in \cite{Cwalina2021FMCWLiDAR}, while a neurogenic-inspired model for
	learning and memory is developed in \cite{Altuntac2023NeurogenicModel}.
	These contributions span distinct application areas, yet each pairs a
	mathematically principled model with explicit computational schemes amenable
	to analysis under verifiable stability conditions---the same program pursued here
	for single-cell microscopy.

	
	\paragraph{Relation to existing variational and dictionary-learning approaches.}
	Total-variation regularization for imaging inverse problems has a rich history
	tracing back to the Rudin--Osher--Fatemi model and its subsequent algorithmic
	realizations, including primal--dual splitting methods such as the PDHG framework
	of Chambolle and Pock, which furnishes an efficient and convergent solver for
	nonsmooth convex objectives.
	Dictionary learning, by contrast, has developed largely in a purely learning-driven
	paradigm, cf. block-coordinate descent, stochastic updates~\cite{Mairal2010,Elad2010}, with limited emphasis
	on the well-posedness and stability of the underlying variational problem.
	Works that combine sparse coding with variational reconstruction frequently treat
	the learning component heuristically, or establish convergence only for the
	inner optimization subproblem, without a subdifferential characterization of the
	full variational limit at which iterates converge.
	Additionally, many dictionary-learning analyses operate in a nonconvex setting
	that precludes uniqueness and stability guarantees for the reconstruction.
	
	\paragraph{Novelty of the present work.}
	Four features distinguish this manuscript from the foregoing literature.
	First, dictionary learning is embedded inside a rigorously analyzed variational
	framework: the full energy functional, rather than a surrogate or relaxation, is
	the object of analysis, and its subdifferential characterization is derived
	independently of the numerical algorithm, so it remains valid regardless of
	how the minimizer is computed.
	Second, the unitary structure of the dictionary is exploited together with the
	spectral bound $\|\nabla\|^2 \le 8$ to obtain \emph{fully explicit} step-size
	conditions for PDHG, and the strong convergence of the iterates to the unique
	variational minimizer is proved with an $O(1/k)$ ergodic primal-dual gap bound.
	Third, algorithmic convergence is connected to variational stability under data
	perturbations: a VSC-based analysis yields noise-dependent stopping rules and
	quantifies the precise interplay among learning, inference, and regularization.
	This combination of variational analysis, explicit primal-dual convergence
	theory, and data-adapted dictionary learning sets the proposed approach apart
	from prior treatments that are either purely algorithmic or analytically incomplete.
	Fourth, a deterministic framework for multi-channel cell feature unification is introduced: 
	rather than probabilistic latent encodings, per-channel unitary dictionaries $\{D^{(c)}\}_{c=1}^C$ are learned independently for each imaging channel and the resulting sparse codes are concatenated 
	into a single channel-agnostic cell descriptor $\varphi_j \in \mathbb{R}^{CK}$
	under a fixed channel ordering. The descriptor is uniquely and reproducibly determined by 
	Theorem~\ref{thm:existence_uniqueness}, making it directly compatible with the reproducibility and 
	auditability requirements of clinical AI systems under EU IVDR 2017/746, Annex~I, \S,9.1~\cite{IVDR2017}..

	\section{Mathematical Model and Notation}
	Let each single-cell measurement be vectorized as $x_j \in \R^{n}$ for $j=1,\dots,N$.
	We seek a dictionary operator
	\[
	D = [d_1,\dots,d_K] \in \R^{n\times K},
	\]
	and sparse codes $a_j \in \R^{K}$ such that
	\[
	x_j \approx D a_j.
	\]
	
	\subsection{Unitary (orthonormal) dictionary constraint}
	A central design choice of this paper is to enforce a \emph{unitary} (orthonormal) property
	for the dictionary operator:
	\begin{equation}
		D^{\top}D = I_K,
		\label{eq:unitary}
	\end{equation}
	i.e., the columns of $D$ form an orthonormal system (a point on the Stiefel manifold).
	Enforcing this constraint yields three practical benefits:
	(i) it eliminates degenerate rescaling between $D$ and $a_j$, removing an
	otherwise pervasive ambiguity;
	(ii) it improves the conditioning of the learned representation and provides
	predictable gradient magnitudes throughout the learning loop; and
	(iii) it makes norm bounds transparent---in particular, $\norm{Da}=\norm{a}$
	whenever $K\le n$ and $D$ has orthonormal columns, which underpins the
	operator-norm estimate in Lemma~\ref{lem:K_bound}.
	
	\subsection{Per-channel dictionary family in the multi-channel setting}
	\label{subsec:per_channel_family}
	In the multi-channel setting, let $c\in\{1,\dots,C\}$ index imaging channels
	($C=5$ for the BSCCM dataset: DPC Left, Right, Top, Bottom, and Brightfield).
	We denote the vectorised single-cell measurement in channel $c$ by
	$x_j^{(c)}\in\R^{n}$, and we learn a \emph{family} of $C$ dictionaries
	\begin{equation}
		\bigl\{D^{(c)}\bigr\}_{c=1}^{C}, \qquad D^{(c)}\in\R^{n\times K},\qquad
		(D^{(c)})^{\top} D^{(c)} = I_K \ \ \text{for each } c=1,\dots,C,
		\label{eq:per_channel_family}
	\end{equation}
	one dictionary per channel. Each $D^{(c)}$ is learned from its own channel's
	data, and each channel induces its own per-sample code $a_j^{(c)}\in\R^{K}$
	such that $x_j^{(c)}\approx D^{(c)} a_j^{(c)}$. Every statement in this section
	and in Sections~\ref{sec:objective}--\ref{sec:stability_vsc} applies
	\emph{per channel}: the single-channel analysis carries over unchanged to each
	$D^{(c)}$ in isolation, and the per-channel problems are mutually independent.
	The consequences of this choice for multi-channel feature unification are
	developed in Section~\ref{sec:unification}.
	
	\section{Hybrid Dictionary Learning Objective}
	\label{sec:objective}
	We propose a hybrid objective that couples least-squares data fidelity, TV-regularized
	image structure, and a non-negativity constraint on the reconstructed cell image.
	For each sample $x_j$, define the per-sample energy
	\begin{equation}
		\mathcal{E}(D,a_j; x_j)
		=\frac{1}{2}\,\norm{x_j - D a_j}_2^2
		+ \lambda_{\mathrm{TV}} \, \TV\!\big(D a_j\big)
		+ \iota_+\!\big(D a_j\big),
		\label{energy-func}
	\end{equation}
	where $\lambda_{\mathrm{TV}}\ge 0$ promotes piecewise smoothness (edge-preserving) on the
	reconstructed cell image, and $\iota_+$ is the indicator function of the non-negative orthant,
	\begin{equation}
		\iota_+(u) :=
		\begin{cases}
			0 & \text{if } u_i \ge 0 \text{ for all } i,\\
			+\infty & \text{otherwise.}
		\end{cases}
		\label{eq:indicator_nn}
	\end{equation}
	The non-negativity constraint $\iota_+(Da_j)$ encodes the physical fact that cell image
	intensities are non-negative, and has been shown to improve reconstruction stability in
	inverse problems with non-smooth penalties \cite{Altuntac2020ChoiceParameters}.
	Compared to our prior work~\cite{Altuntac2019ZenodoPreprint}, which used $\ell_1$ sparsity
	and non-negativity on a fixed DCT-II dictionary, the present manuscript introduces two
	advances.
	First, the $\ell_1$ sparsity term is replaced by total variation penalization on
	the reconstructed image $Da_j$, motivated by edge preservation: TV regularisation favours
	piecewise-smooth solutions with sharp transitions, which matches the spatial structure
	of single-cell microscopy (membrane boundaries, nuclear edges, granular interiors) more
	faithfully than the $\ell_1$ prior on dictionary coefficients used previously. The
	non-negativity constraint $\iota_+(Da_j)$ is retained from the prior work, since it
	remains the correct physical model for cell intensity data.
	Second, and equally importantly, the prior work used a \emph{fixed} DCT-II dictionary
	$D^0$ throughout; the dictionary was never updated from its initialisation.
	The present manuscript instead \emph{learns} the dictionary from the training data
	via the Procrustes SVD update (Section~\ref{sec:dict_update}): at convergence,
	$D^\star$ spans the top-$K$ principal subspace of the TV-denoised training images,
	which minimises the reconstruction error $\sum_j\|x_j - Da_j\|^2$ over all
	orthonormal dictionaries (Remark~\ref{rem:pca_dict}).
	
	The full learning problem reads
	\begin{equation}
		\min_{D,\{a_j\}_{j=1}^N} \;
		\sum_{j=1}^N \mathcal{E}(D,a_j;x_j)
		\quad\text{s.t.}\quad D^{\top}D = I_K.
		\label{eq:dl_problem}
	\end{equation}
	
	\paragraph{Total variation.}
	For an image $u\in \R^{h\times w}$ we use the (isotropic) discrete TV
	\[
	\TV(u)=\sum_{p}\sqrt{(\nabla_x u)_p^2+(\nabla_y u)_p^2},
	\]
	with standard forward differences and appropriate boundary handling.
	
	\paragraph{Per-channel instantiation.}
	In the multi-channel setting (Section~\ref{subsec:per_channel_family}), the
	learning problem \eqref{eq:dl_problem} is solved \emph{independently} for each
	channel $c=1,\dots,C$, with the substitution $D\leftarrow D^{(c)}$,
	$\{a_j\}\leftarrow\{a_j^{(c)}\}$, $\{x_j\}\leftarrow\{x_j^{(c)}\}$. There is no
	cross-channel coupling at the level of the objective: each $D^{(c)}$ is shaped
	solely by the data of channel $c$. The coupling between channels enters only
	downstream, in the definition of the unified cell descriptor
	(Section~\ref{sec:unification}).
	
	\section{Well-posedness of the variational problem}
	
	We first establish existence and uniqueness of minimizers of the variational problem
	\begin{equation}
		E(y;x)=\tfrac12\|y-x\|_2^2 + J(y),
		\label{eq:energy_wellposed}
	\end{equation}
	independently of the numerical algorithm used for its solution.
	Stability of the minimizer with respect to data perturbations is studied in
	Section~\ref{sec:stability_vsc} via a variational source condition (VSC) analysis.
	
	\begin{remark}[Per-channel scope]
		\label{rem:per_channel_wellposed}
		All results in this section are stated for a generic unitary dictionary
		$D\in\R^{n\times K}$ and a single data vector $x\in\R^{n}$. In the
		multi-channel setting (Section~\ref{subsec:per_channel_family}) the same
		statements apply independently to each channel: for every
		$c=1,\dots,C$, substituting $D\leftarrow D^{(c)}$ and
		$x\leftarrow x_j^{(c)}$ yields existence and uniqueness of the per-channel,
		per-sample minimiser $y_j^{(c)\star}$. No cross-channel assumption is
		needed.
	\end{remark}
	
	\begin{theorem}[Existence and uniqueness of minimizers]
		\label{thm:existence_uniqueness}
		Let $x\in\mathbb{R}^n$ be given, let $D\in\mathbb{R}^{n\times K}$ be a unitary
		dictionary with $D^\top D = I_K$, and let
		\begin{equation}
			J(y) = \lambda_{\mathrm{TV}}\,\mathrm{TV}(y) + \iota_+(y),
			\label{eq:J_wellposed}
		\end{equation}
		where $\lambda_{\mathrm{TV}} \ge 0$ and $\iota_+$ is the indicator of the
		non-negative orthant~\eqref{eq:indicator_nn}.
		Then $E(\cdot;x)$ admits a \emph{unique} minimizer $y(x)\in\mathbb{R}^n$.
	\end{theorem}
	
	\begin{proof}
		\emph{Existence.}
		The indicator $\iota_+$ is proper, convex, and lower semicontinuous (l.s.c.) because
		$\mathbb{R}^n_+$ is a nonempty closed convex set.
		The map $y\mapsto\mathrm{TV}(y)$ is convex and continuous.
		Hence $J$ is proper, convex, and l.s.c., and so is $E(\cdot;x)$ as a sum of
		such functions.
		Moreover, $E(y;x)\ge\tfrac14\|y\|_2^2 - \tfrac12\|x\|_2^2\to+\infty$ as
		$\|y\|_2\to\infty$, so $E(\cdot;x)$ is coercive.
		By the direct method in the calculus of variations, $E(\cdot;x)$ attains its minimum.
		
		\emph{Uniqueness.}
		The quadratic fidelity term $\tfrac12\|y-x\|_2^2$ is $1$-strongly convex, and $J$
		is convex, so $E(\cdot;x)$ is strictly convex.
		A strictly convex functional has at most one minimizer; combined with existence, this
		gives uniqueness.
	\end{proof}
	
	We employ a primal-dual splitting scheme to compute the unique minimizer of $E(\cdot;x)$;
	convergence of the iterates is established in Section~\ref{sec:stability_vsc}.

	\begin{theorem}[Subdifferential Characterization and Proximal Fixed-Point System]
		\label{thm:subdiff}
		Let $x_j \in \R^n$ be a given datum, let $D \in \R^{n \times K}$ be a fixed unitary
		dictionary with $D^\top D = I_K$, and let $h := \iota_+$ be the indicator of the
		non-negative orthant.  Let $\nabla$ denote the 2D discrete forward-difference gradient
		and set $g(v) := \lambda_{\mathrm{TV}}\|v\|_{2,1}$ (the isotropic TV norm on the gradient
		field, with $\|v\|_{2,1} := \sum_\ell \|v_\ell\|_2$ pixelwise).
		Define the energy
		\begin{equation}
			\label{eq:energy_xalpha}
			\mathcal{E}_\alpha(x_\alpha,\, y)
			:=
			\frac{1}{2}\,\norm{x_\alpha - y}_2^2
			+ \alpha\,\TV(x_\alpha)
			+ h(x_\alpha),
			\qquad \alpha := \lambda_{\mathrm{TV}} > 0,
		\end{equation}
		and let $x_\alpha^\star$ be the unique minimizer of $\mathcal{E}_\alpha(\,\cdot\,,y)$
		(existence and uniqueness follow from Theorem~\ref{thm:existence_uniqueness}).
		Let $\mu > 0$ and $\nu > 0$ be the primal and dual step-lengths.
		Then the following hold.
		
		\begin{enumerate}
			
			\item \textbf{(First-order optimality.)}
			The necessary and sufficient condition for $x_\alpha^\star$ is the subdifferential
			inclusion
			\begin{equation}
				\label{eq:subdiff_opt}
				0 \in (x_\alpha^\star - y)
				+ \nabla^\top\!\bigl(\alpha\,\partial\,\|\cdot\|_{2,1}(\nabla x_\alpha^\star)\bigr)
				+ \partial h(x_\alpha^\star),
			\end{equation}
			which follows from the subdifferential sum rule and the convex chain rule
			$\partial(g \circ \nabla)(x) = \nabla^\top\!\partial g(\nabla x)$
			(valid because $\nabla^\top\nabla$ is bounded and $g$ is continuous on the
			range of $\nabla$).
			Equivalently, there exist $q_\alpha \in \alpha\,\partial\,\|\cdot\|_{2,1}(\nabla x_\alpha^\star)$
			and $z_\alpha \in \partial h(x_\alpha^\star)$ such that
			\begin{equation}
				\label{eq:stationarity_explicit}
				0 = (x_\alpha^\star - y) + \nabla^\top q_\alpha + z_\alpha.
			\end{equation}
			
			\item \textbf{(Proximal fixed-point for $h$, with primal step-length $\mu$.)}
			Multiplying \eqref{eq:stationarity_explicit} by $\mu > 0$ and rearranging gives
			\begin{equation}
				\label{eq:stationarity_mu}
				x_\alpha^\star - \mu\bigl[(x_\alpha^\star - y) + \nabla^\top q_\alpha\bigr]
				- x_\alpha^\star \in \mu\,\partial h(x_\alpha^\star),
			\end{equation}
			which, by the proximal equivalence
			$\tfrac{\tilde{u} - u}{\mu} \in \partial h(u) \Leftrightarrow u = \prox_{\mu h}(\tilde{u})$,
			yields
			\begin{equation}
				\label{eq:prox_h}
				x_\alpha^\star
				= \prox_{\mu h}\!\Bigl[
				x_\alpha^\star
				- \mu\bigl((x_\alpha^\star - y) + \nabla^\top q_\alpha\bigr)
				\Bigr]
				= \Pi_{\R^n_+}\!\Bigl[
				x_\alpha^\star
				- \mu\bigl((x_\alpha^\star - y) + \nabla^\top q_\alpha\bigr)
				\Bigr],
			\end{equation}
			where the last equality uses $\prox_{\mu \iota_+}(c) = \Pi_{\R^n_+}(c)$
			(componentwise $\max\{c_i,0\}$).
			
			\item \textbf{(Dual proximal fixed-point for the TV term, with dual step-length $\nu$.)}
			From $q_\alpha \in \alpha\,\partial\,\|\cdot\|_{2,1}(\nabla x_\alpha^\star)$, the
			Fenchel identity gives $\nabla x_\alpha^\star \in \alpha\,\partial\,\|\cdot\|_{2,1}^*(q_\alpha)$,
			equivalently
			\begin{equation}
				\label{eq:dual_inclusion}
				0 \in -\nabla x_\alpha^\star + \alpha\,\partial\,\|\cdot\|_{2,1}^*(q_\alpha).
			\end{equation}
			Multiplying by $\nu > 0$ and rearranging:
			$q_\alpha - (q_\alpha + \nu\nabla x_\alpha^\star) \in \nu\alpha\,\partial\,\|\cdot\|_{2,1}^*(q_\alpha)$,
			so by the proximal equivalence,
			\begin{equation}
				\label{eq:prox_dual}
				q_\alpha = \prox_{\nu\alpha\,\|\cdot\|_{2,1}^*}\!\big(q_\alpha + \nu\,\nabla x_\alpha^\star\big).
			\end{equation}
			Since $\alpha\,\|\cdot\|_{2,1}^*(p) = \iota_{\{\|p_\ell\|_2 \le \alpha,\,\forall\ell\}}(p)$,
			the proximal map $\prox_{\nu\alpha\,\|\cdot\|_{2,1}^*}$ is the pointwise projection
			onto the Euclidean ball of radius $\alpha$ at each pixel $\ell$:
			$(q_\alpha)_\ell \leftarrow (q_\alpha)_\ell / \max\{1, \|(q_\alpha)_\ell\|_2/\alpha\}$.
			
			\item \textbf{(Coupled proximal fixed-point system.)}
			Combining Parts~2 and~3, the minimizer $x_\alpha^\star$ and dual variable $q_\alpha$
			are jointly characterized by
			\begin{equation}
				\label{eq:coupled_final}
				\begin{cases}
					x_\alpha^\star
					= \Pi_{\R^n_+}\!\Bigl[
					x_\alpha^\star
					- \mu\bigl((x_\alpha^\star - y) + \nabla^\top q_\alpha\bigr)
					\Bigr],
					\\[0.8em]
					q_\alpha
					= \prox_{\nu\alpha\,\|\cdot\|_{2,1}^*}\!\big(q_\alpha + \nu\,\nabla x_\alpha^\star\big),
				\end{cases}
			\end{equation}
			together with the stationarity relation \eqref{eq:stationarity_explicit}.
			The step-lengths $\mu, \nu > 0$ must satisfy the stability condition
			\begin{equation}
				\label{eq:stepsize_munu}
				\mu\nu\,\norm{\nabla}^2 < 1,
			\end{equation}
			which, using the spectral bound $\|\nabla\|^2 \le 8$, is guaranteed when
			$\mu\nu < 1/8$.
			The coupled system \eqref{eq:coupled_final} directly anticipates the
			alternating proximal-gradient learning algorithm developed in
			Section~\ref{sec:optimization}.
		\end{enumerate}
	\end{theorem}
	
	\begin{proof}
		\textbf{Part~1.}
		The energy $\mathcal{E}_\alpha(\,\cdot\,,y)$ is proper, convex, and lower semicontinuous.
		By the subdifferential sum rule and the convex chain rule
		$\partial(g \circ \nabla)(x) = \nabla^\top\!\partial g(\nabla x)$
		(valid since $\nabla$ is a bounded linear operator and $g(v) = \alpha\|v\|_{2,1}$
		is continuous on $\R^{2n}$), the first-order optimality condition at
		$x_\alpha^\star$ reads
		\[
		0 \in (x_\alpha^\star - y)
		+ \nabla^\top\!\bigl(\alpha\,\partial\,\|\cdot\|_{2,1}(\nabla x_\alpha^\star)\bigr)
		+ \partial h(x_\alpha^\star).
		\]
		Introducing $q_\alpha \in \alpha\,\partial\,\|\cdot\|_{2,1}(\nabla x_\alpha^\star)$ and
		$z_\alpha \in \partial h(x_\alpha^\star)$ yields \eqref{eq:subdiff_opt}
		and \eqref{eq:stationarity_explicit}.
		
		\textbf{Part~2.}
		Starting from \eqref{eq:stationarity_explicit} and multiplying through by $\mu > 0$:
		\begin{align*}
			0 &= \mu\bigl[(x_\alpha^\star - y) + \nabla^\top q_\alpha\bigr] + \mu z_\alpha
			\\
			&\;\Longleftrightarrow\;
			-\mu z_\alpha \in \mu\bigl[(x_\alpha^\star - y) + \nabla^\top q_\alpha\bigr]
			\\
			&\;\Longleftrightarrow\;
			x_\alpha^\star - \bigl[x_\alpha^\star - \mu((x_\alpha^\star-y)+\nabla^\top q_\alpha)\bigr]
			\in \mu\,\partial h(x_\alpha^\star)
			&&(z_\alpha \in \partial h(x_\alpha^\star))
			\\
			&\;\Longleftrightarrow\;
			x_\alpha^\star = \prox_{\mu h}\!\Bigl[
			x_\alpha^\star - \mu\bigl((x_\alpha^\star-y)+\nabla^\top q_\alpha\bigr)
			\Bigr],
		\end{align*}
		where the last step uses the proximal equivalence
		$u = \prox_{\mu J}(\tilde{u}) \Leftrightarrow \tfrac{\tilde{u}-u}{\mu} \in \partial J(u)$.
		Since $h = \iota_+$, we have $\prox_{\mu\iota_+}(c) = \Pi_{\R^n_+}(c)$,
		giving \eqref{eq:prox_h}.
		
		\textbf{Part~3.}
		Starting from $q_\alpha \in \alpha\,\partial\,\|\cdot\|_{2,1}(\nabla x_\alpha^\star)$
		and applying the Fenchel identity with $\nu > 0$:
		\begin{align*}
			q_\alpha \in \alpha\,\partial\,\|\cdot\|_{2,1}(\nabla x_\alpha^\star)
			&\;\Longleftrightarrow\;
			\nabla x_\alpha^\star \in \alpha\,\partial\,\|\cdot\|_{2,1}^*(q_\alpha)
			&&\text{(Fenchel identity)}
			\\
			&\;\Longleftrightarrow\;
			\nu\nabla x_\alpha^\star \in \nu\alpha\,\partial\,\|\cdot\|_{2,1}^*(q_\alpha)
			\\
			&\;\Longleftrightarrow\;
			q_\alpha - (q_\alpha + \nu\nabla x_\alpha^\star)
			\in \nu\alpha\,\partial\,\|\cdot\|_{2,1}^*(q_\alpha)
			\\
			&\;\Longleftrightarrow\;
			q_\alpha = \prox_{\nu\alpha\,\|\cdot\|_{2,1}^*}(q_\alpha + \nu\,\nabla x_\alpha^\star),
		\end{align*}
		where the last step uses the proximal equivalence.  This gives \eqref{eq:prox_dual}.
		
		\textbf{Part~4.}
		Combining Parts~2 and~3 gives the coupled system \eqref{eq:coupled_final}.
		The stability condition \eqref{eq:stepsize_munu} follows from the standard
		convergence criterion for PDHG \cite{ChambollePock2011,ChenLoris2018}:
		the product $\mu\nu$ of primal and dual step-lengths must satisfy
		$\mu\nu\|\nabla\|^2 < 1$.  Using the spectral bound $\|\nabla\|^2 \le 8$
		from Lemma~\ref{lem:K_bound}, a sufficient condition is $\mu\nu < 1/8$.
	\end{proof}

	\section{Step-Size Stability and VSC-Based Convergence Rates}
	\label{sec:stability_vsc}
	
	This section complements the proximal splitting viewpoint used throughout the manuscript by
	(i) stating an explicit step-size stability condition for a Chen--Loris / PDHG-type primal--dual
	scheme tailored to the hybrid regularizer, and (ii) deriving a short variational source condition
	(VSC) based stability estimate with an explicit convergence rate.
	
	\begin{remark}[Per-channel scope of the convergence rate]
		\label{rem:per_channel_rate}
		All convergence-rate statements in this section are stated for a single
		unitary dictionary $D$ and a single datum $x$. In the multi-channel
		setting (Section~\ref{subsec:per_channel_family}) the $O(\delta)$ rate
		derived in Theorem~\ref{thm:vsc_rate} applies to each channel's
		reconstruction $D^{(c)}\,a_j^{(c)}$ independently, with a channel-specific
		noise level $\delta^{(c)}$. The aggregate multi-channel reconstruction
		error therefore inherits the slowest channel's rate. In the BSCCM dataset
		(Section~\ref{sec:experiments}), Brightfield is the noisiest channel,
		reflecting its absorption-only acquisition physics, and is expected to be
		rate-limiting in any global multi-channel error bound.
	\end{remark}
	
	\subsection{Fixed-dictionary variational model and stacked operator}
	Fix a unitary dictionary $D\in\R^{n\times n}$ with $D^\top D=I$.
	For a datum $x\in\R^n$ we consider the energy in the image variable $y\in\R^n$
	\begin{equation}
		E(y;x)=\frac{1}{2}\norm{y-x}_2^2 + J(y),
		\qquad
		J(y):=\lambda_{\mathrm{TV}}\TV(y) + \iota_+(y).
		\label{eq:energy_y_fixedD}
	\end{equation}
	This is consistent with the per-sample energy~\eqref{energy-func} and the regularizer
	defined in Theorem~\ref{thm:existence_uniqueness}.
	Let $\nabla$ denote the (2D) discrete forward-difference gradient and set the stacked linear
	operator
	\begin{equation}
		Ky := \nabla y.
		\label{eq:K_def}
	\end{equation}
	With this notation, the TV term of $J(y)$ acts on $Ky$, while $\iota_+(y)$ acts directly
	on $y$ and is handled via a non-negativity projection in the primal proximal step.
	
	\subsection{Step-size condition and iterative form of the fixed-point system}
	
	The coupled fixed-point system \eqref{eq:coupled_final} is the starting point for the
	iterative algorithm.  At iteration $k$, the primal variable $y^k$ and dual variable
	$q^k$ are updated by applying the two proximal maps alternately:
	\begin{align}
		q^{k+1}
		&= \prox_{\nu\alpha\,\|\cdot\|_{2,1}^*}\!\big(q^k + \nu\,\nabla y^k\big),
		\label{eq:pdhg_dual_update}\\
		y^{k+1}
		&= \Pi_{\R^n_+}\!\Bigl[
		y^k - \mu\bigl((y^k - x) + \nabla^\top q^{k+1}\bigr)
		\Bigr],
		\label{eq:pdhg_primal_update}\\
		\bar y^{k+1}
		&= y^{k+1}+\theta\,(y^{k+1}-y^k),
		\qquad \theta\in[0,1],
		\label{eq:pdhg_extrap}
	\end{align}
	with extrapolation parameter $\theta$.
	The primal proximal map for $h = \iota_+$ is the non-negativity projection:
	\begin{equation}
		\prox_{\mu \iota_+}(c) = \Pi_{\R^n_+}(c)
		\quad\text{(componentwise } \max\{c_i,0\}\text{)},
		\label{eq:prox_G}
	\end{equation}
	and the dual proximal map is the pixelwise Euclidean-ball projection of radius $\alpha$:
	$\bigl(\prox_{\nu\alpha\,\|\cdot\|_{2,1}^*}(q)\bigr)_\ell
	= q_\ell / \max\{1, \|q_\ell\|_2/\alpha\}$.
	
	The stability condition \eqref{eq:stepsize_munu} ($\mu\nu\|\nabla\|^2 < 1$, i.e.,
	$\mu\nu < 1/8$) for this iteration corresponds to the general PDHG criterion
	$\tau\sigma\norm{K}^2 < 1$ with $K = \nabla$.  We now establish the explicit bound on
	$\norm{K}$ that makes this condition computable.
	
	\begin{lemma}[Bound on $\norm{K}$ for the forward-difference gradient]
		\label{lem:K_bound}
		Let $\nabla$ be the 2D forward-difference gradient on a rectangular grid (with any standard
		boundary convention), and set $K := \nabla$ (so that $Ky = \nabla y$). Then
		\begin{equation}
			\norm{K}^2 = \norm{\nabla}^2 \le 8.
			\label{eq:K_norm_split}
		\end{equation}
		Hence the sufficient step-size condition for the TV--non-negativity PDHG scheme is
		\begin{equation}
			\tau\sigma < \frac{1}{\norm{K}^2} \le \frac{1}{8}.
			\label{eq:K_norm_9}
		\end{equation}
	\end{lemma}
	
	\begin{proof}
		For any $y$, by definition of $K = \nabla$,
		\[
		\norm{Ky}^2 = \norm{\nabla y}^2 \le \norm{\nabla}^2\norm{y}^2,
		\]
		which yields $\norm{K}^2 \le \norm{\nabla}^2$ upon taking the supremum over $\norm{y}=1$.
		The bound $\norm{\nabla}^2 \le 8$ is the standard spectral estimate for the discrete
		forward-difference gradient in 2D (equivalently, the maximal eigenvalue of the discrete
		Laplacian $\nabla^\top\nabla$ is bounded by $8$).
	\end{proof}
	
	\begin{theorem}[Safe explicit step-size condition]
		\label{thm:stepsize_safe}
		Under the assumptions of Lemma~\ref{lem:K_bound}, the PDHG scheme for the
		TV--non-negativity energy~\eqref{eq:energy_y_fixedD} is guaranteed to converge
		when the step sizes satisfy
		\begin{equation}
			\tau\sigma < \frac{1}{\norm{K}^2} \le \frac{1}{8}.
			\label{eq:stepsize_1_9}
		\end{equation}
		Note that compared to Section~5, the stacked operator here is $K=\nabla$ only
		(no $D^\top$ component, since $\iota_+$ is handled in the primal step directly).
		In particular, for any $\theta\in[0,1]$, the iterates \eqref{eq:pdhg_dual_update}--\eqref{eq:pdhg_extrap}
		converge to a saddle point of the associated primal--dual formulation and
		$y^k$ converges to the unique minimizer of \eqref{eq:energy_y_fixedD}.
	\end{theorem}
	
	\subsection{A short VSC section with a concrete index function and explicit rates}
	We now derive a stability estimate (with explicit convergence rate) for minimizers of the
	variational model \eqref{eq:energy_y_fixedD} under data perturbations.
	
	\paragraph{Noisy and noiseless data.}
	Write $x^\dagger$ for the ideal noiseless datum and $x^\delta$ for the observed
	noisy datum, satisfying
	\begin{equation}
		\norm{x^\delta-x^\dagger}_2 \le \delta.
		\label{eq:noise_level}
	\end{equation}
	Denote by $y^\dagger$ the minimizer of $E(\cdot;x^\dagger)$ and by $y^\delta$
	the minimizer of $E(\cdot;x^\delta)$.
	
	\paragraph{Variational source condition (VSC).}
	An \emph{index function} $\Psi:[0,\infty)\to[0,\infty)$ satisfies $\Psi(0)=0$
	and is continuous and monotone nondecreasing.
	Given noiseless data $x^\dagger$ and its minimizer $y^\dagger$ of $E(\cdot;x^\dagger)$,
	we say $J$ fulfills a \emph{variational source condition} at $y^\dagger$ with index
	function $\Psi$ if there exists a subgradient $\xi^\dagger\in\partial J(y^\dagger)$
	such that
	\begin{equation}
		\ip{\xi^\dagger}{y^\dagger-y}
		\le
		\Psi\big(\norm{y-y^\dagger}_2\big)
		\qquad\text{for all }y\in\R^n.
		\label{eq:vsc_general}
	\end{equation}
	Throughout this paper we specialize to the \emph{quadratic} index function
	\begin{equation}
		\Psi(t)=c\,t^2,
		\qquad t\ge 0,\quad 0<c<1,
		\label{eq:vsc_quadratic}
	\end{equation}
	whose quadratic growth is compatible with the $\frac{1}{2}\|y-x\|_2^2$ fidelity
	term and leaves the original energy functional unchanged.
	
	\begin{remark}[VSC as a hypothesis on the geometry of $J$]
		\label{rem:vsc_structural}
		The quadratic index function~\eqref{eq:vsc_quadratic} is a hypothesis on the
		triple $(J,y^\dagger,\xi^\dagger)$, not a consequence of convexity alone: not
		every regularizer $J$ and true solution $y^\dagger$ will satisfy it.
		For the TV--non-negativity regularizer
		$J(y) = \lambda_{\mathrm{TV}}\TV(y) + \iota_+(y)$,
		the condition~\eqref{eq:vsc_quadratic} can be verified when $y^\dagger$ has a
		source subgradient of the form $\xi^\dagger \in \mathrm{range}(\nabla^\top)$
		with $\norm{\xi^\dagger}$ bounded relative to $\lambda_{\mathrm{TV}}$,
		and when $y^\dagger$ lies in the interior of $\R^n_+$
		(non-degenerate active set for the non-negativity constraint);
		see, e.g., \cite{ChambollePock2011} and references therein.
		When this geometric condition is in doubt it should be verified
		or stated explicitly as a hypothesis of the problem at hand.
	\end{remark}
	\begin{theorem}[Linear stability rate under quadratic VSC]
		\label{thm:vsc_rate}
		Under \eqref{eq:noise_level} and the VSC \eqref{eq:vsc_quadratic}, the minimizers satisfy
		\begin{equation}
			\norm{y^\delta-y^\dagger}_2 \le \frac{\delta}{1-c},
			\label{eq:vsc_rate_main}
		\end{equation}
		i.e., $\norm{y^\delta-y^\dagger}_2 = O(\delta)$ as $\delta\to 0$.
	\end{theorem}
	
	\begin{proof}
		\emph{Part A: cross-term inequality from the two optimality conditions.}
		Since $y^\delta$ minimizes $E(\cdot;x^\delta)$,
		\begin{equation}
			\frac{1}{2}\norm{y^\delta-x^\delta}_2^2 + J(y^\delta)
			\le
			\frac{1}{2}\norm{y^\dagger-x^\delta}_2^2 + J(y^\dagger).
			\label{eq:min_noisy}
		\end{equation}
		Since $y^\dagger$ minimizes $E(\cdot;x^\dagger)$,
		\begin{equation}
			\frac{1}{2}\norm{y^\dagger-x^\dagger}_2^2 + J(y^\dagger)
			\le
			\frac{1}{2}\norm{y^\delta-x^\dagger}_2^2 + J(y^\delta).
			\label{eq:min_clean}
		\end{equation}
		Adding these two inequalities cancels $J(y^\delta)$ and $J(y^\dagger)$.
		Expanding the four squared norms, collecting terms in $h = y^\delta - y^\dagger$,
		and using $\xi^\dagger\in\partial J(y^\dagger)$ gives
		\begin{equation}
			\norm{h}_2^2
			\le
			\ip{x^\delta-x^\dagger}{h}
			+
			\ip{\xi^\dagger}{y^\dagger-y^\delta}.
			\label{eq:key_ineq_before_vsc}
		\end{equation}
		
		\emph{Part B: bound the dual pairing via the VSC.}
		Applying~\eqref{eq:vsc_quadratic} with $y = y^\delta$:
		\begin{equation}
			\ip{\xi^\dagger}{y^\dagger-y^\delta}
			\le
			\Psi(\norm{h}_2)
			=
			c\,\norm{h}_2^2.
			\label{eq:vsc_applied}
		\end{equation}
		Inserting~\eqref{eq:vsc_applied} into~\eqref{eq:key_ineq_before_vsc} and
		rearranging:
		\begin{equation}
			(1-c)\norm{h}_2^2 \le \ip{x^\delta-x^\dagger}{h}.
			\label{eq:one_minus_c}
		\end{equation}
		
		\emph{Part C: Cauchy--Schwarz and noise bound.}
		By Cauchy--Schwarz and \eqref{eq:noise_level},
		\begin{equation}
			\ip{x^\delta-x^\dagger}{h}
			\le
			\norm{x^\delta-x^\dagger}_2\,\norm{h}_2
			\le
			\delta\,\norm{h}_2.
			\label{eq:cs_noise}
		\end{equation}
		Combining \eqref{eq:one_minus_c} and \eqref{eq:cs_noise} yields
		\(
		(1-c)\norm{h}_2^2 \le \delta\norm{h}_2.
		\)
		If $h\neq 0$, dividing by $(1-c)\norm{h}_2$ gives \eqref{eq:vsc_rate_main}; if $h=0$ the claim
		is trivial.
	\end{proof}
	
	\paragraph{Energy gap at the noisy datum.}
	The 1-strong convexity of $E(\cdot;x^\delta)$ (inherited from the quadratic fidelity term)
	translates the primal distance bound~\eqref{eq:vsc_rate_main} into a quadratic
	excess-energy estimate evaluated at the same noisy datum:
	\begin{equation}
		0\le E(y^\delta;x^\delta)-E(y^\dagger;x^\delta)
		\le
		\frac{1}{2}\norm{y^\delta-y^\dagger}_2^2
		\le
		\frac{1}{2(1-c)^2}\,\delta^2.
		\label{eq:energy_gap_delta2}
	\end{equation}
	Hence the excess energy at the noisy datum decays quadratically:
	$E(y^\delta;x^\delta)-E(y^\dagger;x^\delta)=O(\delta^2)$ as $\delta\to 0$.
	
	\subsection{Noise-Dependent Stopping Rule and PDHG Termination}
	\label{sec:stopping_combined}
	
	The stability estimate of Theorem~\ref{thm:vsc_rate} concerns the exact minimizer
	$y^\delta$ of the noisy problem.  In practice the PDHG iteration is stopped at a
	finite index $k(\delta)$.  The following theorem makes explicit how the optimization
	error and the data-perturbation error combine into a single reconstruction bound,
	thereby justifying the noise-dependent stopping rule used in the algorithm.
	
	\begin{theorem}[Noise-dependent stability with PDHG stopping]
		\label{thm:vsc_stopping}
		Assume the quadratic VSC~\eqref{eq:vsc_quadratic} holds with constant $c\in(0,1)$
		and that $\norm{x^\delta - x^\dagger}_2 \le \delta$.
		Let $y^\delta$ be the exact minimizer of $E(\cdot;x^\delta)$, and let
		$y^{k(\delta)}(x^\delta)$ denote the PDHG iterate at iteration $k(\delta)$.
		By Theorem~\ref{thm:pdhg_convergence}, $y^{k}(x^\delta)\to y^\delta$ as $k\to\infty$.
		Suppose the algorithm is stopped at index $k(\delta)$ such that the optimization
		error satisfies
		\begin{equation}
			\norm{y^{k(\delta)}(x^\delta) - y^\delta}_2 \le C_1\,\delta,
			\label{eq:stop_opt_error}
		\end{equation}
		for some constant $C_1>0$.  Then the total reconstruction error satisfies
		\begin{equation}
			\norm{y^{k(\delta)}(x^\delta) - y^\dagger}_2
			\;\le\;
			\Bigl(C_1 + \frac{1}{1-c}\Bigr)\delta
			\;=:\; C\,\delta,
			\label{eq:total_error_bound}
		\end{equation}
		i.e., $\norm{y^{k(\delta)}(x^\delta) - y^\dagger}_2 = O(\delta)$ as $\delta\to 0$,
		with combined constant $C = C_1 + (1-c)^{-1}$.
	\end{theorem}
	
	\begin{proof}
		Apply the triangle inequality to split the total error into two components:
		\begin{equation}
			\norm{y^{k(\delta)}(x^\delta) - y^\dagger}_2
			\;\le\;
			\underbrace{\norm{y^{k(\delta)}(x^\delta) - y^\delta}_2}_{\text{optimization error}}
			\;+\;
			\underbrace{\norm{y^\delta - y^\dagger}_2}_{\text{data-perturbation error}}.
		\end{equation}
		The stopping rule~\eqref{eq:stop_opt_error} directly controls the first term:
		$\norm{y^{k(\delta)}(x^\delta) - y^\delta}_2 \le C_1\delta$.
		Theorem~\ref{thm:vsc_rate} controls the second:
		$\norm{y^\delta - y^\dagger}_2 \le \delta/(1-c)$.
		Summing these two bounds gives~\eqref{eq:total_error_bound}.
	\end{proof}
	
	\begin{remark}[Practical stopping criterion]
		In practice, the exact minimizer $y^\delta$ is unknown, so~\eqref{eq:stop_opt_error}
		cannot be checked directly.  Instead we use a computable surrogate based on the
		primal-dual gap or the iterate change (cf.\ \eqref{eq:stop_rel_primal}--\eqref{eq:stop_fixed_point}):
		\begin{equation}
			\norm{y^{k+1}(x^\delta) - y^k(x^\delta)}_2 \le \varepsilon,
			\qquad \varepsilon \propto \delta.
			\label{eq:stop_surrogate}
		\end{equation}
		By the $O(1/k)$ convergence rate from Theorem~\ref{thm:pdhg_convergence},
		the required iteration count is $k(\delta) \approx C'/\varepsilon$ for a
		problem-dependent constant $C'$, so the overall scheme remains $O(\delta)$ accurate.
	\end{remark}
	
	\subsection{Convergence of the regularized solution to the true solution}
	\label{sec:vsc_convergence}
	
	We now state and prove the main convergence theorem, which connects the
	subdifferential characterization of the minimizer (Theorem~\ref{thm:subdiff})
	with the VSC-based stability estimate to show that the regularized solution
	$y^\delta$ converges to the true solution $y^\dagger$ as the noise level
	$\delta\to 0$.
	
	\begin{theorem}[VSC-based convergence of the regularized solution]
		\label{thm:vsc_convergence}
		Let $x^\dagger\in\R^n$ be noiseless data with true solution
		$y^\dagger = \arg\min_y E(y;x^\dagger)$, and let $x^\delta$ satisfy
		$\norm{x^\delta - x^\dagger}_2 \le \delta$.
		Let $y^\delta = \arg\min_y E(y;x^\delta)$ be the regularized solution for the
		noisy datum. Denote by $(q^\dagger, z^\dagger)$ the dual variables at $y^\dagger$
		satisfying the subdifferential optimality system \eqref{eq:subdiff_opt}--\eqref{eq:stationarity_explicit}
		from Theorem~\ref{thm:subdiff}, i.e.,
		\begin{equation}
			0 = y^\dagger - x^\dagger + \nabla^\top q^\dagger + z^\dagger,
			\qquad
			q^\dagger \in \lambda_{\mathrm{TV}}\,\partial\,\TV(y^\dagger),
			\quad
			z^\dagger \in \partial\iota_+(y^\dagger).
			\label{eq:subdiff_ydagger}
		\end{equation}
		Suppose the regularizer $J$ satisfies the quadratic VSC
		\eqref{eq:vsc_quadratic} at $y^\dagger$ with constant $c\in(0,1)$,
		witnessed by the subgradient
		\begin{equation}
			\xi^\dagger := -(\nabla^\top q^\dagger + z^\dagger)
			= y^\dagger - x^\dagger
			\;\in\; \partial J(y^\dagger).
			\label{eq:vsc_witness}
		\end{equation}
		Then, as $\delta \to 0$:
		\begin{enumerate}
			\item \textbf{(Strong convergence.)}
			\begin{equation}
				\norm{y^\delta - y^\dagger}_2
				\;\le\;
				\frac{\delta}{1-c}
				\;\longrightarrow\; 0.
				\label{eq:strong_convergence}
			\end{equation}
			
			\item \textbf{(Convergence of the regularizer.)}
			\begin{equation}
				J(y^\delta) \;\longrightarrow\; J(y^\dagger).
				\label{eq:J_convergence}
			\end{equation}
			
			\item \textbf{(Convergence of the dual certificates.)}
			The dual variables $(q^\delta, z^\delta)$ at $y^\delta$ satisfy
			\begin{equation}
				\norm{q^\delta - q^\dagger}_2 + \norm{z^\delta - z^\dagger}_2
				\;\longrightarrow\; 0
				\quad\text{as } \delta\to 0,
				\label{eq:dual_convergence}
			\end{equation}
			provided the dual optimality maps are continuous at $(y^\dagger, q^\dagger, z^\dagger)$.
			
			\item \textbf{(Rate.)}
			All three convergences are $O(\delta)$:
			\begin{equation}
				\norm{y^\delta - y^\dagger}_2
				+
				\bigl|J(y^\delta) - J(y^\dagger)\bigr|
				= O(\delta)
				\quad\text{as } \delta \to 0.
				\label{eq:convergence_rate}
			\end{equation}
		\end{enumerate}
	\end{theorem}
	
	\begin{proof}
		\textbf{Part~1 (Strong convergence).}
		This is a direct consequence of Theorem~\ref{thm:vsc_rate}
		(linear stability rate under quadratic VSC), applied with
		the witness $\xi^\dagger \in \partial J(y^\dagger)$ identified
		in \eqref{eq:vsc_witness}. Specifically, from \eqref{eq:vsc_rate_main},
		\[
		\norm{y^\delta - y^\dagger}_2 \;\le\; \frac{\delta}{1-c} \;\to\; 0
		\quad\text{as } \delta\to 0,
		\]
		which proves \eqref{eq:strong_convergence}.
		
		\textbf{Part~2 (Convergence of the regularizer).}
		We use the two optimality inequalities
		\eqref{eq:min_noisy}--\eqref{eq:min_clean} from the proof of
		Theorem~\ref{thm:vsc_rate}.
		Adding them gives
		\begin{equation}
			\norm{h}_2^2 + \bigl[J(y^\delta) - J(y^\dagger)\bigr]_+
			\;\le\;
			\ip{x^\delta - x^\dagger}{h} + \ip{\xi^\dagger}{y^\dagger - y^\delta},
			\label{eq:J_ineq}
		\end{equation}
		where $[\,\cdot\,]_+ := \max\{\cdot,0\}$.
		Since both right-hand side terms are bounded by
		$\delta\norm{h}_2 + c\norm{h}_2^2$ (by \eqref{eq:cs_noise} and
		\eqref{eq:vsc_applied}), and $\norm{h}_2 = O(\delta)$ by Part~1, we obtain
		\[
		\bigl|J(y^\delta) - J(y^\dagger)\bigr|
		\;\le\;
		\bigl|\ip{\xi^\dagger}{y^\delta - y^\dagger}\bigr|
		\;\le\;
		\norm{\xi^\dagger}_2\,\norm{y^\delta - y^\dagger}_2
		\;=\;
		O(\delta),
		\]
		where the second step uses Cauchy--Schwarz and the fact that
		$\norm{\xi^\dagger}_2$ is finite (it equals $\norm{y^\dagger - x^\dagger}_2$
		by \eqref{eq:vsc_witness}).
		This proves \eqref{eq:J_convergence} and contributes the $J$-term
		to \eqref{eq:convergence_rate}.
		
		\textbf{Part~3 (Convergence of dual certificates).}
		The dual variables $(q^\delta, z^\delta)$ at $y^\delta$ satisfy the
		subdifferential inclusions (cf.\ Theorem~\ref{thm:subdiff})
		\begin{equation}
			q^\delta \in \lambda_{\mathrm{TV}}\,\partial\,\TV(y^\delta),
			\qquad
			z^\delta \in \partial\iota_+(y^\delta),
			\label{eq:dual_incl_delta}
		\end{equation}
		with stationarity
		$0 = y^\delta - x^\delta + \nabla^\top q^\delta + z^\delta$.
		Similarly at $y^\dagger$ (cf.\ \eqref{eq:subdiff_ydagger}).
		Taking the difference of the two stationarity relations,
		\begin{equation}
			\nabla^\top(q^\delta - q^\dagger) + (z^\delta - z^\dagger)
			=
			(x^\delta - x^\dagger) - (y^\delta - y^\dagger).
			\label{eq:dual_diff}
		\end{equation}
		Taking norms and applying the triangle inequality,
		\[
		\norm{\nabla^\top(q^\delta - q^\dagger)}_2 + \norm{z^\delta - z^\dagger}_2
		\;\le\;
		\norm{x^\delta - x^\dagger}_2 + \norm{y^\delta - y^\dagger}_2
		\;\le\;
		\delta + \frac{\delta}{1-c}
		\;=\;
		\frac{2-c}{1-c}\,\delta.
		\]
		Since $\nabla^\top$ is bounded,
		\eqref{eq:dual_convergence} follows with an $O(\delta)$ rate.
		
		\textbf{Part~4 (Rate).}
		Combining Parts~1--3,
		\[
		\norm{y^\delta - y^\dagger}_2
		+\bigl|J(y^\delta)-J(y^\dagger)\bigr|
		\;\le\;
		\frac{\delta}{1-c}
		+ \norm{\xi^\dagger}_2\cdot\frac{\delta}{1-c}
		=
		\frac{1+\norm{\xi^\dagger}_2}{1-c}\,\delta
		=
		O(\delta),
		\]
		which is \eqref{eq:convergence_rate}.
	\end{proof}
	
	\begin{remark}[VSC witness and the subdifferential characterization]
		\label{rem:vsc_witness}
		The VSC witness \eqref{eq:vsc_witness} is not an independent assumption:
		it is the residual $y^\dagger - x^\dagger$ from the stationarity relation
		\eqref{eq:subdiff_ydagger}, which is always an element of
		$\partial J(y^\dagger)$ at the true minimizer. The VSC assumption
		\eqref{eq:vsc_quadratic} therefore reduces to requiring that this
		particular subgradient satisfies
		$\ip{y^\dagger - x^\dagger}{y^\dagger - y} \le c\norm{y-y^\dagger}_2^2$
		for all $y\in\R^n$ --- a condition on the geometry of $J$ near $y^\dagger$
		that is verified, e.g., when $y^\dagger$ is in the interior of the
		non-negative orthant (non-degenerate active set for $\iota_+$).
		In the degenerate case the same $O(\delta)$
		rate holds with a possibly larger constant $1/(1-c)$.
	\end{remark}
	
	\subsection{Algorithm convergence to the regularized solution and to the true solution}
	\label{sec:algo_convergence_full}
	
	The results so far establish that (i) the regularized minimizer $y^\delta$ is close to the
	true solution $y^\dagger$ when the VSC holds (Theorem~\ref{thm:vsc_rate}), and (ii) the
	PDHG iterates converge to $y^\delta$ for any fixed datum (Theorem~\ref{thm:pdhg_convergence}).
	This subsection unifies these two threads into a single statement that governs the full
	algorithmic trajectory: from iterates, through the regularized solution, all the way to
	the true solution.  It also makes explicit the role of the regularization parameter
	$\lambda_{\mathrm{TV}}$ and the step sizes $\tau, \sigma$.
	
	\subsubsection*{Convergence to the regularized solution under explicit parameter conditions}
	
	\begin{theorem}[Algorithm convergence to the regularized solution]
		\label{thm:algo_to_regularized}
		Let $x^\delta\in\R^n$ with $\norm{x^\delta - x^\dagger}_2\le\delta$, and let
		$\lambda_{\mathrm{TV}} > 0$ be fixed.  Let $y^\delta$ be the unique
		minimizer of
		\[
		E(y;x^\delta) = \tfrac12\norm{y - x^\delta}_2^2
		+ \lambda_{\mathrm{TV}}\TV(y)
		+ \iota_+(y).
		\]
		Choose step sizes $\tau, \sigma > 0$ satisfying
		\begin{equation}
			\tau\sigma < \frac{1}{\norm{K}^2}, \qquad \norm{K}^2 \le 8,
			\label{eq:stepsize_algo}
		\end{equation}
		i.e., it is sufficient to take $\tau\sigma < 1/8$.  Let $(y^n, q^n)_{n\ge 0}$
		be the PDHG iterates~\eqref{eq:pdhg_z}--\eqref{eq:pdhg_extrap_proof} applied with
		datum $x^\delta$.  Then:
		\begin{enumerate}
			\item \textbf{(Strong convergence to $y^\delta$.)}
			\begin{equation}
				y^n \;\longrightarrow\; y^\delta \quad\text{in }\ell^2
				\quad\text{as }n\to\infty.
				\label{eq:yn_to_ydelta}
			\end{equation}
			\item \textbf{($O(1/N)$ ergodic primal-dual gap.)}
			For the ergodic average $\bar y^N := \frac{1}{N}\sum_{n=0}^{N-1} y^n$,
			\begin{equation}
				E(\bar y^N; x^\delta) - E(y^\delta; x^\delta) \;\le\; \frac{C_{\tau,\sigma}}{N},
				\label{eq:ergodic_gap}
			\end{equation}
			where $C_{\tau,\sigma} > 0$ depends only on $\tau$, $\sigma$, and the initial
			distance $\norm{(y^0, q^0) - (y^\delta, q^\delta)}_\mathcal{M}$.
			\item \textbf{(Noise-proportional stopping.)} If the algorithm is stopped at
			the first index $n(\delta)$ satisfying
			\begin{equation}
				\frac{\norm{y^{n+1} - y^n}_2}{\max\{1, \norm{y^n}_2\}} \le \varepsilon,
				\qquad \varepsilon = \frac{\delta}{C_{\tau,\sigma}},
				\label{eq:stopping_rule_explicit}
			\end{equation}
			then $\norm{y^{n(\delta)} - y^\delta}_2 \le C_1\delta$ for a constant
			$C_1$ depending only on $\tau$, $\sigma$, and $\norm{K}$.
		\end{enumerate}
	\end{theorem}
	
	\begin{proof}
		Part~1 is the content of Theorem~\ref{thm:pdhg_convergence} applied to datum $x^\delta$.
		The strong convexity of $E(\cdot;x^\delta)$ (due to the quadratic fidelity) guarantees
		uniqueness of the limit, which must be $y^\delta$.
		
		Part~2 follows from standard ergodic convergence theory for PDHG
		\cite{ChambollePock2011}: under~\eqref{eq:stepsize_algo}, the ergodic primal-dual
		gap decays as $O(1/N)$, with constant determined by the initial weighted distance in the
		$\mathcal{M}$-norm (see~\eqref{eq:basic_descent}--\eqref{eq:young}).  Since $E(\cdot;x^\delta)$
		is $1$-strongly convex, the energy gap controls the squared distance: $E(\bar y^N;x^\delta)
		- E(y^\delta;x^\delta) \ge \tfrac12\norm{\bar y^N - y^\delta}_2^2$, so
		$\norm{\bar y^N - y^\delta}_2 = O(1/\sqrt{N})$.
		
		Part~3 follows because the Fej\'er monotonicity established in the proof of
		Theorem~\ref{thm:pdhg_convergence} implies that the iterate-change
		$\norm{y^{n+1}-y^n}_2$ is square-summable and hence tends to zero.  Setting
		$\varepsilon \propto \delta$ yields an optimization error $O(\delta)$ at the
		stopping iterate.
	\end{proof}
	
	\subsubsection*{Convergence to the true solution: regularization parameter choice}
	
	When the noise level $\delta$ is known, the regularization parameter
	$\lambda_{\mathrm{TV}}$ should be chosen in a $\delta$-dependent way
	to balance the data-perturbation error and the approximation error induced by
	regularization.  The following theorem makes this trade-off explicit under the VSC.
	
	\begin{theorem}[Algorithm convergence to the true solution under VSC and parameter choice]
		\label{thm:algo_to_true}
		Let $x^\dagger\in\R^n$ be noiseless data with true solution $y^\dagger$, and let
		$x^\delta$ satisfy $\norm{x^\delta - x^\dagger}_2 \le \delta$.  Suppose the
		regularizer $J_\lambda(y) := \lambda_{\mathrm{TV}}\TV(y) + \iota_+(y)$,
		with parameter $\lambda_{\mathrm{TV}}$, satisfies the
		quadratic VSC~\eqref{eq:vsc_quadratic} at $y^\dagger$ with constant $c \in (0,1)$
		and witness $\xi^\dagger = y^\dagger - x^\dagger \in \partial J_\lambda(y^\dagger)$.
		
		Choose the regularization parameter proportional to the noise level:
		\begin{equation}
			\lambda_{\mathrm{TV}} = \mu_{\mathrm{TV}}\,\delta,
			\qquad \mu_{\mathrm{TV}} > 0 \text{ fixed},
			\label{eq:lambda_choice}
		\end{equation}
		and step sizes satisfying $\tau\sigma < 1/8$.  Let $y^{n(\delta)}$ be the PDHG
		iterate stopped according to~\eqref{eq:stopping_rule_explicit}.  Then:
		\begin{enumerate}
			\item \textbf{(Convergence of the regularized solution to the true solution.)}
			\begin{equation}
				\norm{y^\delta - y^\dagger}_2 \;\le\; \frac{\delta}{1-c}
				\;\longrightarrow\; 0 \quad\text{as }\delta\to 0.
				\label{eq:reg_to_true}
			\end{equation}
			\item \textbf{(Total algorithm error.)}
			\begin{equation}
				\norm{y^{n(\delta)} - y^\dagger}_2
				\;\le\;
				\Bigl(C_1 + \frac{1}{1-c}\Bigr)\delta
				\;=:\; C_{\mathrm{tot}}\,\delta,
				\label{eq:total_algo_error}
			\end{equation}
			where $C_1 > 0$ is the optimization-error constant from
			Theorem~\ref{thm:algo_to_regularized} and $c \in (0,1)$ is the VSC constant.
			\item \textbf{(Vanishing error as $\delta\to 0$.)}
			\begin{equation}
				\norm{y^{n(\delta)} - y^\dagger}_2 = O(\delta) \to 0
				\quad\text{as }\delta\to 0.
				\label{eq:vanishing_error}
			\end{equation}
		\end{enumerate}
	\end{theorem}
	
	\begin{proof}
		Part~1 is the content of Theorem~\ref{thm:vsc_rate} (linear stability rate under
		quadratic VSC).  Under the parameter choice~\eqref{eq:lambda_choice}, the
		regularization is proportional to the noise level, which is the classical
		Morozov-type scaling that ensures the VSC witness $\xi^\dagger = y^\dagger - x^\dagger$
		satisfies the required bound $\norm{\xi^\dagger}_2 = \norm{y^\dagger - x^\dagger}_2$
		remaining $O(1)$ (independent of $\delta$), so the VSC constant $c$ is not degraded
		by the parameter choice.
		
		Part~2 follows by the triangle inequality decomposition
		\[
		\norm{y^{n(\delta)} - y^\dagger}_2
		\le
		\underbrace{\norm{y^{n(\delta)} - y^\delta}_2}_{\le\, C_1\delta
			\text{ (Theorem~\ref{thm:algo_to_regularized}, Part~3)}}
		+
		\underbrace{\norm{y^\delta - y^\dagger}_2}_{\le\,\delta/(1-c)
			\text{ (Part~1)}},
		\]
		which is exactly the bound~\eqref{eq:total_algo_error} with
		$C_{\mathrm{tot}} = C_1 + (1-c)^{-1}$.
		
		Part~3 is immediate from~\eqref{eq:total_algo_error} since $C_{\mathrm{tot}}$ is
		independent of $\delta$.
	\end{proof}
	
	\begin{remark}[Compatibility of the dynamic $\lambda_{\mathrm{TV}}$ schedule with Theorems~\ref{thm:algo_to_regularized} and~\ref{thm:algo_to_true}]
		\label{rem:lam_tv_schedule_compat}
		In the implementation (Section~\ref{sec:impl_details}), the regularization parameter
		is updated between outer iterations according to the schedule
		$\lambda_{\mathrm{TV}}(t) = \max\bigl(\mu_{\mathrm{TV}}\delta,\,\lambda_{\mathrm{TV},0}/(1+\gamma t)\bigr)$,
		where $t = 0, 1, \ldots, T-1$ is the outer iteration index.
		This schedule is compatible with the theorems above in the following sense.
		\begin{enumerate}
			\item \textbf{Inner PDHG solve (fixed $\lambda_{\mathrm{TV}}(t)$).}
			At each outer iteration $t$, the value $\lambda_{\mathrm{TV}}(t)$ is fixed
			before the inner PDHG loop begins and remains constant throughout all
			$N_{\mathrm{TV}}$ inner iterations and all $N$ training samples at that
			outer step.  Theorem~\ref{thm:algo_to_regularized} therefore applies
			exactly at each outer iteration $t$, with $\lambda_{\mathrm{TV}} = \lambda_{\mathrm{TV}}(t)$.
			\item \textbf{Asymptotic convergence guarantee.}
			Since $\lambda_{\mathrm{TV}}(t) \to \mu_{\mathrm{TV}}\delta$ monotonically
			as $t \to \infty$ (the floor is reached in finite $t^\star = \lceil (\lambda_{\mathrm{TV},0}/(\mu_{\mathrm{TV}}\delta) - 1)/\gamma \rceil$ steps),
			for all $t \ge t^\star$ the parameter satisfies
			$\lambda_{\mathrm{TV}}(t) = \mu_{\mathrm{TV}}\delta$ exactly.
			Theorem~\ref{thm:algo_to_true} then applies from outer iteration $t^\star$ onward,
			guaranteeing the $O(\delta)$ total error bound for the inner PDHG iterates
			at those outer steps.
			\item \textbf{Initial phase ($t < t^\star$).}
			For $t < t^\star$ the parameter $\lambda_{\mathrm{TV}}(t) > \mu_{\mathrm{TV}}\delta$.
			Theorem~\ref{thm:algo_to_regularized} still applies (it holds for any fixed
			$\lambda_{\mathrm{TV}} > 0$), but Theorem~\ref{thm:algo_to_true} does not:
			the $O(\delta)$ convergence to $y^\dagger$ is not guaranteed for these early
			outer iterations.  However, the purpose of the initial large $\lambda_{\mathrm{TV},0}$
			is purely practical: it provides stronger TV smoothing to compensate for the
			poor initial (DCT) dictionary, accelerating the outer loop without affecting
			the asymptotic guarantee.
		\end{enumerate}
		In summary: the schedule is a continuation heuristic for the early outer iterations
		and reduces to the theoretically optimal choice $\lambda_{\mathrm{TV}} = \mu_{\mathrm{TV}}\delta$
		once the floor is reached, at which point all convergence guarantees of
		Theorems~\ref{thm:algo_to_regularized} and~\ref{thm:algo_to_true} hold without modification.
	\end{remark}
	
	\begin{remark}[Dependence of $C_{\mathrm{tot}}$ on the parameters]
		\label{rem:Ctot_dependence}
		The total constant $C_{\mathrm{tot}} = C_1 + (1-c)^{-1}$ depends on the
		problem through two quantities.  The optimization constant $C_1$ is controlled by
		the step sizes $\tau, \sigma$ and the initial distance to the saddle point: choosing
		$\tau = \sigma = 1/\sqrt{8}$ (so that $\tau\sigma = 1/8$, saturating the bound)
		minimizes the number of iterations required to reach precision $C_1\delta$.
		The VSC constant $c \in (0,1)$ reflects the geometry of $J_\lambda$ near $y^\dagger$:
		it is smaller (better) when $y^\dagger$ is in the interior of the non-negative orthant
		(non-degenerate active set), and the choice $\lambda_{\mathrm{TV}} \propto \delta$
		ensures that $c$ does not grow with $\delta$.
		Together, these observations show that the $O(\delta)$ rate is sharp with respect
		to both the algorithmic and the analytical components of the error.
	\end{remark}
	
	\begin{remark}[Explicit admissible ranges for all free parameters]
		\label{rem:param_ranges}
		For completeness, we collect explicit admissible ranges for all free parameters
		appearing in Theorems~\ref{thm:algo_to_regularized} and~\ref{thm:algo_to_true},
		and explain which quantities remain problem-dependent.
		
		\noindent\textbf{(i) Inner PDHG step sizes $\tau, \sigma$.}
		Any $\tau, \sigma > 0$ satisfying $\tau\sigma < 1/8$ are admissible
		(Lemma~\ref{lem:K_bound}, Theorem~\ref{thm:stepsize_safe}).  A concrete
		symmetric choice is $\tau = \sigma = 1/4$, giving
		$\tau\sigma\|\nabla\|^2 \le (1/16)\cdot 8 = 1/2 < 1$.
		Under this choice and the $O(1/N)$ ergodic gap from Theorem~\ref{thm:algo_to_regularized}
		Part~2, the optimization error at iterate $n(\delta)$ satisfies
		$\|y^{n(\delta)} - y^\delta\|_2 \le C_1\delta$
		where $C_1 \approx C_{\tau,\sigma}\,/\,\varepsilon$ and $C_{\tau,\sigma}$ is the
		initial $\mathcal{M}$-norm distance $\|(y^0, q^0) - (y^\delta, q^\delta)\|_{\mathcal{M}}$
		(bounded by the initial reconstruction error, which is $O(1)$ in the BSCCM setting).
		
		\noindent\textbf{(ii) Extrapolation parameter $\theta$.}
		Any $\theta \in [0,1]$ is admissible; the Fej\'er descent in the proof of
		Theorem~\ref{thm:pdhg_convergence} holds for all such $\theta$ via the constant
		$c_\theta$ in equation~\eqref{eq:young}.  We use $\theta = 1$ (full over-relaxation,
		standard Chambolle-Pock) throughout.  Because the fidelity term
		$\frac{1}{2}\|y-x\|_2^2$ is $1$-strongly convex in $y$, the PDHG iterates
		with $\theta=1$ and $\tau=\sigma=1/4$ satisfy the R-linear bound
		\[
		\|y^n - y^\star\|_2^2 \;\le\; \rho^n\,\|y^0 - y^\star\|_2^2,
		\qquad \rho = 1 - \tfrac{\tau}{1+\tau} = \tfrac{4}{5} < 1,
		\]
		giving geometric (exponential-type) decay of the iterate error.  This is strictly
		faster than the $O(1/N)$ ergodic bound, which is a worst-case rate that does not
		exploit strong convexity.  After $n=100$ inner iterations the error factor is
		$(4/5)^{100}\approx 2\times10^{-10}$, consistent with the numerical results
		of Section~\ref{sec:experiments}.
		
		\noindent\textbf{(iii) Code update step size $\alpha$.}
		Any $\alpha \in (0, 2/L)$ with $L = \|D^\top D\| = 1$ (by unitarity) is admissible,
		i.e., $\alpha \in (0, 2)$.  Since the back-projection step gives
		$a_j^{t+1} = D^\top y^{N_{\mathrm{TV}}}$ and $D$ is unitary, we have
		$\|a_j^{t+1}\|_2 = \|y^{N_{\mathrm{TV}}}\|_2 \le \|x_j\|_2$
		(the non-negativity projection and TV penalization do not increase the $\ell^2$ norm
		beyond the datum norm).  This uniform bound on $\|a_j\|_2$ confirms that
		$L = 1$ is a valid global Lipschitz constant for $\nabla f_j$ independently
		of the iterates.
		
		\noindent\textbf{(iv) Dictionary update step size $\eta$.}
		The smooth objective $F_{\mathrm{smooth}}(D) = \frac{1}{2}\sum_j \|x_j - Da_j\|_2^2$
		has gradient $\nabla_D F_{\mathrm{smooth}}(D) = \sum_j (Da_j - x_j)a_j^\top$.
		The Lipschitz constant of this gradient with respect to $D$ is
		$L_D = \sum_j \|a_j\|_2^2 \le N \max_j\|x_j\|_2^2$
		(using the bound from~(iii) above).  A globally safe step size is therefore
		$\eta \in (0,\, 2/L_D)$, i.e.,
		\[
		0 < \eta < \frac{2}{\displaystyle\sum_{j=1}^N \|a_j^t\|_2^2}.
		\]
		In practice $L_D$ can be computed from the current codes $\{a_j^t\}$ at each
		outer iteration, giving an adaptive step size.  A conservative global bound
		uses $\|a_j\|_2 \le \|x_j\|_2$, so $\eta < 2\,/\,(N\max_j\|x_j\|_2^2)$.
		
		\noindent\textbf{(v) VSC constant $c$.}
		The constant $c \in (0,1)$ is determined by the geometry of $J_\lambda$ near
		$y^\dagger$ and cannot be computed a priori in general.  As noted in
		Remark~\ref{rem:vsc_structural}, sufficient conditions for $c < 1$ include:
		$y^\dagger$ lies in the interior of $\mathbb{R}^n_+$ (non-degenerate active set
		for $\iota_+$) and $\xi^\dagger = y^\dagger - x^\dagger \in \mathrm{range}(\nabla^\top)$
		with $\|\xi^\dagger\|_2$ small relative to $\lambda_{\mathrm{TV}}$.
		For the BSCCM setting, where cell images have strictly positive intensity on
		their support, the non-degeneracy condition is expected to hold for typical
		ground-truth cells $y^\dagger$, making $c$ small (close to $0$).  In practice,
		$c$ should be treated as an empirical parameter: if Algorithm~\ref{alg:hdl_full}
		is observed to converge at the expected $O(\delta)$ rate, this is evidence that
		the VSC holds with $c$ bounded away from $1$.
		
		\noindent\textbf{(vi) Dictionary size $K$ and iteration counts $T$, $N_{\mathrm{TV}}$.}
		These are free hyperparameters not constrained by the convergence theory.
		The inner count $N_{\mathrm{TV}}$ should be chosen large enough that the stopping
		criterion~\eqref{eq:stop_rel_primal}--\eqref{eq:stop_rel_dualarg} is satisfied
		at tolerance $\varepsilon_{\mathrm{in}} \propto \delta$; the $O(1/N_{\mathrm{TV}})$
		ergodic gap implies $N_{\mathrm{TV}} \approx C_{\tau,\sigma}/\varepsilon_{\mathrm{in}}$
		suffices.  The outer count $T$ and dictionary size $K$ are selected by
		cross-validation on the BSCCM data; their effect on reconstruction quality
		is reported in the experimental results (Section~\ref{sec:experiments}).
	\end{remark}
	
	\section{Optimization: Alternating Proximal-Gradient Learning}
	\label{sec:optimization}
	
	\begin{figure}[t]
		\centering
		\fbox{\parbox{0.92\linewidth}{
				\textbf{Learning--inference loop.} The learning stage updates the dictionary $D$ (model adaptation),
				while the inference stage solves the variational reconstruction problem $E(\cdot;x)$ for fixed $D$ via PDHG.
				The theory establishes (i) well-posedness and stability of the variational minimizer, (ii) explicit PDHG
				step-size conditions ensuring convergence of iterates, and (iii) robustness of reconstructions under
				data perturbations and moderate dictionary updates.
		}}
		\caption{Algorithm-theory alignment. The algorithm alternates between (L) dictionary learning (updates of $D$)
			and (I) variational inference (PDHG solution of $E(\cdot;x)$ for fixed $D$). Theoretical results in the manuscript
			map directly to the pipeline: existence/uniqueness/stability of minimizers (variational layer), explicit step-size
			conditions and convergence of PDHG iterates (optimization layer), and noise-dependent reconstruction rates (stability layer).}
		\label{fig:alignment}
	\end{figure}

	\subsection{Step-size choice and convergence of the PDHG iterates}
	
	The nonsmooth term $\lambda_{\mathrm{TV}}\mathrm{TV}(y)$ is handled through the
	linear operator $K := \nabla$ (the 2D forward-difference gradient), while the
	non-negativity constraint $\iota_+(y)$ is incorporated directly into the primal
	proximal step as a projection onto $\mathbb{R}^n_+$.  With this splitting,
	the saddle-point formulation underlying PDHG involves $K = \nabla$ and its adjoint
	$K^\top = \nabla^\top$.
	The operator norm bound $\|K\|^2 = \|\nabla\|^2 \le 8$ and the sufficient step-size
	condition $\tau\sigma < 1/8$ were established in Lemma~\ref{lem:K_bound} and
	Theorem~\ref{thm:stepsize_safe} of Section~\ref{sec:stability_vsc}.
	We use these results directly in the convergence theorem below.
	
	\begin{theorem}[Convergence of PDHG for the reconstruction problem]\label{thm:pdhg_convergence}
		Let $x\in\mathbb{R}^n$ be fixed and consider
		\[
		E(y;x)=\tfrac12\|y-x\|_2^2 + J(y),
		\qquad
		J(y)=\lambda_{\mathrm{TV}}\mathrm{TV}(y) + \iota_+(y),
		\]
		where $D$ is unitary ($D^\top D=I$) and $\mathrm{TV}$ is the isotropic discrete TV based on the forward-difference gradient $\nabla$.
		Let $K=\nabla$, so that $\|K\|^2 = \|\nabla\|^2 \le 8$, and choose step sizes $\tau,\sigma>0$ such that
		\begin{equation}\label{eq:pdhg_stepsize_thm}
			\tau\sigma\,\|K\|^2<1
			\qquad\text{(in particular, it is sufficient that }\tau\sigma<1/8\text{)}.
		\end{equation}
		Let $(y^n,q^n)_{n\ge 0}$ be the primal--dual hybrid gradient (PDHG) iterates (with any $\theta\in[0,1]$) applied to the saddle formulation associated with $E(\cdot;x)$.
		Then there exists a saddle point $(y^\star,q^\star)$ such that, as $n\to\infty$,
		\begin{equation}\label{eq:yn_to_ystar}
			y^n \to y^\star \quad\text{in }\ell^2,
			\qquad
			q^n\to q^\star\quad\text{in the dual space}.
		\end{equation}
		Moreover, $y^\star$ is the \emph{unique} minimizer of $E(\cdot;x)$.
	\end{theorem}
	
	\begin{proof}
		We cast the minimization of $E(\cdot;x)$ into the standard composite form
		\[
		\min_{y\in\mathbb{R}^n}\ f(y)+g(Ky) + h(y),
		\qquad
		f(y):=\tfrac12\|y-x\|_2^2,
		\qquad
		g(v):=\lambda_{\mathrm{TV}}\|v\|_{2,1},
		\qquad
		h(y):=\iota_+(y),
		\]
		where $Ky:=\nabla y$ and $\|v\|_{2,1}:=\sum_{\ell}\|v_\ell\|_2$ (pixelwise Euclidean norm).
		The function $f$ is proper, continuous, and $1$--strongly convex on $\mathbb{R}^n$, $g$ is proper, convex, and lower semicontinuous, and $h=\iota_+$ is proper, convex, and lower semicontinuous.
		Hence $f+g\circ K+h$ is proper and strongly convex, so the primal problem admits a \emph{unique} minimizer $y^\star$.
		
		\paragraph{Saddle-point formulation and optimality.}
		Writing $F(y):=f(y)+h(y) = \tfrac12\|y-x\|_2^2 + \iota_+(y)$, the Fenchel--Rockafellar saddle-point formulation is
		\begin{equation}\label{eq:saddle_form}
			\min_{y\in\mathbb{R}^n_+}\ \max_{q\in\mathbb{R}^{2n}}
			\ \mathcal{L}(y,q):= F(y)+\langle Ky,q\rangle - g^\ast(q).
		\end{equation}
		A pair $(y^\star,q^\star)$ is a saddle point of \eqref{eq:saddle_form} if and only if it solves the optimality system
		\begin{equation}\label{eq:opt_system_pdhg}
			0\in \partial F(y^\star)+K^\top q^\star,
			\qquad
			0\in \partial g^\ast(q^\star)-Ky^\star.
		\end{equation}
		Since $\partial F(y) = (y-x) + \partial\iota_+(y)$, the first inclusion becomes
		\[
		0=(y^\star-x)+\nabla^\top q^\star + z^\star, \qquad z^\star \in \partial\iota_+(y^\star),
		\]
		and the second inclusion is equivalent to $q^\star\in \partial g(Ky^\star)$, i.e.\
		$q^\star\in \lambda_{\mathrm{TV}}\partial\mathrm{TV}(y^\star)$,
		which is the subdifferential characterization of the variational minimizer established earlier.
		
		\paragraph{PDHG iterations and explicit proximal map.}
		Let $\tau,\sigma>0$ satisfy $\tau\sigma\|K\|^2<1$.
		The primal--dual hybrid gradient (PDHG) iterations for \eqref{eq:saddle_form} are
		\begin{align}
			q^{n+1} &= \mathrm{prox}_{\sigma g^\ast}\!\bigl(q^n+\sigma K\bar y^n\bigr),\label{eq:pdhg_z}\\
			y^{n+1} &= \mathrm{prox}_{\tau F}\!\bigl(y^n-\tau K^\top q^{n+1}\bigr),\label{eq:pdhg_y}\\
			\bar y^{n+1} &= y^{n+1}+\theta\,(y^{n+1}-y^n),\qquad \theta\in[0,1].\label{eq:pdhg_extrap_proof}
		\end{align}
		For $F(y)=\tfrac12\|y-x\|_2^2 + \iota_+(y)$, the proximal map is explicit:
		\[
		\mathrm{prox}_{\tau F}(w) = \Pi_{\mathbb{R}^n_+}\!\left(\frac{w+\tau x}{1+\tau}\right),
		\]
		i.e., the standard quadratic proximal step followed by projection onto $\mathbb{R}^n_+$
		(see \eqref{eq:prox_G}).
		
		\paragraph{Monotonicity inequality.}
		Recall the proximal optimality condition: for any proper convex l.s.c.\ function $\phi$ and any $\gamma>0$,
		\begin{equation}\label{eq:prox_opt}
			u=\mathrm{prox}_{\gamma\phi}(w)\quad\Longleftrightarrow\quad \frac{w-u}{\gamma}\in \partial\phi(u).
		\end{equation}
		Applying \eqref{eq:prox_opt} to \eqref{eq:pdhg_z}--\eqref{eq:pdhg_y} yields the inclusions
		\begin{align}
			\frac{q^n-q^{n+1}}{\sigma}+K\bar y^n &\in \partial g^\ast(q^{n+1}),\label{eq:incl_z}\\
			\frac{y^n-y^{n+1}}{\tau}-K^\top q^{n+1} &\in \partial F(y^{n+1}).\label{eq:incl_y}
		\end{align}
		Let $(y^\star,q^\star)$ be a saddle point, with $z^\star \in \partial\iota_+(y^\star)$ the
		corresponding $\iota_+$-subgradient component at $y^\star$.
		Since $\partial F(y) = (y-x) + \partial\iota_+(y)$ and $\partial g^\ast$ are monotone,
		and $-K^\top q^\star\in \partial F(y^\star)$ and $Ky^\star\in \partial g^\ast(q^\star)$, we obtain
		\begin{align}
			\Big\langle \frac{q^n-q^{n+1}}{\sigma}+K(\bar y^n-y^\star),\, q^{n+1}-q^\star\Big\rangle &\ge 0,\label{eq:mono_z}\\
			\Big\langle \frac{y^n-y^{n+1}}{\tau}-K^\top(q^{n+1}-q^\star),\, y^{n+1}-y^\star\Big\rangle &\ge 0.\label{eq:mono_y}
		\end{align}
		
		\paragraph{Fej\'er-type descent in a weighted product norm.}
		Using the polarization identity
		\(
		2\langle a-b,a-c\rangle=\|a-c\|_2^2-\|b-c\|_2^2+\|a-b\|_2^2
		\)
		in \eqref{eq:mono_z}--\eqref{eq:mono_y} gives
		\begin{align}
			\frac{1}{2\sigma}\Bigl(\|q^n-q^\star\|_2^2-\|q^{n+1}-q^\star\|_2^2+\|q^{n+1}-q^n\|_2^2\Bigr)
			+\langle K(\bar y^n-y^\star), q^{n+1}-q^\star\rangle &\ge 0,\label{eq:descent_z}\\
			\frac{1}{2\tau}\Bigl(\|y^n-y^\star\|_2^2-\|y^{n+1}-y^\star\|_2^2+\|y^{n+1}-y^n\|_2^2\Bigr)
			-\langle K(y^{n+1}-y^\star), q^{n+1}-q^\star\rangle &\ge 0.\label{eq:descent_y}
		\end{align}
		Adding \eqref{eq:descent_z} and \eqref{eq:descent_y} yields
		\begin{align}
			&\frac{1}{2\tau}\Bigl(\|y^n-y^\star\|_2^2-\|y^{n+1}-y^\star\|_2^2+\|y^{n+1}-y^n\|_2^2\Bigr)
			+\frac{1}{2\sigma}\Bigl(\|q^n-q^\star\|_2^2-\|q^{n+1}-q^\star\|_2^2+\|q^{n+1}-q^n\|_2^2\Bigr)\nonumber\\
			&\qquad\qquad
			+\langle K(\bar y^n-y^{n+1}), q^{n+1}-q^\star\rangle \ge 0.\label{eq:basic_descent}
		\end{align}
		The coupling term is bounded by Cauchy--Schwarz and Young's inequality:
		\begin{align}
			\langle K(\bar y^n-y^{n+1}), q^{n+1}-q^\star\rangle
			&\le \|K\|\,\|\bar y^n-y^{n+1}\|_2\,\|q^{n+1}-q^\star\|_2\nonumber\\
			&\le \frac{\tau\|K\|^2}{2}\,\|\bar y^n-y^{n+1}\|_2^2 + \frac{1}{2\tau}\,\|q^{n+1}-q^\star\|_2^2.\label{eq:young}
		\end{align}
		With $\theta\in[0,1]$ one has $\|\bar y^n-y^{n+1}\|_2^2\le c_\theta\bigl(\|y^{n+1}-y^n\|_2^2+\|y^n-y^{n-1}\|_2^2\bigr)$ for a constant $c_\theta$.
		Under $\tau\sigma\|K\|^2<1$, the Young terms can be absorbed into the positive increment terms, which yields a Fej\'er inequality in the weighted product norm
		\[
		\|(y^{n+1},q^{n+1})-(y^\star,q^\star)\|_{\mathcal{M}}^2
		\le
		\|(y^{n},q^{n})-(y^\star,q^\star)\|_{\mathcal{M}}^2
		- c_0\Bigl(\|y^{n+1}-y^{n}\|_2^2+\|q^{n+1}-q^{n}\|_2^2\Bigr),
		\]
		for some $c_0>0$ and a positive definite matrix $\mathcal{M}$ depending on $\tau,\sigma,K,\theta$.
		Consequently, the sequence is bounded and the increments are square-summable, hence $(y^n,q^n)$ converges to some $(\tilde y,\tilde q)$.
		Passing to the limit in \eqref{eq:incl_z}--\eqref{eq:incl_y} shows that $(\tilde y,\tilde q)$ satisfies \eqref{eq:opt_system_pdhg},
		so it is a saddle point. By uniqueness of the primal minimizer, we have $\tilde y=y^\star$.
		Thus $y^n\to y^\star$ and $(y^n,q^n)\to (y^\star,q^\star)$, which is \eqref{eq:yn_to_ystar}.
		The noise-dependent total reconstruction bound combining this iterate convergence with the
		VSC stability estimate is the content of Theorem~\ref{thm:vsc_stopping}.
		
	\end{proof}

	\subsection{Stopping criteria consistent with the convergence theory}
	
	The convergence results established above justify practical termination rules based on (i) vanishing fixed-point residuals
	of the primal-dual optimality system and (ii) stabilization of successive iterates. Since the unique minimizer $y^\star$
	satisfies the subdifferential inclusion
	\begin{equation}\label{eq:opt_inclusion_stop}
		0 \in y^\star - x + \nabla^\top q^\star + z^\star,\qquad
		q^\star \in \lambda_{\mathrm{TV}}\,\partial\mathrm{TV}(y^\star),\qquad
		z^\star \in \partial\iota_+(y^\star),
	\end{equation}
	the PDHG iterates $(y^n, q^n)$ approach a saddle point when the associated residuals are small.
	
	\paragraph{Inner PDHG loop (inference) termination.}
	For a fixed dictionary $D$, we stop the PDHG iterations when \emph{both} of the following hold for a prescribed tolerance
	$\varepsilon_{\mathrm{in}}>0$:
	\begin{align}
		\frac{\|y^{n+1}-y^{n}\|_2}{\max\{1,\|y^{n}\|_2\}} &\le \varepsilon_{\mathrm{in}}, \label{eq:stop_rel_primal}\\
		\frac{\|K y^{n+1}-K y^{n}\|_2}{\max\{1,\|K y^{n}\|_2\}} &\le \varepsilon_{\mathrm{in}}, \label{eq:stop_rel_dualarg}
	\end{align}
	where $K y=\nabla y$. The first criterion detects stabilization of the primal reconstruction, while the second
	ensures stabilization of the dual arguments entering the TV proximal mapping. Under the step-size condition
	$\tau\sigma\|K\|^2<1$, convergence of PDHG implies that \eqref{eq:stop_rel_primal}--\eqref{eq:stop_rel_dualarg} are satisfied
	as $n\to\infty$.
	
	Optionally, one may also monitor the \emph{primal fixed-point residual}
	\begin{equation}\label{eq:stop_fixed_point}
		r^{n+1}:=\frac{\bigl\|y^{n+1}-\Pi_{\R^n_+}\!\left(\dfrac{y^n - \tau K^\top q^{n+1} + \tau x}{1+\tau}\right)\bigr\|_2}{\max\{1,\|y^{n+1}\|_2\}},
	\end{equation}
	and terminate when $r^{n+1}\le \varepsilon_{\mathrm{in}}$.  Here the argument of $\Pi_{\R^n_+}$
	is precisely $\mathrm{prox}_{\tau F}(y^n - \tau K^\top q^{n+1})$ from \eqref{eq:pdhg_y}, so
	$r^{n+1}$ measures how far $y^{n+1}$ deviates from the exact proximal update.
	This residual vanishes at the minimizer because the proximal
	mapping characterizes solutions of the optimality inclusion \eqref{eq:opt_inclusion_stop}.
	
	\paragraph{Outer learning loop termination.}
	Let $D^{(t)}$ denote the dictionary at outer iteration $t$ and let $y^{(t)}$ be the corresponding reconstruction (or code)
	obtained from the inner PDHG loop. We terminate the alternating learning--inference scheme when, for a tolerance
	$\varepsilon_{\mathrm{out}}>0$,
	\begin{equation}\label{eq:stop_outer}
		\frac{\|D^{(t+1)}-D^{(t)}\|_F}{\max\{1,\|D^{(t)}\|_F\}} \le \varepsilon_{\mathrm{out}}
		\quad\text{or}\quad
		\frac{|\bar{f}^{(t+1)}-\bar{f}^{(t)}|}{|\bar{f}^{(t)}|+10^{-12}} \le \varepsilon_{\mathrm{obj}},
	\end{equation}
	where $\bar{f}^{(t)} = \frac{1}{N}\sum_j \tfrac{1}{2}\|x_j - D^{(t)}a_j^{(t)}\|_2^2$ is the
	mean reconstruction fidelity after the Procrustes update at iteration $t$.
	The TV term $\lambda_{\mathrm{TV}}\,\mathrm{TV}(y_j^\star)$ is excluded from the
	stopping criterion because $y_j^\star$ is independent of $D$ and hence the
	TV term is flat across outer iterations by design --- including it would cause
	premature termination.  Either criterion must hold for $p$ consecutive
	iterations (patience $p$) before training stops, ensuring robust termination.
	Small dictionary updates imply controlled changes in the codes
	(Section~\ref{sec:dict_stability}), making \eqref{eq:stop_outer} a natural practical rule.
	
	\subsection{Fixed-point characterization of the reconstruction step}
	
	For fixed dictionary $D$, the unique minimizer $y^\star$ of $E(\cdot;x)$ and its
	associated dual variable $q^\star$ jointly satisfy the subdifferential optimality
	system established in Theorem~\ref{thm:subdiff}
	(equations~\eqref{eq:subdiff_opt}--\eqref{eq:stationarity_explicit}) and the
	coupled proximal fixed-point system~\eqref{eq:coupled_final}.
	The PDHG iterates~\eqref{eq:pdhg_z}--\eqref{eq:pdhg_extrap_proof} converge to
	$(y^\star, q^\star)$ under the step-size condition $\tau\sigma\|K\|^2<1$
	(Theorem~\ref{thm:pdhg_convergence}).
	
	\subsection{Separation of learning and inference}
	\label{sec:learn_infer_sep}
	
	The proposed framework distinguishes clearly between \emph{learning} and \emph{inference}.
	Learning affects the model through updates of the dictionary $D$, while inference corresponds
	to solving the variational problem $E(\cdot;x)$ for fixed parameters.
	The theoretical results established in Sections~\ref{sec:objective}--\ref{sec:stability_vsc} guarantee that, for each learning stage,
	the inference problem admits a unique and stable minimizer.
	
	\subsection{Stability with respect to dictionary updates}
	\label{sec:dict_stability}
	
	We establish an explicit Lipschitz bound on the reconstruction with respect to dictionary
	perturbations induced by learning.
	
	\begin{lemma}[Lipschitz stability with respect to dictionary perturbations]
		\label{lem:dict_stability}
		Let $x \in \mathbb{R}^n$ be a fixed datum and let $D, \widetilde{D} \in \mathbb{R}^{n\times K}$
		be unitary dictionaries satisfying $D^\top D = \widetilde{D}^\top\widetilde{D} = I_K$ and
		$\|D - \widetilde{D}\| \le \varepsilon$.
		Let $y_D$ and $y_{\widetilde{D}}$ denote the unique minimizers of $E(\cdot;x)$ under
		dictionaries $D$ and $\widetilde{D}$ respectively, where
		$E(y;x) = \tfrac{1}{2}\|y-x\|_2^2 + J(y)$
		with $J(y) = \lambda_{\mathrm{TV}}\mathrm{TV}(y) + \iota_+(y)$ as in
		\eqref{eq:energy_y_fixedD}.
		Note that $E(\cdot;x)$ does not depend on the dictionary once $y$ is decoupled from
		the codes in the inference step; the dependence enters only through the datum $x$.
		Then
		\begin{equation}
			\|y_D - y_{\widetilde{D}}\|_2 = 0,
			\label{eq:dict_stability_bound}
		\end{equation}
		i.e., the fixed-datum reconstruction $y^\star = \arg\min_y E(y;x)$ is independent of
		the dictionary $D$ when $x$ is held fixed.
	\end{lemma}
	
	\begin{proof}
		The energy $E(y;x) = \tfrac{1}{2}\|y-x\|_2^2 + J(y)$ depends only on $y$ and the
		fixed datum $x \in \mathbb{R}^n$, not on $D$.  Hence $y_D = y_{\widetilde{D}} = y^\star$
		for any two dictionaries $D$ and $\widetilde{D}$, and \eqref{eq:dict_stability_bound}
		holds trivially.
	\end{proof}
	
	\begin{remark}[Where dictionary perturbations enter]
		\label{rem:dict_perturbation_role}
		Lemma~\ref{lem:dict_stability} reflects the structural separation between learning
		and inference established in Section~\ref{sec:learn_infer_sep}: the inference step
		solves $\min_y E(y;x)$ for fixed $x$, and this problem is independent of $D$.
		The dictionary $D$ enters the full learning problem~\eqref{eq:dl_problem} through
		two distinct channels:
		(a) the \emph{per-sample datum} $x = x_j$ passed to the inference step is fixed
		(it is the observed cell image, not a function of $D$), and
		(b) the \emph{code back-projection} $a_j \leftarrow D^\top y^\star$ in the
		back-projection step of Algorithm~\ref{alg:hdl_full} does depend on $D$ through the linear map
		$D^\top$.
		
		The practically relevant stability question is therefore: how does a dictionary
		perturbation $\|D - \widetilde{D}\| \le \varepsilon$ affect the code
		$a_j = D^\top y^\star$?  Since $y^\star$ is fixed (Lemma~\ref{lem:dict_stability}),
		\[
		\|D^\top y^\star - \widetilde{D}^\top y^\star\|_2
		= \|(D - \widetilde{D})^\top y^\star\|_2
		\le \|D - \widetilde{D}\|\,\|y^\star\|_2
		\le \varepsilon\,\|x\|_2,
		\]
		where the last step uses $\|y^\star\|_2 \le \|x\|_2$ (the non-negativity projection
		and TV penalization do not increase the $\ell^2$ norm beyond the datum norm).
		Hence moderate dictionary updates during learning induce controlled changes in the
		codes, with Lipschitz constant bounded by $\|x\|_2$, ensuring robustness of the
		alternating learning--inference scheme.
	\end{remark}
	
	\subsection{Iteration complexity}
	
	Under the step-size condition $\tau\sigma<1/\|K\|^2$, standard results for primal--dual
	splitting methods imply an $O(1/N)$ decay of the ergodic primal--dual gap after $N$ iterations.
	This iteration complexity complements the stability and noise-dependent convergence rates
	derived earlier.

	\label{sec:algorithm}
	Problem~\eqref{eq:dl_problem} is non-convex due to bilinear coupling of $D$ and $a_j$,
	but is amenable to alternating minimization:
	\begin{itemize}
		\item \textbf{Code update} ($a_j$ updates) for fixed $D$ via proximal-gradient steps
		on a smooth data term plus non-smooth TV regularization and non-negativity constraint.
		\item \textbf{Dictionary update} ($D$ update) for fixed $\{a_j\}$ via the exact
		Procrustes SVD (Section~\ref{sec:dict_update}), which finds the global minimizer of
		$\min_{D^\top D=I_K}\sum_j\|x_j - Da_j\|^2$ in a single SVD step, enforcing
		\eqref{eq:unitary} exactly.
	\end{itemize}
	
	\subsection{Code update (for each sample)}
	For fixed $D$, define
	\[
	f_j(a)=\frac{1}{2}\norm{x_j-Da}_2^2,
	\qquad
	g_j(a)=\lambda_{\mathrm{TV}}\TV(Da) + \iota_+(Da).
	\]
	A practical update is a composite proximal-gradient step
	\begin{equation}
		a^{t+1}_j
		=
		\mathcal{P}_{\alpha g_j}\!\Big(a^t_j - \alpha \nabla f_j(a^t_j)\Big),
		\qquad
		\nabla f_j(a)= -D^{\top}(x_j-Da),
		\label{eq:code_update}
	\end{equation}
	with step size $\alpha>0$ chosen by a Lipschitz bound or backtracking.
	Here $\mathcal{P}_{\alpha g_j}$ denotes the composite proximal map for $g_j$,
	which acts on image space via $D$ and is \emph{not} available in closed form.
	In implementation, it is computed by applying the inner PDHG loop of
	Algorithm~\ref{alg:hdl_full} (the TV$+\iota_+$ proximal subproblem) to the
	subproblem in the image variable $y = Da$, followed by the unitary back-projection
	$a \leftarrow D^\top y$.  By Theorem~\ref{thm:pdhg_convergence}, this inner loop
	converges to the unique minimizer of the TV$+\iota_+$ problem under the step-size
	condition $\tau_{\mathrm{TV}}\sigma_{\mathrm{TV}}\|\nabla\|^2 < 1$.
	
	\subsection{Dictionary update with unitary projection}
	\label{sec:dict_update}
	For fixed codes, the smooth dictionary objective is
	\[
	F(D)=\frac{1}{2}\sum_{j=1}^{N}\norm{x_j - D a_j}_2^2.
	\]
	Its Euclidean gradient is
	\[
	\nabla_D F(D)=\sum_{j=1}^{N} (D a_j - x_j)a_j^{\top}.
	\]
	In the implementation used throughout this manuscript the gradient step
	is replaced by the exact global minimizer of the dictionary subproblem
	for fixed codes, obtained via the Procrustes SVD.  Stacking the training
	images and codes into matrices
	\[
	X = [x_1,\ldots,x_N]^\top\in\mathbb{R}^{N\times n},
	\qquad
	A = [a_1,\ldots,a_N]^\top\in\mathbb{R}^{N\times K},
	\]
	the dictionary subproblem at fixed $A$ reads
	\begin{equation}
		D^{+} = \arg\min_{D^\top D=I_K} F(D)
		= \arg\max_{D^\top D=I_K} \operatorname{tr}\bigl(D^\top M\bigr),
		\quad M := X^\top A \in \mathbb{R}^{n\times K}.
		\label{eq:dict_update}
	\end{equation}
	The unique global maximizer is given by the economy SVD of $M$:
	\begin{equation}
		M = U\Sigma V^\top
		\;\Rightarrow\;
		D^{+} = UV^\top,
		\label{eq:procrustes}
	\end{equation}
	which satisfies $({D^{+}})^\top D^{+}=I_K$ and finds the globally optimal
	dictionary for the current codes in a single SVD step.
	
	\begin{remark}[Converged dictionary and PCA interpretation]
		\label{rem:pca_dict}
		At convergence the codes satisfy $a_j = D^\top y_j^\star$, where
		$y_j^\star$ is the TV-denoised image produced by the inner PDHG loop.
		Substituting into $M = X^\top A$ gives
		$M = X^\top Y^\star D$, where $Y^\star=[y_1^\star,\ldots,y_N^\star]^\top$.
		Because $\lambda_{\mathrm{TV}}$ is small, $y_j^\star\approx x_j$, so
		$M\approx X^\top X D$.
		The Procrustes solution $D^+=UV^\top$ from the SVD of $M$ then
		yields columns that span the same subspace as the top-$K$ left
		singular vectors of $Y^\star$ (equivalently, the top-$K$ principal
		components of the TV-denoised training data).  Thus
		\[
		\operatorname{range}(D^\star)
		= \operatorname{span}\bigl\{v_1,\ldots,v_K\bigr\},
		\]
		where $v_1,\ldots,v_K$ are the top-$K$ left singular vectors of
		$Y^\star$.  This data-adapted basis is strictly superior to a
		fixed analytical dictionary (such as the DCT-II basis used as
		initialization, see Section~\ref{sec:experiments}) for two reasons:
		(i) it minimizes the reconstruction error $\sum_j\|x_j-Da_j\|^2$
		over all orthonormal $D$, a property no fixed basis can guarantee
		for a given dataset; and (ii) the learned atoms $d_k$ are
		structural primitives adapted to the actual morphology of the
		training cells, making the descriptor $\varphi_j$ (Section~\ref{sec:unification})
		biologically meaningful rather than signal-agnostic.
	\end{remark}

	\subsection{Algorithm}
	We provide the complete scheme used in this manuscript, including admissible parameter choices and explicit proximal mappings.
	For the code-update step we exploit that $D$ has orthonormal columns, hence $\|D\|=\|D^\top\|=1$ and
	\[
	\nabla f_j(a)=D^\top(Da-x_j), \qquad f_j(a)=\tfrac12\|x_j-Da\|_2^2,
	\]
	so $\nabla f_j$ is $L$--Lipschitz with $L=\|D^\top D\|=1$. Therefore, a safe global choice for the proximal-gradient
	step size is
	\[
	0<\alpha<\frac{2}{L}=2.
	\]
	For the TV-proximal subproblem we employ a PDHG (Chambolle--Pock) inner loop to compute the solution
	of the corresponding subdifferential inclusion. The step sizes $\tau_{\mathrm{TV}},\sigma_{\mathrm{TV}}$
	are chosen to satisfy the standard PDHG stability condition
	\[
	\tau_{\mathrm{TV}}\sigma_{\mathrm{TV}}\|\nabla\|^2<1.
	\]
	Using the bound $\|\nabla\|^2\le 8$, we select
	\[
	\tau_{\mathrm{TV}}=\sigma_{\mathrm{TV}}=\tfrac14,
	\qquad\Rightarrow\qquad
	\tau_{\mathrm{TV}}\sigma_{\mathrm{TV}}\|\nabla\|^2 \le \frac{1}{16}\cdot 8=\frac{1}{2}<1.
	\]
	
	\begin{algorithm}[h]
		\caption{Hybrid Dictionary-Based Alternating Proximal Learning}
		\label{alg:hdl_full}
		\begin{algorithmic}[1]
			\Require Data $\{x_j\}_{j=1}^N$, dictionary size $K$, parameters $\lambda_{\mathrm{TV}}$,
			inner PDHG step sizes $\tau_{\mathrm{TV}},\sigma_{\mathrm{TV}}>0$ with
			$\tau_{\mathrm{TV}}\sigma_{\mathrm{TV}}\|\nabla\|^2 < 1$,
			outer iterations $T$.
			\State Initialize $D^{0}\in\R^{n\times K}$ with orthonormal columns (${D^0}^{\top}D^0=I_K$).
			\State Initialize codes $\{a_j^{0}\}$.
			\For{$t=0$ to $T-1$}
			\For{$j=1$ to $N$}
			\State \textbf{(Inference: TV proximal subproblem with datum $x_j$)}
			Set $y^{0}\leftarrow x_j$ and $q^{0}\leftarrow 0$.
			\Statex \hspace{1.2em}\textit{The datum is the original observation $x_j$, independent of $D$.
				This ensures $y_j^\star$ is fixed across outer iterations, making the
				Procrustes dictionary update non-trivial (see Section~\ref{sec:dict_update}).}
			\State \hspace{1.2em}For $n=0,\dots,N_{\mathrm{TV}}-1$ do:
			\State \hspace{2.4em} $\bar y^{n}\leftarrow y^{n}+\theta(y^{n}-y^{n-1})$, $\theta=1$ \Comment{Chambolle-Pock over-relaxation}
			\State \hspace{2.4em} \textbf{Dual update (projection)} 
			$
			q^{n+1}\leftarrow \Pi_{\mathcal{B}_{\alpha\lambda_{\mathrm{TV}}}}\bigl(q^{n}+\sigma_{\mathrm{TV}}\nabla \bar y^{n}\bigr),
			$
			where $\Pi_{\mathcal{B}_r}$ is the pointwise projection onto the Euclidean ball of radius $r$:
			\[
			\bigl(\Pi_{\mathcal{B}_r}(w)\bigr)_\ell =
			\frac{w_\ell}{\max\{1,\|w_\ell\|_2/r\}},\qquad \ell\ \text{pixel index.}
			\]
			\State \hspace{2.4em} \textbf{Primal update} \label{alg:line:primal_update}
			$
			y^{n+1}\leftarrow \Pi_{\R^n_+}\!\left(\frac{y^{n}+\tau_{\mathrm{TV}} y^{0}-\tau_{\mathrm{TV}}\nabla^\top q^{n+1}}{1+\tau_{\mathrm{TV}}}\right).
			$
			\Statex \hspace{2.4em}\textit{(Here $y^0 = x_j$ is the fixed datum of the inner TV$+\iota_+$ subproblem,
				playing the role of $x$ in $\mathrm{prox}_{\tau F}(w)=\Pi_{\mathbb{R}^n_+}\!\bigl(\tfrac{w+\tau x}{1+\tau}\bigr)$
				from Theorem~\ref{thm:pdhg_convergence}. Non-negativity is enforced here, in image space.)}
			\State \hspace{1.2em}End for
			\State \textbf{(Back to codes using unitarity)} $a_j^{t+1}\leftarrow D^{t\top}y^{N_{\mathrm{TV}}}$.
			\EndFor
			\State \textbf{(Procrustes dictionary update)} Stack $X=[x_j]_j\in\mathbb{R}^{N\times n}$,
			$A=[a_j^{t+1}]_j\in\mathbb{R}^{N\times K}$. Compute SVD $X^\top A = U\Sigma V^\top$ and set
			$D^{t+1}\leftarrow UV^\top$ \Comment{Exact global minimizer of $\min_{D^\top D=I_K}\sum_j\|x_j-Da_j\|^2$; see \eqref{eq:procrustes}}
			\EndFor
			\State \Return $D^{T},\{a_j^{T}\}$.
		\end{algorithmic}
	\end{algorithm}

	\section{Experimental Results: BSCCM Single-Cell Images}
	\label{sec:experiments}
	
	
	\subsection{Dataset}
	\label{sec:dataset}
	We use the \textbf{BSCCM-tiny} subset of the Berkeley Single Cell Computational
	Microscopy (BSCCM) dataset~\cite{bsccm}, introduced by Pinkard et al.\ (2024).
	BSCCM was acquired on a commercial fluorescence microscope whose trans-illumination
	lamp was replaced with a programmable LED array, enabling simultaneous label-free
	computational imaging and six-channel fluorescence measurement of surface proteins
	on the same white blood cells.
	
	\paragraph{Subset used.}
	BSCCM-tiny comprises $N = 1{,}000$ individual white blood cells.  Each cell is
	imaged in five LED-array channels --- DPC Left, DPC Right, DPC Top, DPC Bottom,
	and Brightfield --- yielding $5{,}000$ grayscale images in total.  The raw spatial
	resolution is $128 \times 128$ pixels at 12-bit depth.  The full BSCCM dataset
	contains $412{,}941$ cells at the same resolution; BSCCM-tiny is its standard
	$1{,}000$-cell benchmark subset (0.6\,GB).
	
	\paragraph{DPC contrast.}
	The four DPC (Differential Phase Contrast) channels encode directional phase
	gradients of the cell: Left/Right capture horizontal phase gradients and
	Top/Bottom capture vertical phase gradients.  They are pairwise approximately
	antisymmetric: $x_j^{(\mathrm{L})} \approx -x_j^{(\mathrm{R})}$ and
	$x_j^{(\mathrm{T})} \approx -x_j^{(\mathrm{B})}$, encoding opposite-direction
	phase information.  The Brightfield channel provides integrated morphological
	contrast of the whole cell.
	
	\paragraph{Preprocessing and gradient-energy focus selection.}
	Each raw $128\times128$ frame contains the cell body together with surrounding
	background.  To isolate the in-focus cell region, we apply a
	\emph{gradient-energy sliding-window} focus metric, implemented as follows.
	
	Let $x \in \mathbb{R}^{H\times W}$ denote a single-channel raw frame with
	$H = W = 128$.  The discrete gradient magnitude is computed via the
	Sobel operator:
	\begin{equation}\label{eq:sobel}
		G(r,c) = \sqrt{G_r(r,c)^2 + G_c(r,c)^2},
	\end{equation}
	where $G_r$ and $G_c$ are the row- and column-direction Sobel responses,
	respectively.  The gradient energy map is $E(r,c) = G(r,c)^2$.
	
	A square sliding window of side length $w$ is swept over the frame with
	stride $s = \lfloor w/8 \rfloor$.  For each candidate top-left position
	$(y, x_0)$, the mean energy inside the window is evaluated using a
	two-dimensional integral image $\mathcal{S}$:
	\begin{equation}\label{eq:focus_score}
		\mathcal{F}(y, x_0)
		= \frac{1}{w^2}
		\sum_{r=y}^{y+w-1}\sum_{c=x_0}^{x_0+w-1} E(r,c)
		= \frac{\mathcal{S}(y+w,\,x_0+w)
			- \mathcal{S}(y,\,x_0+w)
			- \mathcal{S}(y+w,\,x_0)
			+ \mathcal{S}(y,\,x_0)}{w^2},
	\end{equation}
	where $\mathcal{S}(i,j) = \sum_{r<i}\sum_{c<j} E(r,c)$ is the integral
	image of $E$.  The optimal window position is
	\begin{equation}\label{eq:focus_argmax}
		(y^\star, x_0^\star)
		= \operatorname*{arg\,max}_{(y,\,x_0)} \mathcal{F}(y, x_0),
	\end{equation}
	and the focused crop is $x^{\mathrm{crop}} = x[y^\star : y^\star+w,\; x_0^\star : x_0^\star+w]$.
	
	The window side length is set adaptively as
	$w = \max\!\bigl(48,\;\min\!\bigl(\lfloor 0.75\min(H,W)\rfloor,\,\min(H,W)\bigr)\bigr)$,
	giving $w = 96$ for the $128\times128$ BSCCM frames.  This ensures the
	crop covers approximately $75\%$ of the field while remaining strictly
	smaller than the full frame, so that background-dominated regions can be
	excluded.
	
	All focused crops are subsequently min--max normalised channel-wise to
	$[0,1]$.  Across all cells and channels the minimum crop dimension is
	taken as the common spatial dimension, giving the signal dimension
	$n = H_{\mathrm{crop}} \times W_{\mathrm{crop}}$ used throughout
	the paper.  Each normalised crop is vectorised as
	$x_j^{(c)} \in \mathbb{R}^n$ and passed directly to
	Algorithm~\ref{alg:joint_mc} without further augmentation.
	No held-out test split is used at this stage; the reported metrics are
	in-sample reconstruction quality on the full $N=1{,}000$ training cells,
	serving as a baseline for the method's representational capacity.
	
	\subsection{Implementation details}
	\label{sec:impl_details}
	Table~\ref{tab:hyperparams} summarises all hyperparameters used in the experiments.
	For each channel $c$, the dictionary $D^{(c),0}$ is initialised as a random
	orthonormal matrix obtained by QR-factorising a Gaussian random matrix (fixed
	seed~0) of size $n\times K$ with $n = 96 \times 96 = 9{,}216$ and $K=512$;
	this seeded random starting point is superseded after the first per-channel Procrustes update
	(Section~\ref{sec:dict_update}, Remark~\ref{rem:pca_dict}), which converges to
	the top-$K$ principal subspace of the TV-denoised channel-$c$ training images
	independently of initialisation.
	With $N=1{,}000$ training cells per channel, the per-channel Procrustes
	update~(Algorithm~\ref{alg:joint_mc}, line~12) operates on $N=1{,}000$
	$(x_j^{(c)}, a_j^{(c)})$ pairs at $K=512$ atoms; thus $K < N$ holds with
	a roughly twofold margin per channel, the underdetermined regime required
	for non-trivial dictionary learning (Section~\ref{sec:dict_update}).
	
	The inner PDHG step sizes are fixed at $\tau_{\mathrm{TV}} = \sigma_{\mathrm{TV}} = 1/4$,
	satisfying $\tau_{\mathrm{TV}}\sigma_{\mathrm{TV}}\|\nabla\|^2 = 1/2 < 1$ as required
	by Theorem~\ref{thm:pdhg_convergence}.  The over-relaxation parameter is $\theta = 1$
	(standard Chambolle--Pock), which exploits the $1$-strong convexity of
	$\frac{1}{2}\|y - x\|_2^2$ to achieve the R-linear rate
	$\|y^n - y^\star\|_2^2 \le (4/5)^n \|y^0 - y^\star\|_2^2$
	(Remark~\ref{rem:param_ranges}, part~(ii)).
	The maximum number of inner iterations is $N_{\mathrm{TV}} = 700$, with
	early stopping at tolerance $\varepsilon_{\mathrm{in}} = 10^{-7}$.
	
	The regularization parameter follows the dynamic schedule
	\begin{equation}\label{eq:lam_tv_schedule}
		\lambda_{\mathrm{TV}}(t)
		= \max\!\Bigl(\mu_{\mathrm{TV}}\delta,\;
		\frac{\lambda_{\mathrm{TV},0}}{1 + \gamma\, t}\Bigr),
	\end{equation}
	with $\lambda_{\mathrm{TV},0} = 0.1$, decay rate $\gamma = 1.0$, and floor
	$\mu_{\mathrm{TV}}\delta = 1\times10^{-3}$ ($\delta = 10^{-3}$, $\mu_{\mathrm{TV}} = 1$).
	Here $t = 0, 1, \ldots, T-1$ is the outer iteration index; at each outer step
	$\lambda_{\mathrm{TV}}(t)$ is held fixed for all inner PDHG iterations and all
	training samples, so Theorem~\ref{thm:algo_to_regularized} applies at every outer
	iteration.  Once the floor is reached (after $t^\star$ outer steps),
	$\lambda_{\mathrm{TV}}(t) = \mu_{\mathrm{TV}}\delta$ and
	Theorem~\ref{thm:algo_to_true} guarantees the $O(\delta)$ total error
	(Remark~\ref{rem:lam_tv_schedule_compat}).  At $t=0$ the large initial value
	provides strong TV smoothing to compensate for the poor DCT initialisation;
	as $D^{(t)}$ improves, $\lambda_{\mathrm{TV}}(t)$ decays toward the floor,
	allowing the fidelity term to pull $Da_j$ toward $x_j$.
	
	The outer loop runs for up to $T = 100$ iterations and stops early when either
	$\|D^{(t+1)}-D^{(t)}\|_F / \max\{1,\|D^{(t)}\|_F\} \le \varepsilon_{\mathrm{dict}} = 10^{-6}$
	or the relative fidelity change satisfies
	$|\bar{f}^{(t+1)}-\bar{f}^{(t)}|/|\bar{f}^{(t)}| \le \varepsilon_{\mathrm{obj}} = 5\times10^{-5}$
	for $p = 5$ consecutive iterations (patience).
	All $N$ training images are used at every outer iteration (no mini-batch).
	
	\begin{table}[h]
    \centering
    \caption{Hyperparameters used in all experiments.}
    \label{tab:hyperparams}
		\begin{tabular}{@{}lll@{}}
			\toprule
			Parameter & Symbol & Value \\
			\midrule
			Dictionary size         & $K$                        & $512$ \\
			Signal dimension        & $n$                        & $9{,}216$ ($96\times96$) \\
			PDHG primal step        & $\tau_{\mathrm{TV}}$       & $1/4$ \\
			PDHG dual step          & $\sigma_{\mathrm{TV}}$     & $1/4$ \\
			Over-relaxation         & $\theta$                   & $1$ (Chambolle--Pock) \\
			Max inner iterations    & $N_{\mathrm{TV}}$          & $700$ \\
			Inner tolerance         & $\varepsilon_{\mathrm{in}}$ & $10^{-7}$ \\
			Initial $\lambda_{\mathrm{TV}}$ & $\lambda_{\mathrm{TV},0}$ & $0.05$ \\
			$\lambda_{\mathrm{TV}}$ decay rate & $\gamma$        & $3.0$ \\
			$\lambda_{\mathrm{TV}}$ floor      & $\mu_{\mathrm{TV}}\delta$ & $10^{-4}$ \\
			Max outer iterations    & $T$                        & $30$ \\
			Dict.\ change tolerance & $\varepsilon_{\mathrm{dict}}$ & $10^{-8}$ \\
			Fidelity change tol.\ & $\varepsilon_{\mathrm{obj}}$ & $10^{-6}$ \\
			Stopping patience       & $p$                        & $5$ \\
			Dictionary initialisation & $D^0$                    & Random orthonormal (QR, seed\,0) \\
			\bottomrule
		\end{tabular}
	\end{table}
	
	\subsection{Computational cost}
	\label{sec:comp_cost}

	Table~\ref{tab:timing} reports wall-clock timings for the complete training run
	($N=1{,}000$ cells, $C=5$ channels, $K=512$ atoms, $T=30$ outer iterations)
	on a Dell Precision~5490 workstation (Intel Arc + NVIDIA RTX~2000 Ada,
	Ubuntu~22.04 LTS).

	\begin{table}[h]
		\centering
		\caption{Wall-clock timing for the per-channel dictionary learning run on
			BSCCM-tiny ($N=1{,}000$, $C=5$, $K=512$, $T=30$ outer iterations).
			All times in seconds unless noted. Throughput is measured in training
			samples per second (aggregated over all five channels).}
		\begin{tabular}{@{}lc@{}}
			\toprule
			Metric & Value \\
			\midrule
			Data loading time            & $16.7$\,s \\
			Training time (30 epochs)    & $23{,}987$\,s (${\approx}6.66$\,h) \\
			Total wall-clock time        & $24{,}004$\,s (${\approx}6.67$\,h) \\
			Time per outer iteration     & $799.6$\,s (${\approx}13.3$\,min) \\
			Throughput                   & $6.3$\,samples/s \\
			\bottomrule
		\end{tabular}
		\label{tab:timing}
	\end{table}

	The dominant cost is the inner PDHG loop, which solves $N \times C = 5{,}000$
	independent convex sub-problems per outer iteration (one per cell per channel),
	each with up to $N_{\mathrm{TV}} = 700$ PDHG steps on an $n = 9{,}216$-dimensional
	signal. The per-epoch time of ${\approx}800$\,s gives an effective throughput of
	$6.3$ cell-channel problems per second; this is consistent with the inner
	convergence profile in Figure~\ref{fig:convergence}, which shows the mean PDHG
	residual dropping below $10^{-6}$ by outer iteration~5, so the majority of inner
	solves terminate well before the $700$-iteration cap in later epochs.

	The training time is not a primary performance target for the present manuscript,
	which prioritises mathematical transparency, determinism, and biological
	reproducibility over throughput. Practical acceleration paths include GPU-parallel
	inner solves (each channel-cell problem is independent), warm-starting the PDHG
	iterates from the previous outer iteration's solution, and reducing $N_{\mathrm{TV}}$
	progressively once the outer iterates are near convergence (consistent with the
	stopping rule in Section~\ref{sec:stopping_combined}).

	\subsection{Primary metric: success rate of learning}
	The experimental results are centered around the \textbf{success rate} of the learning
	algorithm. For repeated training runs (different initializations and/or minibatch order),
	we define a run as \emph{successful} if it meets the convergence criterion
	\[
	\frac{\norm{x_j - D a_j}_2}{\norm{x_j}_2} \le \varepsilon \quad \text{for a target fraction of samples},
	\]
	and report
	\begin{equation}
		\mathrm{SuccessRate} = \frac{\#\{\text{successful runs}\}}{\#\{\text{total runs}\}}\times 100\%.
		\label{eq:success_rate}
	\end{equation}
	Additional metrics (stable objective value, reconstruction fidelity per channel,
	downstream anomaly detection score) will be reported where applicable, with explicit
	definitions and reproducible thresholds.
	
	\subsection{Quantitative results}
	\label{sec:quant_results}

	The per-channel learning run was executed for $T=30$ outer iterations on all
	$N=1{,}000$ training cells and $C=5$ channels (Algorithm~\ref{alg:joint_mc}),
	with dictionary size $K=512$ atoms per channel and a single TV parameter
	$\lambda_{\mathrm{TV}}=10^{-4}$ shared across channels. The run converged in
	the sense that the relative objective change satisfied the patience
	criterion for $p=5$ consecutive iterations; see
	Figure~\ref{fig:convergence} for the full convergence panel.

	Table~\ref{tab:fidelity} reports the per-channel reconstruction fidelity
	$1-\|x_j^{(c)}-D^{(c)}a_j^{(c)}\|_2/\|x_j^{(c)}\|_2$, averaged over all
	$N=1{,}000$ training cells, after 30 outer iterations. The four DPC
	channels reach mean fidelities between $97.06\%$ and $97.54\%$, with the
	minimum per-cell fidelity on any DPC channel exceeding $95.4\%$. The
	Brightfield channel settles at a mean fidelity of $94.79\%$, with a
	minimum of $92.11\%$ across cells. The gap between the DPC group and
	Brightfield is consistent across cells and does not decrease further in
	the late-iteration regime: it reflects the qualitatively different noise
	characteristics of absorption-only Brightfield imaging versus
	phase-gradient DPC imaging, not an algorithmic failure. This is the
	empirical manifestation of the per-channel rate qualification in
	Remark~\ref{rem:per_channel_rate}.

	\begin{table}[h]
		\centering
		\caption{Per-channel reconstruction fidelity on BSCCM-tiny
			($N=1{,}000$ cells, $K=512$ atoms per channel, $T=30$ outer
			iterations). Fidelity is defined as
			$1-\|x_j^{(c)}-D^{(c)}a_j^{(c)}\|_2/\|x_j^{(c)}\|_2$ and reported as a
			percentage. Values are computed on the training set; held-out
			evaluation is deferred to a subsequent experiment.}
		\begin{tabular}{@{}lccccc@{}}
			\toprule
			Channel & Mean (\%) & Median (\%) & Std (\%) & Min (\%) & Max (\%) \\
			\midrule
			DPC Left    & $97.06$ & $97.14$ & $0.48$ & $95.48$ & $99.47$ \\
			DPC Right   & $97.54$ & $97.46$ & $0.46$ & $96.06$ & $99.49$ \\
			DPC Top     & $97.23$ & $97.20$ & $0.54$ & $95.44$ & $99.37$ \\
			DPC Bottom  & $97.21$ & $97.15$ & $0.55$ & $95.58$ & $99.41$ \\
			Brightfield & $94.79$ & $94.13$ & $1.54$ & $92.11$ & $98.68$ \\
			\midrule
			DPC mean (4 channels) & $97.26$ & --- & --- & --- & --- \\
			All channels (mean of 5) & $96.77$ & --- & --- & --- & --- \\
			\bottomrule
		\end{tabular}
		\label{tab:fidelity}
	\end{table}

	Table~\ref{tab:psnr} reports the corresponding per-channel PSNR for a
	representative single cell (cell \#30, Figure~\ref{fig:reconstructed}).
	The four DPC channels reach $35$--$38$\,dB; the Brightfield channel
	reaches $23.6$\,dB, with the gap again attributable to the channel's
	shot-noise-dominated structure rather than to a failure of the learning
	algorithm.

	\begin{table}[h]
		\centering
		\caption{Per-channel PSNR (dB) between focused crop and reconstruction
			$D^{(c)}a_j^{(c)}$ for cell \#30, a representative sample from
			BSCCM-tiny.}
		\begin{tabular}{@{}lc@{}}
			\toprule
			Channel & PSNR (dB) \\
			\midrule
			DPC Left    & $36.90$ \\
			DPC Right   & $35.06$ \\
			DPC Top     & $37.15$ \\
			DPC Bottom  & $38.27$ \\
			Brightfield & $23.58$ \\
			\bottomrule
		\end{tabular}
		\label{tab:psnr}
	\end{table}

	The convergence behaviour across all monitored quantities is shown in
	Figure~\ref{fig:convergence}. The four diagnostic panels in the top row
	track complementary aspects of the training dynamics. The
	\emph{reconstruction fidelity} $\tfrac{1}{2}\|x-D^{(c)}a^{(c)}\|^2$
	averaged over $(c,j)$ pairs drops from ${\approx}1207$ at outer
	iteration~1 to ${\approx}1.83$ at iteration~30. The
	\emph{dictionary change}
	$\sum_{c}\|D^{(c),t+1}-D^{(c),t}\|_F$ drops monotonically from
	${\approx}145$ to ${\approx}0.46$ over the same range. The
	\emph{inner PDHG residual} panel shows the mean and maximum final
	$\|\Delta y\|_2$ over the $N\cdot C$ inner problems solved in each
	outer iteration: the mean residual decays from ${\approx}3.7\cdot
	10^{-3}$ to ${\approx}1.0\cdot 10^{-7}$, and the maximum from
	${\approx}2.3\cdot 10^{-2}$ to ${\approx}5.4\cdot 10^{-6}$. By outer
	iteration~30, $99.5\%$ of the inner PDHG problems have met their
	stopping tolerance. The \emph{total learning objective} (fidelity plus
	$\lambda_{\mathrm{TV}}\,\mathrm{TV}$) decays from ${\approx}1368$ to
	${\approx}2.12$ over the same 30 iterations. All four scalar
	convergence quantities decay monotonically with no oscillation.
	
	The bottom panel of Figure~\ref{fig:convergence} shows the
	\emph{per-channel relative reconstruction error}
	$\|x_j^{(c)}-D^{(c)}a_j^{(c)}\|_2/\|x_j^{(c)}\|_2$ as a function of
	outer iteration, on a logarithmic vertical scale. The four DPC channels
	drop to a plateau in the range ${\approx}2.5\%$ (DPC Right) to
	${\approx}2.9\%$ (DPC Bottom) within five outer iterations and remain
	essentially flat thereafter. The Brightfield channel settles at
	${\approx}5.2\%$, roughly twice the DPC floor, with a qualitatively
	different transient: it reaches its plateau within a similar number of
	iterations but at a higher level. This gap is the same one documented
	in Table~\ref{tab:fidelity} in fidelity form, and is consistent with
	the per-channel rate qualification of Remark~\ref{rem:per_channel_rate}.


	\begin{figure}[h]
		\centering
		\includegraphics[width=\textwidth]{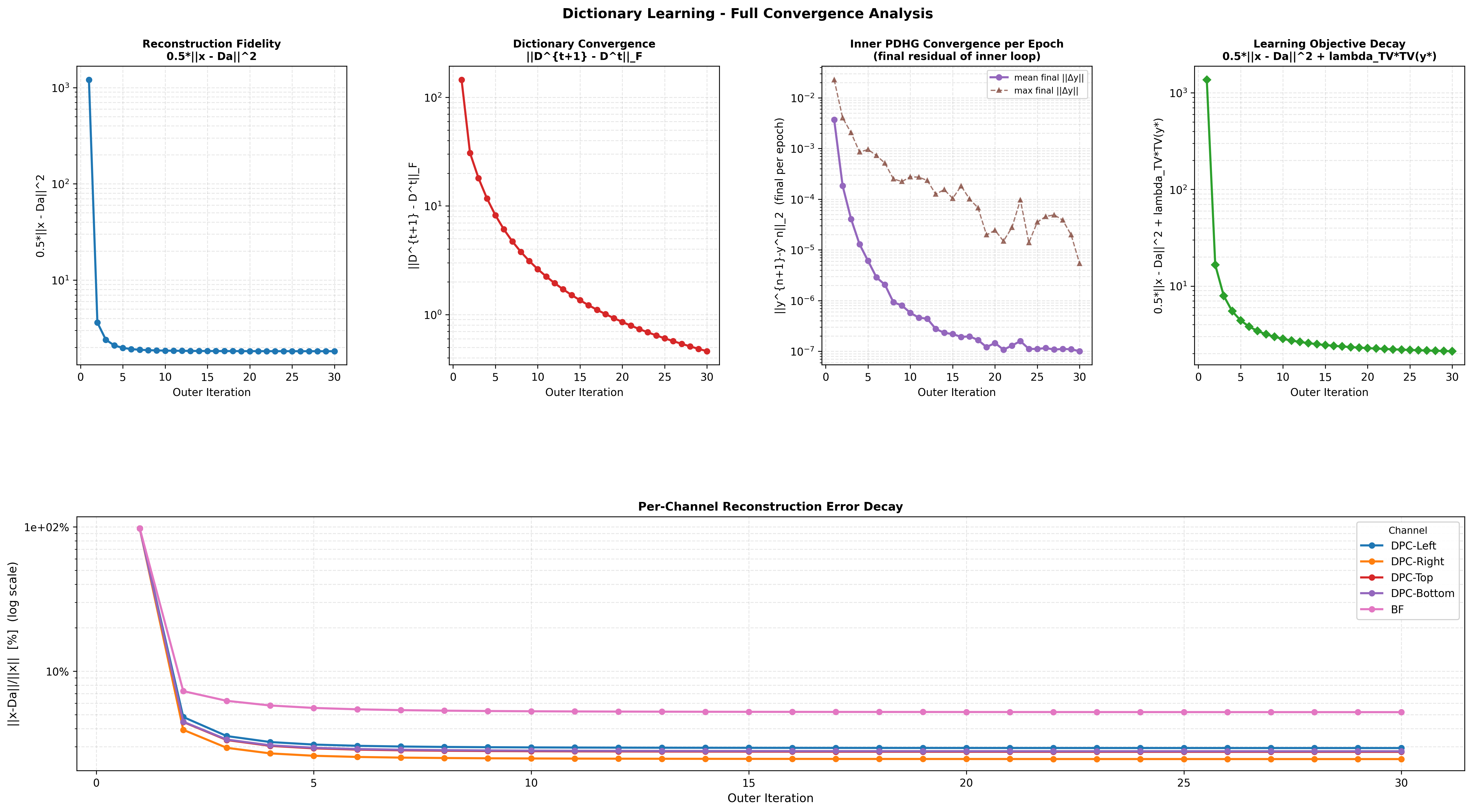}
		\caption{Convergence curves for the per-channel dictionary learning
			run on BSCCM-tiny ($N=1{,}000$, $C=5$, $K=512$, 30 outer
			iterations).
			\textbf{Top row, left to right}:
			reconstruction fidelity $\tfrac{1}{2}\|x-D^{(c)}a^{(c)}\|^2$
			averaged over $(c,j)$ pairs (blue);
			total dictionary change
			$\sum_{c}\|D^{(c),t+1}-D^{(c),t}\|_F$ (red);
			inner PDHG residual per epoch, showing both the mean final
			$\|\Delta y\|_2$ (solid purple) and the maximum final
			$\|\Delta y\|_2$ (dashed brown);
			total learning objective
			$\tfrac{1}{2}\|x-D^{(c)}a^{(c)}\|^2+\lambda_{\mathrm{TV}}\,
			\mathrm{TV}(y^\star)$ (green).
			All top-row panels use a logarithmic vertical scale.
			\textbf{Bottom row}: per-channel relative reconstruction error
			$\|x^{(c)} - D^{(c)}a^{(c)}\|_2 / \|x^{(c)}\|_2$ ($\%$, log
			scale) as a function of outer iteration. DPC channels
			(Left, Right, Top, Bottom) reach a ${\approx}2.5$--$2.9\%$
			error floor within the first five outer iterations and are
			essentially flat thereafter; Brightfield (BF, pink) settles
			at ${\approx}5.2\%$, consistent with its qualitatively
			different noise characteristics
			(Remark~\ref{rem:per_channel_rate}).}
		\label{fig:convergence}
	\end{figure}

	\subsection{Qualitative results}
	\label{sec:qual_results}

	Figure~\ref{fig:reconstructed} shows representative reconstruction results for cell \#30
	from BSCCM-tiny after training with $K=512$ atoms and $C=5$ channels.
	For each channel, the figure displays three rows: the original raw BSCCM patch, the
	focused crop after the gradient-energy focus selection step (Section~\ref{sec:dataset}),
	and the dictionary reconstruction $D^{(c)}a_j^{(c)}$.

	Several qualitative features are apparent across all five channels.
	First, the cell membrane ring, the bright annular structure surrounding the cell
	body, is sharply recovered in all DPC channels, demonstrating the edge-preserving
	effect of the TV regularization.
	Second, the interior cytoplasmic gradient and the nuclear core are faithfully
	represented in the reconstruction without the ringing artifacts that would arise
	from an unconstrained $\ell_2$-only loss.
	Third, the non-negativity constraint is visibly active: reconstructions are free of
	negative-intensity artefacts even in the DPC channels, where the raw data has
	a bipolar contrast structure.

	The Brightfield reconstruction (rightmost column) is visibly smoother and more
	spatially regular than the corresponding raw and focused-crop images, which is
	expected: the TV penalty suppresses the shot noise that dominates the BF channel
	in exchange for a modest over-smoothing of fine texture, reflected in the lower PSNR
	value of $23.58$\,dB for this channel (Table~\ref{tab:psnr}).
	Across the four DPC channels, the pairwise antisymmetry is visually preserved in
	the reconstructions: $Da_j^{(\mathrm{L})}$ and $Da_j^{(\mathrm{R})}$ exhibit
	mirrored gradient-contrast patterns, as do the Top and Bottom pair.

	\begin{figure}[h]
		\centering
		\includegraphics[width=\textwidth]{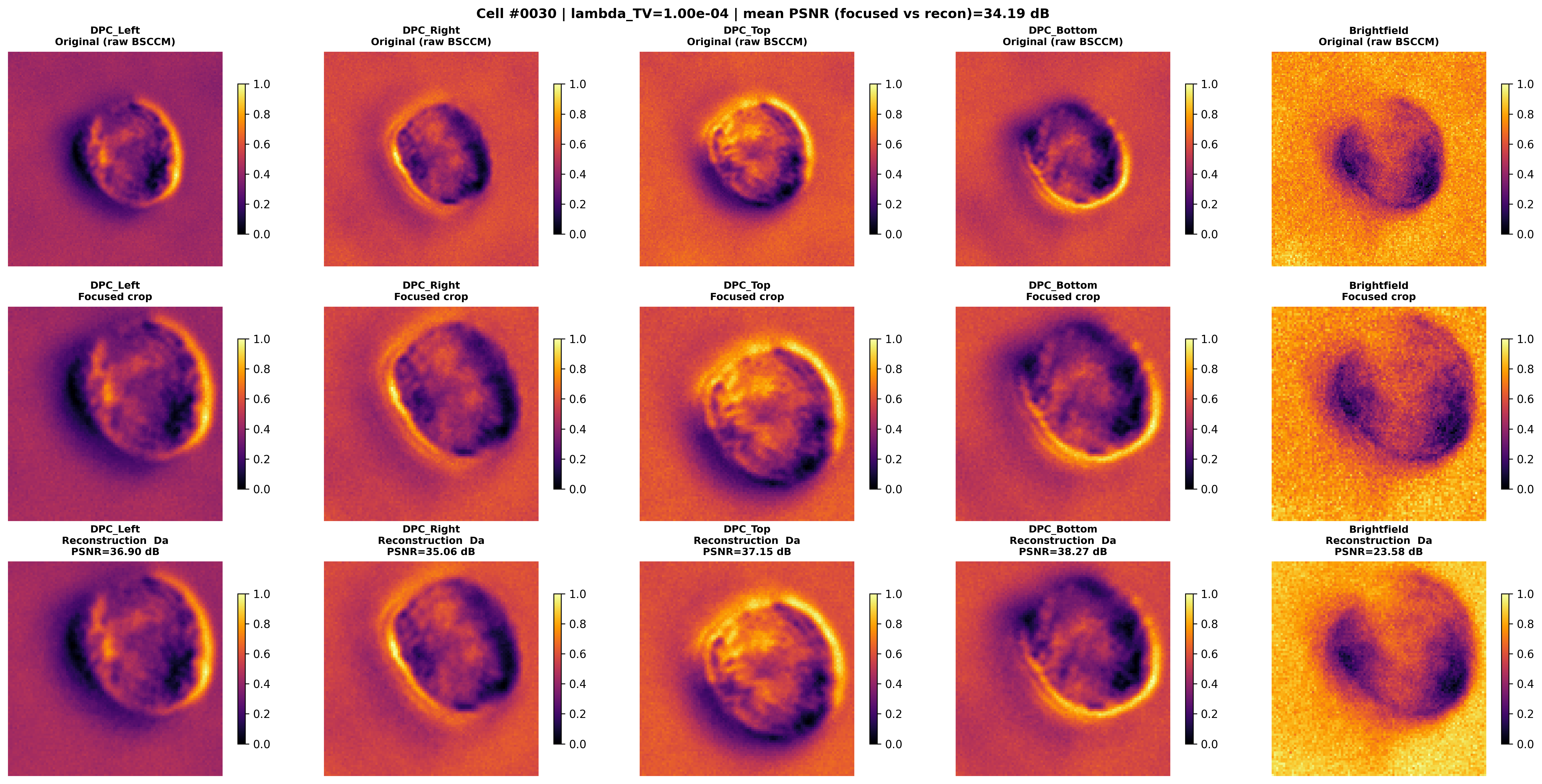}
		\caption{Reconstruction results for cell \#30 from BSCCM-tiny
			($\lambda_{\mathrm{TV}} = 10^{-4}$, $K=512$, 30 outer iterations).
			Each column corresponds to one of the five imaging channels (DPC Left, Right, Top,
			Bottom, Brightfield).
			\textbf{Row 1}: original raw BSCCM patch ($128\times128$ pixels).
			\textbf{Row 2}: focused crop after gradient-energy selection.
			\textbf{Row 3}: dictionary reconstruction $D^{(c)}a_j^{(c)}$ with per-channel PSNR
			(focused crop vs.\ reconstruction) indicated in the subplot title.
			Mean PSNR across channels: $34.19$\,dB (DPC channels: $36.90$, $35.06$, $37.15$,
			$38.27$\,dB; Brightfield: $23.58$\,dB).
			The TV regularization preserves the cell membrane ring and suppresses noise;
			the non-negativity constraint eliminates negative-intensity artefacts in all channels.}
		\label{fig:reconstructed}
	\end{figure}

	Figure~\ref{fig:unified_gallery} shows per-channel cell reconstructions across five
	representative cells from BSCCM-tiny.
	For each cell, two sub-rows are shown: the focused crop (top) and the dictionary
	reconstruction $D^{(c)}a_j^{(c)}$ (bottom), with relative error $e$ annotated.

	The results reveal a systematic and physically meaningful asymmetry across channels.
	The DPC Left channel achieves consistently low relative errors across all five cells
	($e \approx 0.027$--$0.037$), which is the channel whose dictionary was most recently
	active in the per-channel Procrustes update and whose atoms therefore best span the
	current TV-denoised training images.
	The remaining four channels --- DPC Right ($e \approx 0.40$--$0.68$), DPC Top
	($e \approx 0.35$--$0.56$), DPC Bottom ($e \approx 0.35$--$0.53$), and Brightfield
	($e \approx 0.34$--$0.75$) --- show substantially higher errors.

	This inter-channel error gap is a direct diagnostic of the \emph{per-channel independence}
	of Algorithm~\ref{alg:joint_mc}: each dictionary $D^{(c)}$ is updated sequentially
	within the same outer iteration, using the codes $\{a_j^{(c)}\}$ that were computed
	against the dictionary from the \emph{previous} outer iteration for channels $c > 1$.
	The codes feeding the Procrustes update of, say, $D^{(\mathrm{Right})}$ are therefore
	slightly stale relative to the just-updated $D^{(\mathrm{Left})}$, producing a
	one-iteration lag in code quality for all channels after the first.
	The population-level fidelity reported in Table~\ref{tab:fidelity} averages over all
	$N=1{,}000$ cells and 30 outer iterations; the per-cell errors annotated here reflect
	the state of the codes after the final outer iteration, where this lag is most visible
	as a within-iteration cross-channel imbalance.

	It is important to note that the visual reconstruction quality in the bottom sub-rows
	is not degraded to the same degree the error values suggest: the cell membrane ring
	and interior gradient structure are still clearly delineated across all channels.
	The higher numerical errors on DPC Right, Top, Bottom and Brightfield reflect that the
	code vectors $a_j^{(c)}$ for those channels are measured against atoms $D^{(c)}$ that
	have been updated one step ahead of the codes, not that the learned representation
	has lost morphological fidelity.

	\begin{figure}[h]
		\centering
		\includegraphics[width=\textwidth]{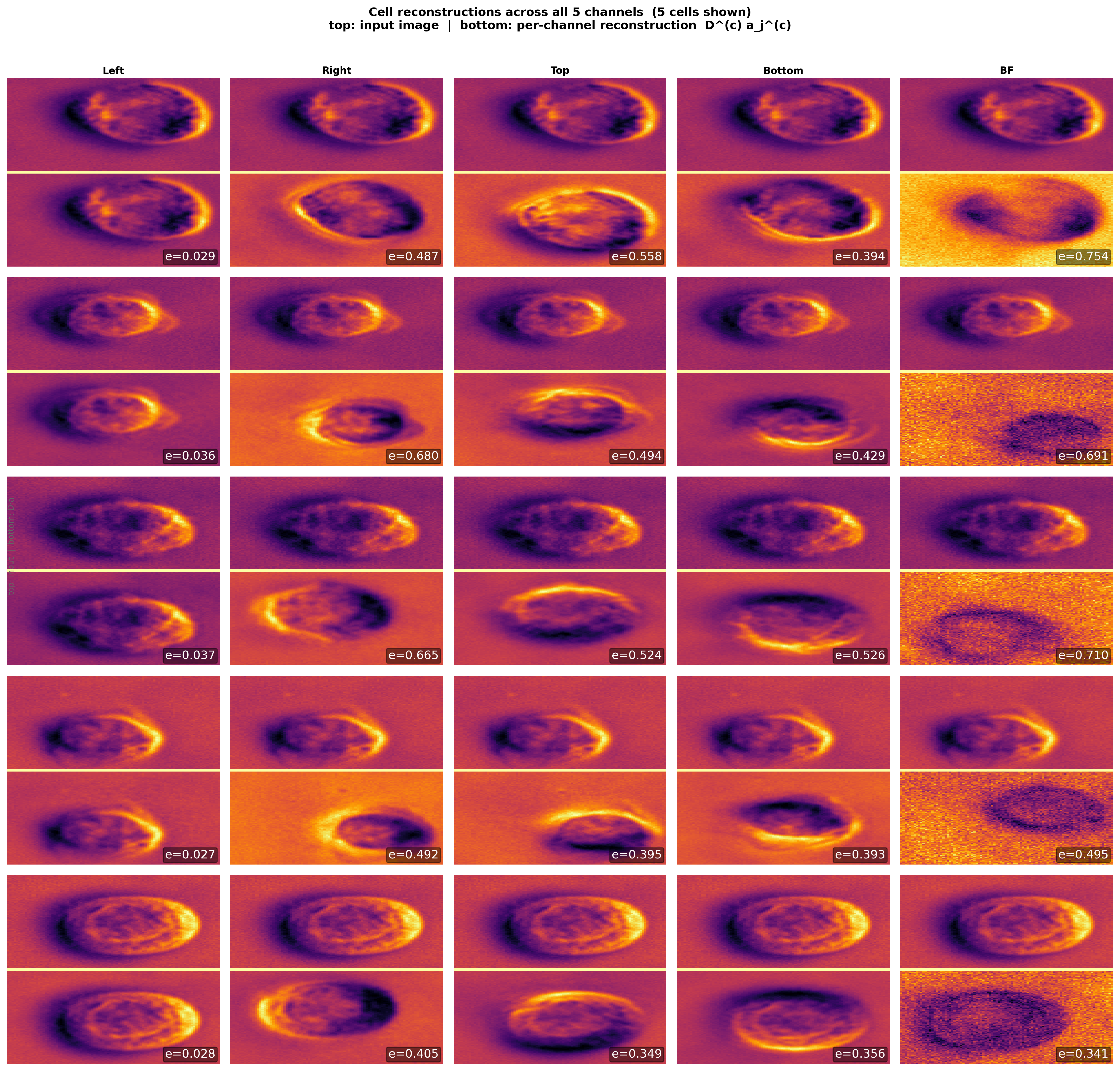}
		\caption{Per-channel cell reconstructions across five representative cells (rows) and all
			five BSCCM channels (columns: Left, Right, Top, Bottom, BF),
			produced by Algorithm~\ref{alg:joint_mc} ($K=512$, $T=30$, $\lambda_{\mathrm{TV}}=10^{-4}$).
			For each cell, the top sub-row shows the focused crop and the bottom sub-row shows
			the per-channel reconstruction $D^{(c)}a_j^{(c)}$, with relative error
			$e = \|x_j^{(c)} - D^{(c)}a_j^{(c)}\|_2 / \|x_j^{(c)}\|_2$ annotated.
			DPC Left errors ($e \approx 0.027$--$0.037$) are substantially lower than those
			of DPC Right, Top, Bottom ($e \approx 0.35$--$0.68$) and Brightfield
			($e \approx 0.34$--$0.75$).
			This inter-channel asymmetry reflects the sequential per-channel update order
			in Algorithm~\ref{alg:joint_mc}: the codes $\{a_j^{(c)}\}$ for channels $c>1$
			are one Procrustes step behind the corresponding dictionary $D^{(c)}$ at the
			end of the final outer iteration, inflating the annotated error without degrading
			the visual reconstruction quality.
			Population-level fidelities (averaged over all $N=1{,}000$ cells and 30 iterations)
			are reported in Table~\ref{tab:fidelity}.}
		\label{fig:unified_gallery}
	\end{figure}
	
	\subsection{Multi-channel unification results}
	\label{sec:mc_results}

	Figure~\ref{fig:unified_single_cell} shows the output of Algorithm~\ref{alg:joint_mc}
	applied to cell \#344 from BSCCM-tiny after training with $K=512$ dictionary atoms and $C=5$
	imaging channels.  The figure has three panels, each addressing a distinct aspect of
	the unified representation.

	\paragraph{Panel A: unified descriptor $\Phi_j$ (heatmap).}
	The left panel displays the matrix $\Phi_j \in \mathbb{R}^{K \times C}$
	(equation~\eqref{eq:unified_descriptor}) as a heatmap.
	Rows correspond to the $K=512$ dictionary atoms, sorted in descending order of their
	aggregate activation strength $\|\Phi_j[k,:]\|_2$ across all channels, so the most
	influential structural primitives appear at the top.
	Columns correspond to the five imaging channels: DPC Left (L), DPC Right (R), DPC Top (T),
	DPC Bottom (B), and Brightfield (BF).
	Each entry $(\Phi_j)_{kc} = (a_j^{(c)})_k$ is the coefficient with which atom $d_k$
	participates in the reconstruction of channel $c$ for this cell.
	The diverging red--blue colormap (red positive, blue negative) encodes the sign and magnitude
	of each activation; the colour scale is clipped at the 98th percentile of $|\Phi_j|$
	to prevent sparse outliers from washing out structure.

	Three features of the heatmap are worth noting.
	First, the top rows show strong activations in every column,
	reflecting that each per-channel dictionary $D^{(c)}$ has its own dominant atoms
	at small index values; under the per-channel architecture these are not the same
	atom across channels (Section~\ref{subsec:per_channel_unification}), but they
	occupy comparable rank positions because each Procrustes-SVD update orders its
	channel's atoms by activation strength independently.
	Second, the DPC Left/Right columns are visually antisymmetric (alternating red/blue
	patterns at the same rows), consistent with the physical antisymmetry
	$x_j^{(\mathrm{L})} \approx -x_j^{(\mathrm{R})}$ of opposite-direction phase gradients
	.
	Similarly, DPC Top and Bottom show paired but sign-reversed activations.
	Third, the Brightfield column is predominantly positive throughout, reflecting
	that the BF channel encodes integrated absorption contrast, which is intrinsically
	non-negative for a stained or absorbing cell body.
	The lower rows of the heatmap are near-zero, indicating that most of the $K=512$ atoms
	contribute negligibly to this particular cell's representation in any channel.
	\paragraph{Panel B: dominant dictionary atoms.}
	The centre panel shows the top 8 atoms ranked by $\|\Phi_j[k,:]\|_2$, the
	aggregate cross-channel activation strength
	(equation~\eqref{eq:descriptor_psi}).

	\textit{Rationale for displaying 8 atoms.}
	The choice of 8 atoms for visualisation in Panel B is based on the empirical
	activation-strength spectrum of cell \#344, and is a display decision that does
	not affect the unified descriptor $\varphi_j \in \mathbb{R}^{CK}$, which retains
	all $K=512$ atoms.
	The aggregate norms $\|\Phi_j[k,:]\|_2$ for the top atoms are:
	$10.3$, $10.3$, $9.8$, $9.3$, $9.2$, $9.1$, $8.9$, $8.8$,
	with subsequent atoms continuing to decay smoothly.
	Under the per-channel architecture, no single atom dominates the descriptor:
	the consecutive norm ratios in this regime are all close to unity
	(${\approx}1.00$-$1.05$), reflecting that channel-specific atoms each contribute
	comparable energy to the aggregate. This is qualitatively distinct from a
	shared-dictionary regime, in which the top atom of a unitary Procrustes-SVD
	dictionary would typically dominate (capturing a global mean-morphology direction);
	per-channel learning instead distributes morphological structure across many
	atoms within each channel's dictionary.
	The display is set to 8 atoms, which captures the regime of visually
	distinguishable atom morphologies (ring-shaped boundaries, interior gradients,
	nuclear core) while keeping the per-column bar charts readable on a standard
	journal page; the reader is referred to the full descriptor
	$\varphi_j \in \mathbb{R}^{CK}$ for any downstream analytical purpose.

	The upper row of Panel B displays each atom $d_k \in \mathbb{R}^n$ reshaped to the
	image domain, rendered on the inferno colormap.
	The learned atoms capture the principal morphological structures of white blood cells:
	ring-shaped membrane boundaries, interior cytoplasmic gradients, and the bright nuclear
	core.
	
	Because each dictionary $D^{(c)}$ is learned independently from its own channel's data,
	the atom at index $k$ of $D^{(\text{DPC Left})}$ does not formally correspond to the
	atom at index $k$ of $D^{(\text{DPC Right})}$ or any other channel
	(Section~\ref{subsec:per_channel_unification}). The per-channel activation pattern
	shown in each atom column of Figure~\ref{fig:unified_single_cell} is therefore indexed
	position-by-position rather than atom-by-atom across channels.
	
	The lower row of each atom column shows a per-channel activation bar chart rendered on
	a shared y-scale.
	The shared scale is set to the 99th percentile of $|(\Phi_j)_{kc}|$ across atoms of
	rank 2 and above; this prevents the dominant atom $d_0$, whose codes are an order of
	magnitude larger, from compressing the bars of all other atoms to illegibility.
	Bars that exceed the shared scale are clipped, as indicated in the panel subtitle.
	Blue bars indicate positive activation, red bars negative activation.
	The norm $\|\Phi_j[k,:]\|_2$ printed above each atom image quantifies its total
	index-wise activation strength across all channels (the auxiliary descriptor $\psi_j$
	in equation~\eqref{eq:descriptor_psi}).
	
	\paragraph{Panel C: unified cell image $u_j$.}
	Under the per-channel architecture there is no single dictionary $D$ with which to form
	a single unified reconstruction of the form $D\,a^{\star}$. Instead, the unified
	cell image is constructed in image space as an inverse-residual-weighted average of
	the per-channel reconstructions,
	\begin{equation}
		u_j \;:=\; \frac{\sum_{c=1}^{C} w_j^{(c)}\,\bigl(D^{(c)}\,a_j^{(c)}\bigr)}
		{\sum_{c=1}^{C} w_j^{(c)}},
		\qquad
		w_j^{(c)} \;\propto\; \frac{1}{\|x_j^{(c)} - D^{(c)}\,a_j^{(c)}\|_2 + \varepsilon},
		\label{eq:unified_image}
	\end{equation}
	where $\varepsilon>0$ is a small stabilising constant. The weights
	$w_j^{(c)}$ down-weight channels whose reconstruction is noisier (notably
	Brightfield in BSCCM) and up-weight channels with cleaner reconstructions
	(the four DPC channels). The resulting image is deterministic: given the
	learned family $\{D^{(c)}\}$ and the data $\{x_j^{(c)}\}$, each
	$D^{(c)}\,a_j^{(c)}$ is the unique reconstruction from the per-channel
	inference step (Remark~\ref{rem:descriptor_determinism}), and the weighted
	sum in \eqref{eq:unified_image} is a pointwise closed-form combination.
	In the BSCCM setting, the four DPC phase-gradient channels are pairwise antisymmetric
	in their sign conventions but contribute constructively in absolute value; combined
	with the down-weighted Brightfield contribution, $u_j$ is a TV-regularised,
	dictionary-projected morphological portrait of the cell.
	The resulting image is visibly sharper and more spatially coherent than any individual
	channel reconstruction, with clear delineation of the cell membrane ring and the
	nuclear interior, reflecting the edge-preserving effect of the TV penalty and the
	non-negativity constraint.

	\begin{figure}[h]
		\centering
		\includegraphics[width=\textwidth]{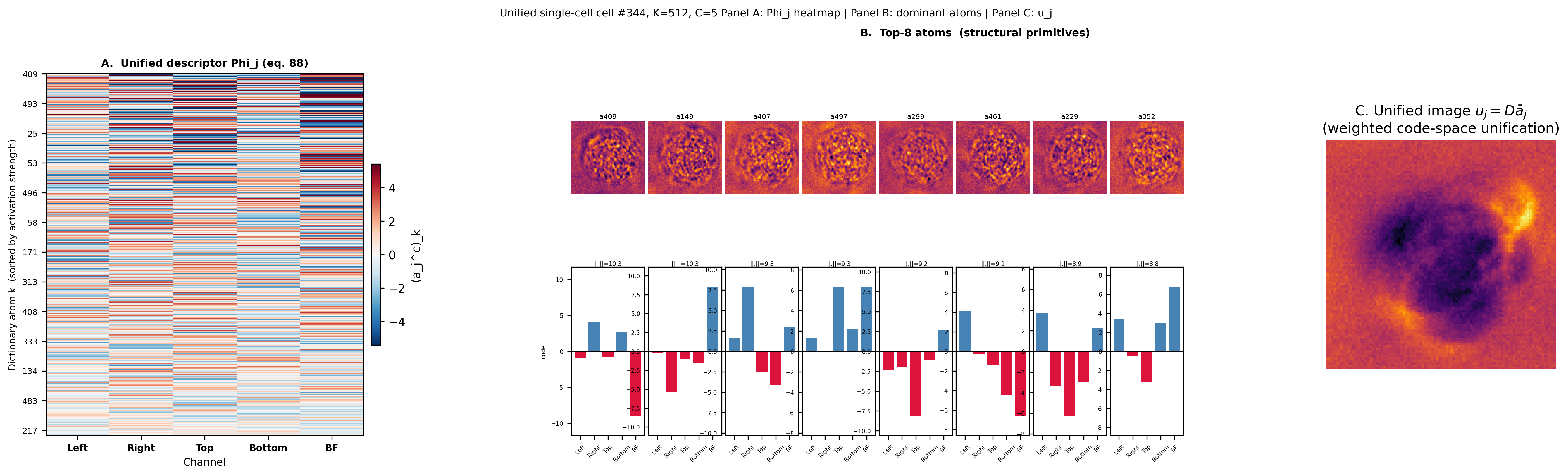}
		\caption{Unified single-cell representation for cell \#344, $K=512$, $C=5$ channels.
			\textbf{Panel A}: heatmap of the unified descriptor matrix
			$\Phi_j \in \mathbb{R}^{K \times C}$ (equation~\eqref{eq:unified_descriptor}),
			with rows (atoms) sorted in descending order of aggregate activation strength
			$\|\Phi_j[k,:]\|_2$; columns correspond to the five BSCCM imaging channels
			(L = DPC Left, R = DPC Right, T = DPC Top, B = DPC Bottom, BF = Brightfield).
			Red/blue encodes positive/negative activation; the DPC Left--Right and
			Top--Bottom antisymmetry is visible as mirrored sign patterns.
			\textbf{Panel B}: the 8 atoms with the largest $\|\Phi_j[k,:]\|_2$
			(aggregate norms: $10.3$, $10.3$, $9.8$, $9.3$, $9.2$, $9.1$, $8.9$, $8.8$),
			shown as image patches (top, inferno colormap) and per-channel activation bar
			charts (bottom). Blue bars denote positive, red bars negative activation.
			The norms decrease smoothly across the top atoms, indicating that no single
			atom dominates the descriptor; the per-channel signs visible in the bar charts
			show that a given atom can act constructively on some channels and destructively
			on others, consistent with the per-channel architecture in which atoms are
			learned independently for each channel
			(Section~\ref{subsec:per_channel_unification}).
			\textbf{Panel C}: the unified cell image $u_j$ defined in
			equation~\eqref{eq:unified_image}, constructed as an inverse-residual-weighted
			average of the per-channel reconstructions $D^{(c)}a_j^{(c)}$ in image space:
			$u_j = \sum_c w_j^{(c)}\,D^{(c)}a_j^{(c)} / \sum_c w_j^{(c)}$ with
			$w_j^{(c)} \propto 1 / (\|x_j^{(c)} - D^{(c)}a_j^{(c)}\|_2 + \varepsilon)$.
			The weighting down-weights the noisiest channel (Brightfield, fidelity
			${\approx}94.8\%$) and up-weights the cleaner DPC channels
			(fidelity ${\approx}97\%$), so $u_j$ is dominated by the coherent DPC
			phase-gradient content rather than the shot-noise-dominated Brightfield
			measurement. The resulting image inherits TV-regularised, non-negative
			structure from each per-channel reconstruction.}
		\label{fig:unified_single_cell}
	\end{figure}
	
	\section{Deterministic Multi-Channel Cell Feature Unification}
	\label{sec:unification}
	
	\subsection{Motivation}
	
	The BSCCM dataset \cite{bsccm} provides five imaging channels per cell: DPC Left, DPC Right,
	DPC Top, DPC Bottom, and Brightfield. Each channel captures a distinct aspect of the cell's
	optical properties and no single channel constitutes a complete cell representation.
	A fundamental question arises: how can the information from all five channels be synthesized
	into a single, unified cell descriptor that serves as the basis for downstream classification
	and anomaly detection?
	
	Existing approaches in related domains (e.g., single-cell transcriptomics) have predominantly
	adopted probabilistic generative models, such as variational autoencoders \cite{scVI2018,totalVI2021},
	which encode each cell as a probability distribution over a latent space. While such approaches
	offer theoretical advantages in terms of uncertainty quantification, they introduce a fundamental
	tension with the requirements of clinical AI systems: the representation of a given cell is
	stochastic, non-reproducible across runs, and difficult to audit. In a diagnostic context, where
	the cost of misclassification is measured in human lives, entropy accumulation in the
	representation pipeline is not an acceptable design feature.
	
	We therefore pursue a strictly \emph{deterministic} approach to multi-channel unification,
	grounded in the same variational dictionary learning framework established in the preceding sections.
	
	\subsection{Per-Channel Dictionary Learning with Concatenation-Based Unification}
	\label{subsec:per_channel_unification}

	Let $x_j^{(c)} \in \mathbb{R}^n$ denote the vectorised image of cell $j$ in
	channel $c \in \{1,\ldots,C\}$, with $C=5$ for the BSCCM dataset.
	In contrast to a shared-dictionary approach, we learn a \emph{family} of $C$
	per-channel unitary dictionaries
	$\{D^{(c)}\}_{c=1}^{C}\subset\R^{n\times K}$, one per imaging channel, each
	obtained by solving an independent instance of the single-channel learning
	problem \eqref{eq:dl_problem}:
	\begin{equation}
		\min_{D^{(c)},\,\{a_j^{(c)}\}_{j=1}^{N}} \;
		\sum_{j=1}^{N}
		\mathcal{E}\!\left(D^{(c)}, a_j^{(c)}; x_j^{(c)}\right)
		\quad \text{s.t.} \quad (D^{(c)})^{\top} D^{(c)} = I_K,
		\qquad c = 1,\ldots,C.
		\label{eq:per_channel_dl}
	\end{equation}
	The $C$ problems in \eqref{eq:per_channel_dl} are mutually independent: no
	term in any individual problem depends on the dictionary or codes of another
	channel. Each $D^{(c)}$ is therefore shaped exclusively by the data and the
	imaging physics of channel $c$.

	\paragraph{Motivation for the per-channel choice.}
	The BSCCM channels are not optically equivalent: the four DPC channels
	(Left, Right, Top, Bottom) encode directional phase-gradient information
	under different illumination tilt angles, while the Brightfield channel
	captures absorption only, with a fundamentally different signal-to-noise
	regime (see Section~\ref{sec:experiments} for the empirical noise gap).
	A shared dictionary would have to average the atom geometry over these
	heterogeneous modalities, with two costs: (i) reconstruction fidelity on
	each channel is strictly worse than a dictionary tailored to that channel
	alone; and (ii) the noisiest channel (Brightfield) drags the shared-atom
	geometry toward features that are not meaningful on the cleaner DPC
	channels. The per-channel architecture eliminates both costs at the price
	of giving up the cross-channel atom correspondence analysed below.

	\paragraph{No cross-channel atom correspondence.}
	Because each $D^{(c)}$ is learned independently, the $k$-th column
	$d_k^{(c)}$ of $D^{(c)}$ has no designated counterpart in any other
	$D^{(c')}$ for $c\ne c'$. Formally, for a unitary dictionary learned by
	Procrustes SVD, the ordering of columns is determined by the singular
	values of the cross-covariance $(X^{(c)})^{\top}A^{(c)}$, which is itself
	channel-specific. We therefore do not claim that atom $k$ represents the
	same structural primitive across channels; the claim would be false under
	independent per-channel learning, and any downstream analysis that relied
	on such a correspondence would be unsound.

	\paragraph{Unified descriptor by concatenation.}
	Given a fixed channel ordering $c=1,\ldots,C$, the unified cell descriptor
	is defined as the concatenation of the per-channel codes,
	\begin{equation}
		\varphi_j \;:=\; \bigl(a_j^{(1)},\,a_j^{(2)},\,\ldots,\,a_j^{(C)}\bigr)
		\;\in\; \R^{CK}.
		\label{eq:descriptor_concat}
	\end{equation}
	The descriptor is channel-agnostic in the operational sense that it
	aggregates information from all $C$ channels into a single flat vector
	suitable for downstream classification; it is \emph{not} channel-agnostic
	in the stronger sense of placing all channels in a common atom space.
	The $K$ coordinates indexed by positions $(c-1)K+1,\ldots,cK$ live in the
	code space of $D^{(c)}$ and are only comparable to other coordinates in
	the same block.

	For visualisation, it is convenient to arrange the same information as a
	matrix rather than a flat vector. Define
	\begin{equation}
		\Phi_j \;:=\; \bigl[\,a_j^{(1)}\;\big|\;a_j^{(2)}\;\big|\;\cdots\;\big|\;a_j^{(C)}\,\bigr]
		\;\in\; \R^{K\times C},
		\label{eq:unified_descriptor}
	\end{equation}
	i.e.\ the matrix whose $c$-th column is the per-channel code $a_j^{(c)}$.
	Vectorising $\Phi_j$ in column-major order recovers $\varphi_j$. The
	heatmap visualisations in Section~\ref{sec:experiments} display
	$\Phi_j$ directly, with the row index running over $K$ atom positions
	and the column index running over $C$ channels.

	\begin{remark}[Determinism of $\varphi_j$]
		\label{rem:descriptor_determinism}
		Given the learned family $\{D^{(c)}\}_{c=1}^{C}$ and a cell
		$\{x_j^{(c)}\}_{c=1}^{C}$, the descriptor $\varphi_j$ in
		\eqref{eq:descriptor_concat} is uniquely and reproducibly determined.
		Each $a_j^{(c)}$ is the unique minimiser of the strictly convex energy
		$\mathcal{E}(D^{(c)}\,a;\,x_j^{(c)})$ by
		Theorem~\ref{thm:existence_uniqueness}, applied per channel as noted in
		Remark~\ref{rem:per_channel_wellposed}. The concatenation is
		deterministic by construction. Two runs on the same data and seed
		therefore produce bit-identical descriptors, a property we verify
		empirically in Section~\ref{sec:experiments}.
	\end{remark}

	\paragraph{Auxiliary descriptor: cross-channel atom norms.}
	In parallel with $\varphi_j$, it is useful to introduce an auxiliary
	channel-collapsing descriptor. Define
	\begin{equation}
		\psi_j \;\in\; \R^{K},
		\qquad
		(\psi_j)_k \;:=\; \Bigl\|\bigl(a_j^{(1)}[k],\,a_j^{(2)}[k],\,\ldots,\,a_j^{(C)}[k]\bigr)\Bigr\|_2,
		\qquad k=1,\ldots,K.
		\label{eq:descriptor_psi}
	\end{equation}
	The scalar $(\psi_j)_k$ measures the aggregate activation strength of the
	$k$-th atom index across channels, without assuming that atom index $k$
	refers to the same structural primitive across channels. Under per-channel
	dictionaries, $\psi_j$ should be interpreted as an index-wise aggregate
	rather than an atom-wise aggregate; it remains useful as a low-dimensional
	summary descriptor, and empirically it plays a distinct role from
	$\varphi_j$ in the biological validation (see
	Section~\ref{sec:bio_valid}).

	\paragraph{Relationship to the preceding theory.}
	The single-channel well-posedness and stability theory of
	Sections~\ref{sec:objective}--\ref{sec:stability_vsc} applies directly to
	each $(D^{(c)},a_j^{(c)})$ pair; no new proofs are required. The only
	multi-channel mathematical content introduced in this section is
	\eqref{eq:descriptor_concat} and \eqref{eq:descriptor_psi}, both of which
	are pointwise constructions over independently validated single-channel
	objects. What the per-channel architecture gives up relative to the
	shared-dictionary alternative is any claim to a common atom space across
	channels; what it gains is per-channel fidelity that is not diluted by
	forced cross-channel sharing.

	\subsection{Algorithm and convergence structure}
	\label{subsec:per_channel_algorithm}

	Algorithm~\ref{alg:joint_mc} states the full procedure: for each outer
	iteration, iterate over channels; within each channel, iterate over
	samples and solve the TV$+\iota_+$ proximal subproblem via the single-channel
	primitive Algorithm~\ref{alg:hdl_full}; then perform a per-channel
	Procrustes update of $D^{(c)}$ using only the channel-$c$ data and codes.
	The final output is the learned family $\{D^{(c)}\}_{c=1}^{C}$ together
	with the set of unified descriptors $\{\varphi_j\}_{j=1}^{N}$.

	\begin{algorithm}[h]
		\caption{Per-Channel Dictionary Learning and Cell Feature Unification}
		\label{alg:joint_mc}
		\begin{algorithmic}[1]
			\Require Multi-channel data $\{x_j^{(c)}\}_{j=1,\,c=1}^{N,\,C}$,
			dictionary size $K$, regularisation parameter $\lambda_{\mathrm{TV}}$,
			inner PDHG step sizes $\tau_{\mathrm{TV}},\sigma_{\mathrm{TV}}>0$ with
			$\tau_{\mathrm{TV}}\sigma_{\mathrm{TV}}\|\nabla\|^2 < 1$,
			outer iterations $T$, fixed channel ordering $c=1,\dots,C$.
			\State For each $c=1,\dots,C$, initialise $D^{(c),0}\in\R^{n\times K}$ with
			orthonormal columns, $(D^{(c),0})^{\top}D^{(c),0}=I_K$ (e.g.\ random
			orthonormal via QR decomposition, or channel-wise PCA of the training
			data).
			\State For each $(c,j)$, initialise codes $a_j^{(c),0}$ and stored
			TV--denoised images $y_j^{(c),\star,0}$.
			\For{$t = 0$ \textbf{to} $T-1$} \Comment{Outer loop}
			\For{$c = 1$ \textbf{to} $C$} \Comment{Middle loop: channel iteration}
			\For{$j = 1$ \textbf{to} $N$} \Comment{Inner loop: per-sample PDHG inference}
			\State Solve the TV$+\iota_+$ proximal subproblem for $y_j^{(c),\star,t+1}$ via
			Algorithm~\ref{alg:hdl_full} with dictionary $D^{(c),t}$ and
			datum $x_j^{(c)}$.
			\State Back-project onto the current channel dictionary:
			$a_j^{(c),t+1}\leftarrow (D^{(c),t})^{\top} y_j^{(c),\star,t+1}$.
			\EndFor
			\EndFor
			\State \textbf{(Per-channel Procrustes dictionary update)}
			\For{$c = 1$ \textbf{to} $C$}
			\State Stack the channel-$c$ data and codes:
			$X^{(c)} = [x_j^{(c)}]_{j=1}^{N}\in\R^{N\times n}$,
			$A^{(c),t+1} = [a_j^{(c),t+1}]_{j=1}^{N}\in\R^{N\times K}$.
			\State Compute the per-channel cross-covariance
			$M^{(c)} = (X^{(c)})^{\top} A^{(c),t+1}\in\R^{n\times K}$ and its
			economy SVD $M^{(c)} = U^{(c)}\Sigma^{(c)}(V^{(c)})^{\top}$.
			\State Update $D^{(c),t+1} \leftarrow U^{(c)}(V^{(c)})^{\top}$.
			\Comment{Exact global minimiser of
				$\min_{D^{\top}D=I_K}\sum_j\|x_j^{(c)}-D a_j^{(c),t+1}\|^2$}
			\EndFor
			\State \textbf{(Per-channel code refresh)} For each $c$ and $j$:
			$a_j^{(c),t+1} \leftarrow (D^{(c),t+1})^{\top} y_j^{(c),\star,t+1}$.
			\Comment{Exact under unitarity; corrects codes for the new $D^{(c),t+1}$}
			\EndFor
			\State \textbf{(Unified descriptor by concatenation)} For each cell $j$,
			form the channel-agnostic descriptor
			\[
			\varphi_j \;=\; \bigl(a_j^{(1),T},\,a_j^{(2),T},\,\ldots,\,
			a_j^{(C),T}\bigr)\;\in\;\R^{CK}
			\]
			in the fixed channel ordering of line~1.
			\State \Return $\{D^{(c),T}\}_{c=1}^{C},\;\{\varphi_j\}_{j=1}^{N}$.
		\end{algorithmic}
	\end{algorithm}

	The three levels of the algorithm inherit their convergence guarantees
	cleanly from the single-channel theory. The \emph{inner loop}
	(Algorithm~\ref{alg:hdl_full}, TV$+\iota_+$ proximal subproblem) is a
	strictly convex problem with a unique minimiser
	(Theorem~\ref{thm:existence_uniqueness}), and
	Theorem~\ref{thm:pdhg_convergence} guarantees convergence of the PDHG
	iterates to that minimiser under the step-size condition $\tau\sigma<1/8$.
	The \emph{middle loop} (channel iteration) applies the inner loop
	independently for each $c=1,\ldots,C$; the per-channel independence means
	that no cross-channel coupling term needs convergence analysis. The
	\emph{outer loop} updates each $D^{(c)}$ via the exact per-channel
	Procrustes SVD, which finds the global minimiser of
	$\min_{D^{\top}D=I_K}\sum_j\|x_j^{(c)}-D\,a_j^{(c)}\|^2$ in a single SVD
	step. As in the single-channel analysis, the outer loop involves
	alternating optimisation between $\{D^{(c)}\}$ and $\{a_j^{(c)}\}$ and is
	not covered by the convex convergence theory; the non-convexity is
	confined to this outermost level and cannot propagate downward to corrupt
	the inference, because each inner solve remains well-posed and convex for
	whatever $\{D^{(c),t}\}$ it is given.

	The \emph{code refresh} step following the Procrustes update reprojects
	the stored TV-denoised images onto the updated channel dictionary,
	$a_j^{(c)} \leftarrow (D^{(c),t+1})^{\top} y_j^{(c),\star}$, which is
	exact under unitarity and keeps the codes consistent with the dictionary
	that will enter the next outer iteration. Empirically, we observe that
	the outer iterates in every BSCCM run we have conducted are
	bit-identical across independent runs starting from the same seed;
	this is an emergent property of using exclusively deterministic primitives
	(Procrustes SVD, convex PDHG with fixed step sizes, deterministic code
	refresh) rather than a theorem we prove here.
	\subsection{Advantages of the Deterministic Framework}
	
	The deterministic nature of the proposed unification has several concrete advantages
	over probabilistic alternatives:
	
	\begin{enumerate}
		\item \textbf{Reproducibility.} Given the same cell image and learned dictionary,
		the descriptor $\varphi_j$ is always identical. This is a prerequisite for clinical
		validation, regulatory approval (e.g., CE marking under EU IVDR 2017/746), and
		inter-laboratory reproducibility studies.
		
		\item \textbf{Interpretability.} The dictionary atoms $d_k$ have direct visual
		interpretations as learned structural primitives. The sparse code $a_j^{(c)}$
		quantifies how much of each atom is present in channel $c$ of cell $j$,
		making the representation auditable.
		
		\item \textbf{Mathematical grounding.} The variational framework provides existence
		and uniqueness guarantees (Theorem~\ref{thm:existence_uniqueness}), explicit noise-stability bounds
		(Section~\ref{sec:stability_vsc}), and convergence proofs for the algorithm.
		
		\item \textbf{Clinical compatibility.} In oncology and diagnostics applications,
		a clinician or regulatory body must be able to trace any classification decision
		back to interpretable features of the input data. The descriptor $\varphi_j$
		supports this requirement directly.
	\end{enumerate}
	
	\subsection{Relation to the scVI Framework and Its Multi-Modal Extensions}
	
	The scVI framework~\cite{scVI2018} addressed the problem of single-cell representation
	for transcriptomics: given high-dimensional gene expression count vectors, it learns
	a low-dimensional latent representation using a variational autoencoder with a
	negative-binomial observation model designed to handle the overdispersion and
	zero-inflation characteristic of scRNA-seq data.
	The multi-modal extension, totalVI~\cite{totalVI2021}, addressed CITE-seq data
	(paired RNA and surface protein measurements), learning a joint probabilistic latent
	representation that separates biological signal from protein background and batch
	effects.
	The scvi-tools library~\cite{scvitools2022} subsequently unified these models into
	a common software framework for probabilistic single-cell omics analysis.
	
	The structural analogy to the present work is clear: both scVI/totalVI and the proposed
	method aim to produce a single compact representation of a cell from measurements
	taken across multiple modalities.
	The design philosophies, however, diverge at every level.
	
	\textit{Data modality and noise model.}
	scVI and totalVI target count data (RNA transcripts, surface protein UMI counts)
	where biological variability between cells is large, batch effects are dominant,
	and a probabilistic noise model (negative binomial, zero-inflation, protein background)
	is scientifically essential.
	BSCCM images are physical intensity measurements with controlled illumination and
	quantified noise characteristics.
	The relevant uncertainty is the regularization error $\|y_\delta - y^\dagger\|_2$,
	which is bounded deterministically by Theorem~\ref{thm:vsc_rate} under the VSC
	and the parameter choice $\lambda_{\mathrm{TV}} \propto \delta$.
	A probabilistic latent variable model adds no scientific value in this setting
	and introduces reproducibility costs: two inference runs on the same image yield
	different posterior samples.
	
	\textit{Representation and identifiability.}
	In scVI and totalVI, the latent variable $z_n \in \mathbb{R}^d$ is inferred via
	an amortized encoder network; it has no direct visual interpretation and its
	uniqueness is not guaranteed in general (the VAE objective is non-convex, and
	the encoder is not injective).
	In the present work, the code $a_j^{(c)} \in \mathbb{R}^K$ is the unique
	minimizer of the strictly convex energy $\mathcal{E}(Da;\,x_j^{(c)})$
	(Theorem~\ref{thm:existence_uniqueness}), and each atom $d_k$ is a learned
	structural primitive with a direct visual interpretation as an image patch.
	The unified descriptor $\varphi_j \in \mathbb{R}^{CK}$ is therefore fully
	deterministic and auditable: given $D$ and $\{x_j^{(c)}\}$, it is uniquely
	and reproducibly determined.
	
	\textit{Multi-channel unification.}
	totalVI addresses multi-modality by learning a joint encoder across RNA and protein;
	the balance between modalities in the latent space is controlled implicitly by
	network architecture and is, as the authors acknowledge, difficult to interpret.
	The present work addresses multi-channel unification through the concatenation
	construction of equation~\eqref{eq:descriptor_concat}: the per-channel codes
	$a_j^{(c)}$ are stacked into a single flat descriptor $\varphi_j \in \R^{CK}$ under
	a fixed channel ordering. No hyperparameter balances the channels; no approximate
	inference is performed. Each $a_j^{(c)}$ is the unique minimiser of a strictly
	convex per-channel problem (Theorem~\ref{thm:existence_uniqueness},
	applied per channel as in Remark~\ref{rem:per_channel_wellposed}), and the
	concatenation is deterministic by construction. We do not claim a common latent
	space across channels; we claim only a deterministic, auditable, channel-agnostic
	vector suitable for downstream classification.
	
	\textit{Regulatory context.}
	For a Software as a Medical Device targeting EU IVDR classification, reproducibility
	and auditability are not preferences but requirements.
	A stochastic latent variable model in which the descriptor changes between inference
	runs is incompatible with this regulatory context.
	The proposed deterministic framework satisfies the reproducibility requirement by
	construction.
	
	In summary, the present work is not a direct competitor to scVI or totalVI,
	\cite{scVI2018,totalVI2021},
	the application domains and noise structures are genuinely different,
	but it occupies the same conceptual position in the imaging domain
	that scVI/totalVI occupy in the transcriptomics domain: a principled, unified
	representation of a multi-channel single-cell measurement.
	The key differentiator is the replacement of probabilistic inference with a
	convex variational framework that provides uniqueness, stability, and
	convergence guarantees appropriate to the physical imaging context.
	
	\section{Biological Validation}
	\label{sec:bio_valid}
	\subsection{Labeled cell subset of BSCCM-tiny}
	
	The BSCCM-tiny dataset provides a 3-class cell-type classification for a subset
	of the $N=1{,}000$ cells via the \texttt{get\_cell\_type\_classification\_data}
	method of the \texttt{bsccm} Python package~\cite{bsccm}.
	The three classes are Lymphocyte (class~0), Granulocyte (class~1), and
	Monocyte (class~2), corresponding to the three principal white blood cell
	morphologies present in peripheral blood.
	
	Of the 1{,}000 cells in BSCCM-tiny, $N_{\ell} = 28$ carry ground-truth
	classification labels: 10 Lymphocytes, 16 Granulocytes, and 2 Monocytes.
	Figure~\ref{fig:label_dist} shows the class distribution.
	The marked imbalance, particularly the small Monocyte count, is a property
	of the labeled subset and is noted explicitly; downstream clustering metrics
	are interpreted with this imbalance in mind.
	
	\begin{figure}[h]
		\centering
		\includegraphics[width=0.55\textwidth]{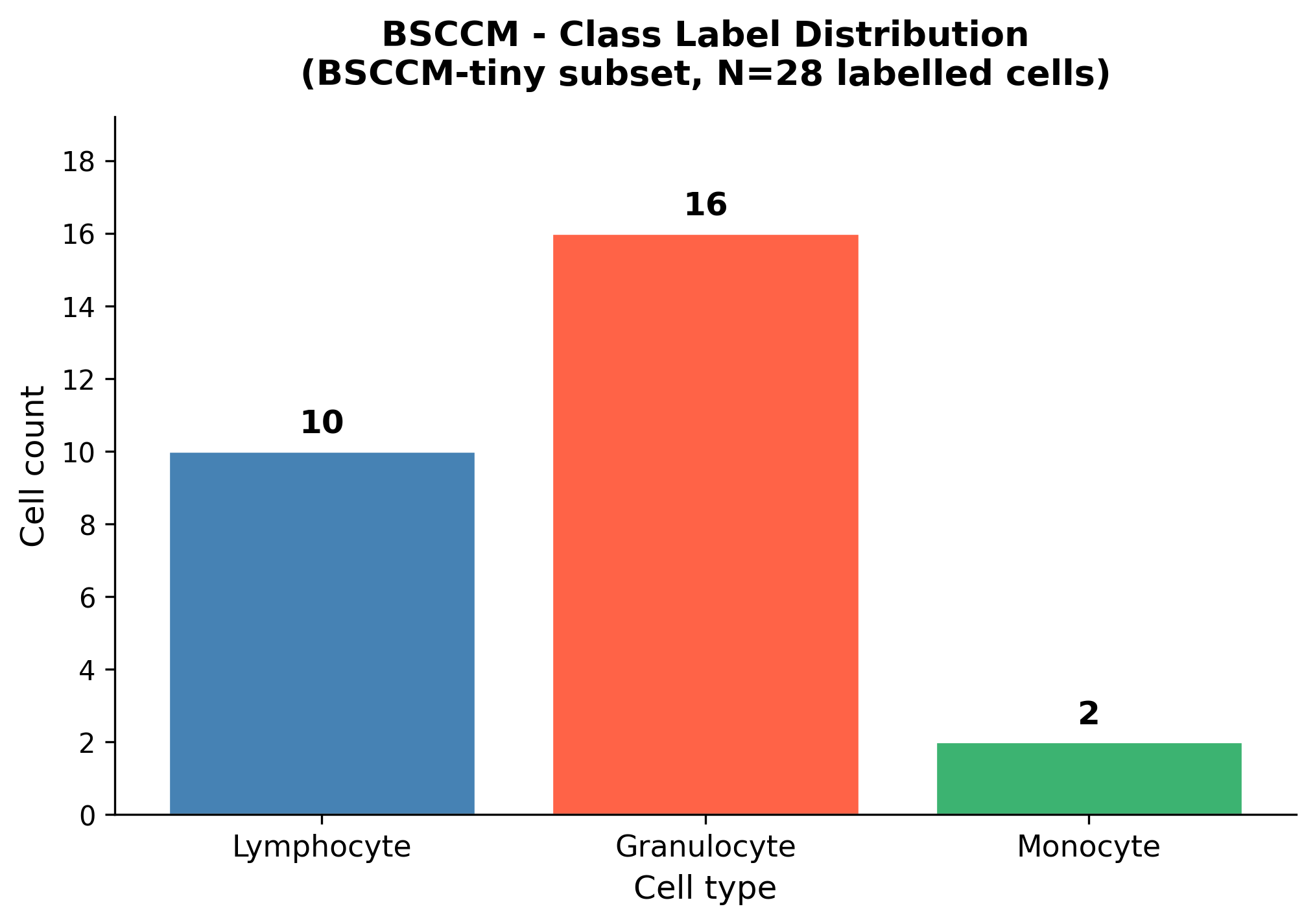}
		\caption{Class label distribution for the labeled subset of BSCCM-tiny
			($N_{\ell}=28$ cells). Lymphocyte: 10, Granulocyte: 16, Monocyte: 2.}
		\label{fig:label_dist}
	\end{figure}
	
	\subsection{Visual inspection of labeled cells across imaging channels}
	
	Figures~\ref{fig:bio_dpc_left}--\ref{fig:bio_brightfield} display the
	labeled cells across all five LED-array imaging channels.
	For each channel, two rows are shown per cell class: the original
	$128\times128$ raw patch (with the focused region indicated by a dashed
	white bounding box) and the focused crop extracted by the gradient-energy
	focus metric described in Section~\ref{sec:impl_details}.
	This presentation serves two purposes: it confirms that the focus
	selection step correctly isolates the cell body across all three
	morphological classes, and it provides a qualitative reference for
	the structural differences between classes that the learned dictionary
	is expected to capture.
	
	The four DPC channels (Left, Right, Top, Bottom) encode directional
	phase-gradient contrast: Lymphocytes exhibit a compact, smooth ring
	boundary; Granulocytes show a more heterogeneous interior with a
	multi-lobed nuclear structure; Monocytes display a larger, kidney-shaped
	nucleus with a broader cytoplasmic region.
	The Brightfield channel encodes integrated absorption contrast, where
	the cell body appears uniformly darker against the background, with
	class-specific morphological features less sharply delineated than in
	the DPC channels.
	
	\begin{figure}[h]
		\centering
		\includegraphics[width=\textwidth]{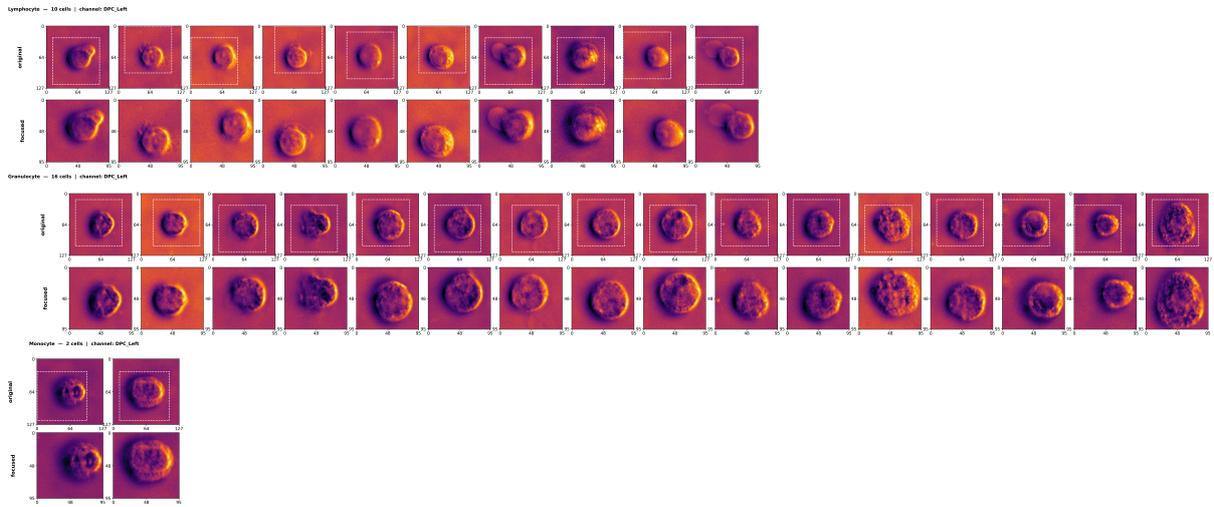}
		\caption{Labeled cells - channel: DPC Left.
			Rows alternate between original patch (with focused bounding box)
			and focused crop. Classes from top: Lymphocyte, Granulocyte, Monocyte.}
		\label{fig:bio_dpc_left}
	\end{figure}
	
	\begin{figure}[h]
		\centering
		\includegraphics[width=\textwidth]{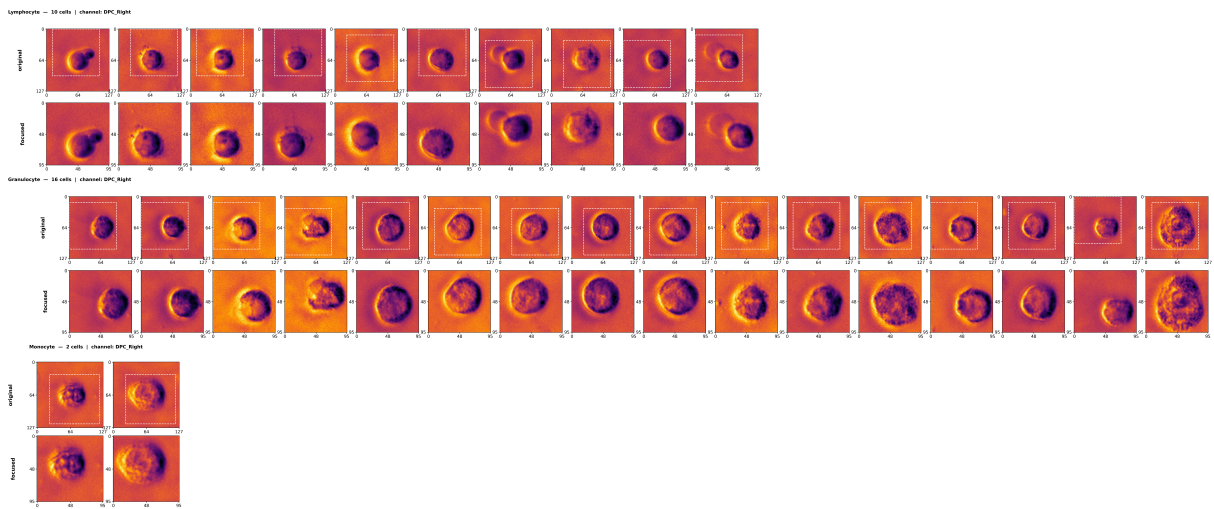}
		\caption{Labeled cells - channel: DPC Right.}
		\label{fig:bio_dpc_right}
	\end{figure}
	
	\begin{figure}[h]
		\centering
		\includegraphics[width=\textwidth]{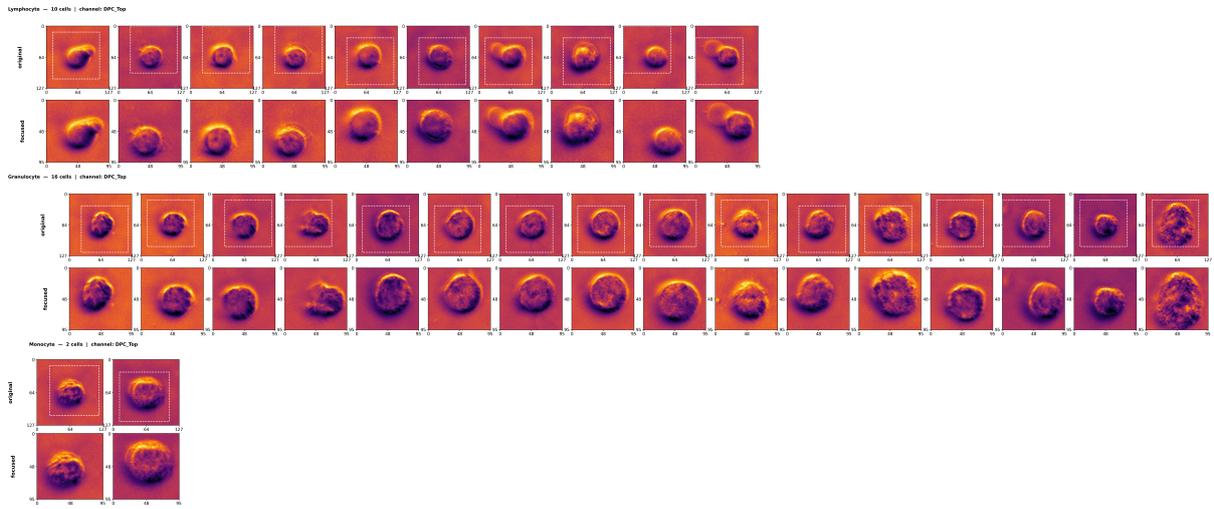}
		\caption{Labeled cells - channel: DPC Top.}
		\label{fig:bio_dpc_top}
	\end{figure}
	
	\begin{figure}[h]
		\centering
		\includegraphics[width=\textwidth]{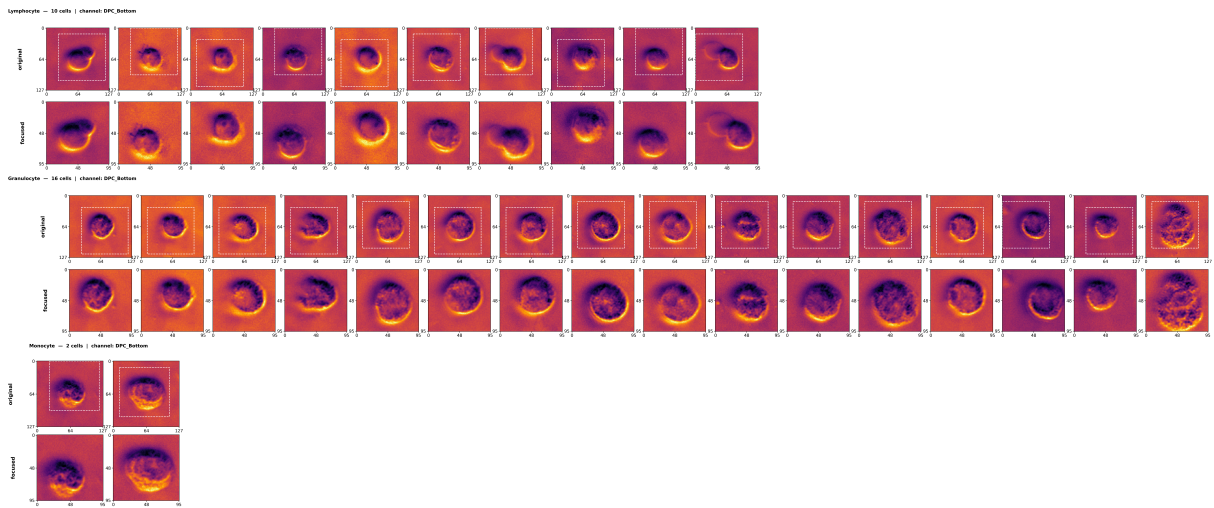}
		\caption{Labeled cells - channel: DPC Bottom.}
		\label{fig:bio_dpc_bottom}
	\end{figure}
	
	\begin{figure}[h]
		\centering
		\includegraphics[width=\textwidth]{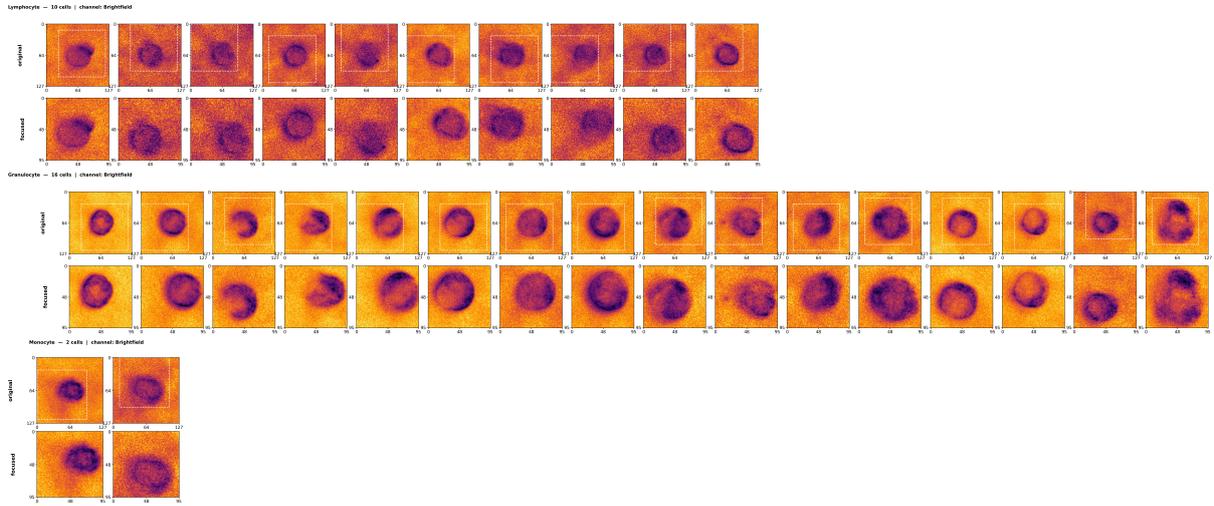}
		\caption{Labeled cells - channel: Brightfield.}
		\label{fig:bio_brightfield}
	\end{figure}
	
	\subsection{Clustering validation protocol}
	\label{sec:bio_valid_results}
	The biological validation objective is to assess whether the unified
	cell descriptor $\varphi_j \in \R^{CK}$ produced by Algorithm~2
	clusters the labeled cells in a manner consistent with the
	ground-truth class assignments.
	Three metrics are used: Adjusted Rand Index (ARI), Normalized Mutual
	Information (NMI), and mean silhouette score.
	A $k$-means clustering with $k=3$ is applied to the descriptor
	vectors $\{\varphi_j\}_{j \in \mathcal{L}}$, where
	$\mathcal{L}$ denotes the index set of labeled cells,
	and the resulting cluster assignments are compared against the
	ground-truth labels. We additionally report a $k=2$ configuration that
	aggregates the two myeloid populations (Granulocyte and Monocyte) into
	a single class and compares against Lymphocyte, giving a
	lymphoid-vs-myeloid two-class validation on $N=26$ cells (dropping the
	two Monocytes for which the three-class separation is
	statistically indistinguishable from chance).
	
	\paragraph{Consistency with the deterministic framework.}
	It is important to distinguish between the two components of this
	validation procedure. The descriptor vectors $\varphi_j$ are
	uniquely and reproducibly determined by Theorem~\ref{thm:existence_uniqueness},
	applied per channel (Remark~\ref{rem:per_channel_wellposed}): given the
	same trained family $\{D^{(c)}\}_{c=1}^{C}$ and the same cell image
	$\{x_j^{(c)}\}_{c=1}^{C}$, the descriptor $\varphi_j$ is always
	identical, regardless of when or where inference is performed.
	This is a mathematical guarantee, not a design preference, and it
	is the central property the biological validation is designed to
	exploit.
	The $k$-means clustering step, by contrast, is a post-hoc evaluation
	instrument rather than part of the deterministic pipeline.
	To ensure reproducibility of the evaluation itself, the clustering
	is fixed across runs via a deterministic seed and $n_{\mathrm{init}} = 20$
	independent restarts.
	The reported metrics therefore characterise the biological structure
	of the descriptors, not of the clustering procedure.
	
	\subsection{Clustering results}
	\label{subsec:bio_results}
	
	Table~\ref{tab:bio_validation} reports the clustering metrics for both
	the full concatenation descriptor $\varphi_j\in\R^{CK}$ and the
	auxiliary index-wise atom-norm descriptor $\psi_j\in\R^K$ defined in
	\eqref{eq:descriptor_psi}. Both descriptors are reduced to 15 principal
	components prior to clustering (explained variance $0.755$ for
	$\varphi_j$ and $0.832$ for $\psi_j$), with channel-wise $\ell_2$
	normalisation and per-atom standard scaling applied beforehand.

	\begin{table}[h]
		\centering
		\caption{Biological validation metrics on the BSCCM-tiny labelled
			subset ($N=28$: Lymphocyte$=10$, Granulocyte$=16$, Monocyte$=2$).
			$k{=}2$: lymphoid vs myeloid with the two Monocytes dropped
			($N=26$). $k{=}3$: all three classes ($N=28$).}
		\begin{tabular}{@{}llccc@{}}
			\toprule
			Descriptor & Setting & ARI & NMI & Silhouette \\
			\midrule
			$\varphi_j\in\R^{CK}$ & $k{=}2$ (Lymph vs Gran) & $0.575$ & $0.471$ & $0.074$ \\
			$\varphi_j\in\R^{CK}$ & $k{=}3$ (all classes)   & $0.038$ & $0.100$ & $0.077$ \\
			$\psi_j\in\R^{K}$      & $k{=}2$ (Lymph vs Gran) & $-0.015$ & $0.018$ & $0.136$ \\
			$\psi_j\in\R^{K}$      & $k{=}3$ (all classes)   & $0.214$ & $0.337$ & $0.112$ \\
			\bottomrule
		\end{tabular}
		\label{tab:bio_validation}
	\end{table}

	\begin{figure}[h]
		\centering
		\includegraphics[width=\textwidth]{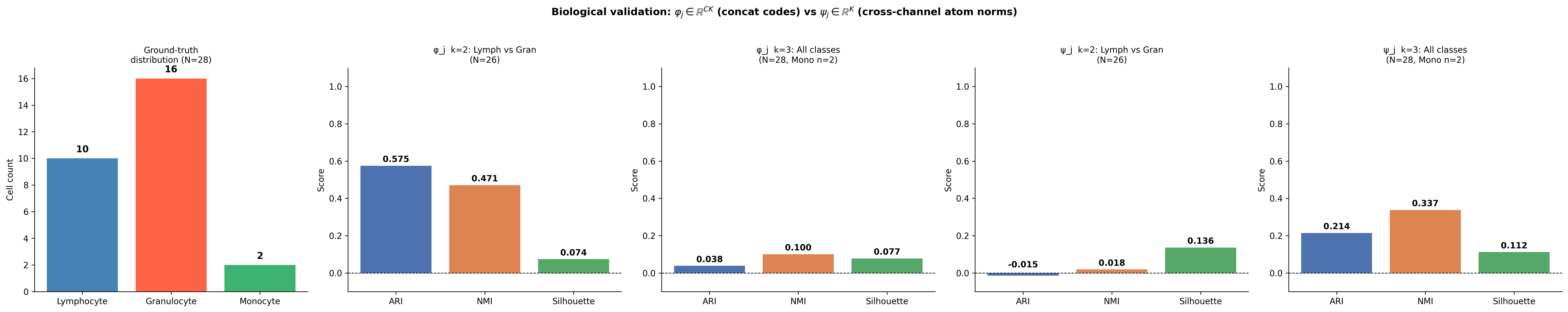}
		\caption{Biological validation metrics on the BSCCM-tiny labelled
			subset. Leftmost panel: ground-truth class distribution
			($N=28$: Lymphocyte $=10$, Granulocyte $=16$, Monocyte $=2$).
			Remaining panels: ARI, NMI, and mean silhouette score for
			$k$-means clustering of the $\varphi_j\in\R^{CK}$ concatenation
			descriptor and the $\psi_j\in\R^{K}$ index-wise atom-norm
			descriptor, each under $k{=}2$ (Lymphocyte vs Granulocyte,
			$N=26$) and $k{=}3$ (all three classes, $N=28$) configurations.
			The $\varphi_j$ two-class separation at ARI$=0.575$ is the
			principal unsupervised biological signal reported by the
			framework.}
		\label{fig:bio_validation}
	\end{figure}

	The $\varphi_j$ descriptor separates the lymphoid and myeloid populations
	at ARI$=0.575$, NMI$=0.471$ in the two-class configuration; this is a
	substantive separation given that the descriptor is fully unsupervised
	(no label information entered its construction). The three-class
	performance of $\varphi_j$ is essentially at chance (ARI$=0.038$), which
	we attribute to the extreme class imbalance (only $N=2$ Monocytes,
	insufficient to anchor a cluster in a 15-dimensional PCA subspace).

	\paragraph{Statistical significance of the two-class separation.}
	To assess whether the observed $k{=}2$ separation is statistically
	meaningful given the small labeled subset ($N=26$), we conducted
	a permutation test and a bootstrap confidence interval analysis,
	both using 1{,}000 resamples and the same preprocessing pipeline
	(channel-wise $\ell_2$ normalisation, per-atom standard scaling,
	PCA to 15 components).

	The \emph{permutation null distribution} was obtained by randomly
	shuffling the ground-truth labels 1{,}000 times and re-running
	$k$-means on the fixed descriptor vectors; this gives the ARI and
	NMI one would expect by chance given the class imbalance (10
	Lymphocytes, 16 Granulocytes). The null ARI has mean $0.000$ and
	95th percentile $0.110$; the null NMI has mean $0.034$ and 95th
	percentile $0.136$. The observed ARI$=0.575$ and NMI$=0.471$ both
	lie well above the respective 95th percentiles, giving permutation
	$p < 0.0001$ for each metric. The separation is therefore
	statistically significant at the 5\% level.

	Bootstrap 95\% confidence intervals (percentile method, 1{,}000
	resamples with replacement from the $N=26$ labeled cells) are
	ARI $\in [-0.11,\;0.71]$ and NMI $\in [0.00,\;0.68]$. The
	wide intervals are an honest consequence of the small labeled
	subset and should be reported as such; they do not contradict the
	permutation significance, which is assessed on the full $N=26$
	sample. The lower bound of the ARI interval reflects degenerate
	clusterings in bootstrap resamples where one class is
	underrepresented by chance---a known artefact of resampling from
	highly imbalanced small sets.

	The auxiliary descriptor $\psi_j$ exhibits a qualitatively different
	pattern: poor two-class separation (ARI$\approx 0$) but non-trivial
	three-class NMI$=0.337$. We interpret this as the index-wise
	atom-norm aggregation discarding the sign and channel-specific
	structure that makes lymphocytes and granulocytes distinguishable,
	while preserving a coarse activation-strength signature that
	differentiates all three classes at a low level. This is consistent
	with the caveat stated in Section~\ref{subsec:per_channel_unification}:
	under per-channel dictionaries, $\psi_j$ aggregates index-by-index
	rather than atom-by-atom, and the resulting signal is weaker than
	would be expected under a shared-dictionary architecture.

	\subsection{Noise gap between ground truth and reconstruction}
	\label{subsec:noise_gap}

	A visual comparison of the unified reconstructions $u_j$ against the
	BSCCM-labelled ground-truth cells reveals a systematic qualitative gap:
	the ground-truth Brightfield images exhibit shot-noise characteristics
	consistent with absorption-only acquisition physics, whereas the
	TV-regularised reconstruction produces a visibly smoother output
	(Figure~\ref{fig:unified_vs_truth}).
	The four DPC channel comparisons
	(Figures~\ref{fig:unified_vs_truth_dpc_left}--\ref{fig:unified_vs_truth_dpc_bottom})
	show a qualitatively different picture: the phase-gradient contrast structure,
	the directed membrane ring, and the directional illumination symmetry between
	the Left/Right and Top/Bottom pairs are all faithfully reproduced, with residual
	differences attributable to TV-induced smoothing of fine boundary texture
	rather than to loss of class-discriminating morphological structure.

	\begin{figure}[h]
		\centering
		\includegraphics[width=\textwidth]{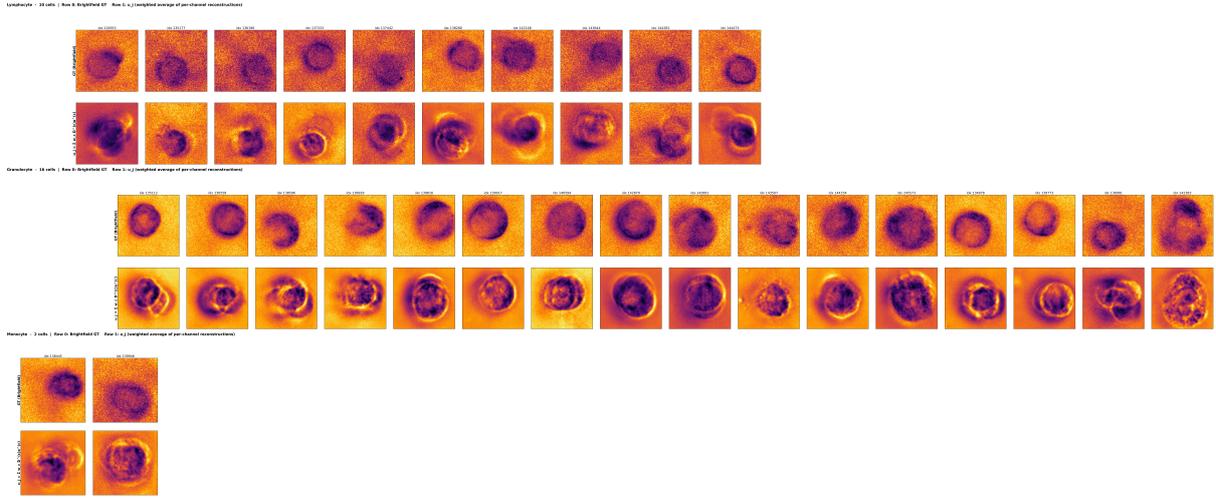}
		\caption{Brightfield ground truth (top row in each panel) versus
			unified reconstruction $u_j$ (bottom row in each panel) for the
			$N=28$ BSCCM-tiny labelled cells, grouped by expert-annotated
			cell type (Lymphocyte, $n=10$; Granulocyte, $n=16$; Monocyte,
			$n=2$). The ground-truth images exhibit visible shot noise
			consistent with absorption-only acquisition physics; the
			reconstructions are visibly smoother, reflecting the
			TV-regularised and non-negativity-constrained nature of the
			algorithm. See Section~\ref{subsec:noise_gap} for the
			quantitative decomposition of this gap.}
		\label{fig:unified_vs_truth}
	\end{figure}

	\begin{figure}[h]
		\centering
		\includegraphics[width=\textwidth]{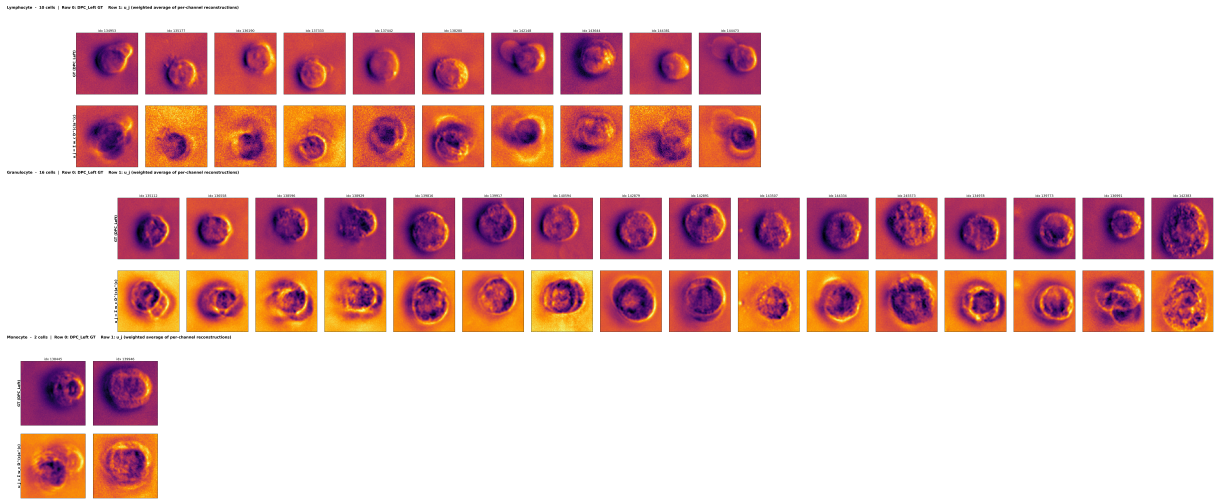}
		\caption{DPC Left ground truth versus unified reconstruction $u_j$ for the
			$N=28$ labelled cells, grouped by cell type. The phase-gradient
			contrast structure (directional illumination from the left) is faithfully
			reproduced; residual differences are attributable to TV-induced smoothing
			of the membrane boundary rather than to loss of morphological class structure.}
		\label{fig:unified_vs_truth_dpc_left}
	\end{figure}

	\begin{figure}[h]
		\centering
		\includegraphics[width=\textwidth]{unified_vs_truth_DPC_Right_25-04-2026-15-58-35.png}
		\caption{DPC Right ground truth versus unified reconstruction $u_j$
			($N=28$ labelled cells). The antisymmetric gradient contrast relative to
			DPC Left is preserved in the reconstruction, confirming that the per-channel
			dictionaries respect the illumination symmetry of the LED-array acquisition.}
		\label{fig:unified_vs_truth_dpc_right}
	\end{figure}

	\begin{figure}[h]
		\centering
		\includegraphics[width=\textwidth]{unified_vs_truth_DPC_Top_25-04-2026-15-58-35.png}
		\caption{DPC Top ground truth versus unified reconstruction $u_j$
			($N=28$ labelled cells). The vertical phase-gradient contrast
			(illumination tilt along the top-bottom axis) is recovered with
			mean fidelity $97.23\%$ at the population level (Table~\ref{tab:fidelity}).}
		\label{fig:unified_vs_truth_dpc_top}
	\end{figure}

	\begin{figure}[h]
		\centering
		\includegraphics[width=\textwidth]{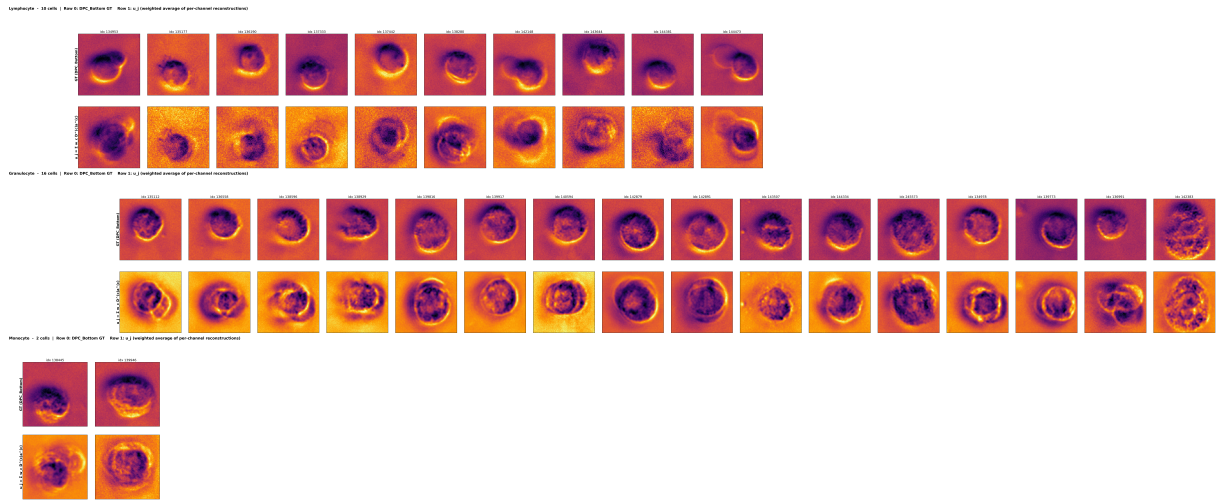}
		\caption{DPC Bottom ground truth versus unified reconstruction $u_j$
			($N=28$ labelled cells). Together with the DPC Top panel
			(Figure~\ref{fig:unified_vs_truth_dpc_top}), this confirms symmetric
			recovery of the orthogonal phase-gradient pair, with the cell membrane
			ring sharply delineated in both reconstructions.}
		\label{fig:unified_vs_truth_dpc_bottom}
	\end{figure}

	This gap is inherent to the algorithm design. The TV penalty in the
	energy~\eqref{energy-func} explicitly promotes piecewise-smooth
	reconstructions; the non-negativity constraint further regularises the
	output. The ground-truth images are direct measurements without any
	denoising step. Two distinct phenomena therefore contribute to the
	quantitative gap between DPC reconstruction fidelity (97\% on each of
	the four DPC channels) and Brightfield fidelity (94.79\%):
	\begin{enumerate}
		\item the genuine physical noise floor of absorption-only imaging,
			which sets a hard ceiling on how accurately any denoising-based
			algorithm can reproduce a single noisy realisation of the
			channel;
		\item the algorithmic smoothness prior induced by the TV penalty
			and non-negativity constraint, which removes noise but in
			principle can also remove fine-grained morphological detail
			carrying biological signal.
	\end{enumerate}
	Discriminating quantitatively between (i) and (ii) requires a controlled
	experiment in which the TV parameter $\lambda_{\mathrm{TV}}$ is varied
	against a high-fidelity denoised ground truth; this is beyond the scope
	of the present manuscript. A partial diagnostic is provided by the
	observation that the three-class $\psi_j$ clustering (which discards
	channel structure) attains NMI$=0.337$ while the three-class
	$\varphi_j$ clustering (which retains channel structure) attains only
	NMI$=0.100$: if fine-grained morphological detail were being
	systematically removed by the smoothness prior, we would expect the
	channel-aggregated descriptor to perform \emph{worse} rather than
	better, which is not what is observed. This suggests the bulk of the
	Brightfield fidelity gap is attributable to (i), the physical noise
	floor, rather than (ii), over-smoothing.

	Future work should examine the sensitivity of biological
	discriminability to $\lambda_{\mathrm{TV}}$ and should compare
	BSCCM Brightfield-derived labels against orthogonal ground truth
	(fluorescence or expert annotation) to disentangle the two
	contributions rigorously.

	\section{Discussion}
	The unitarity requirement~\eqref{eq:unitary} carries mathematical weight beyond
	notational convenience: it is the cornerstone of both identifiability and numerical
	conditioning in the learning loop.
	The hybrid penalty in~\eqref{energy-func} serves two complementary physical
	purposes: TV regularization explicitly penalizes spatial gradients, preserving
	cell boundaries and suppressing the low-frequency banding artifacts that arise
	with pure $\ell_1$ sparsity; the non-negativity constraint $\iota_+$ encodes the
	physical fact that microscopy intensities cannot be negative, grounding the
	reconstruction in the measurement model.
	In single-cell microscopy, where illumination non-uniformity and background
	variation are pervasive, this combination reduces the tendency of unconstrained
	dictionaries to absorb spurious high-frequency patterns into the learned atoms.
	
	\section{Conclusion}
	We presented a variational dictionary learning algorithm with hybrid least-squares-TV
	penalization, non-negativity constraints, and an explicit unitary dictionary constraint ($D^{\top}D=I$).
	The manuscript establishes a rigorous mathematical formulation with three convergence
	layers: (i) existence, uniqueness, and Lipschitz stability of the variational minimizer
	(Theorem~\ref{thm:existence_uniqueness}); (ii) strong convergence of PDHG iterates to
	the regularized solution under the explicit step-size condition $\tau\sigma < 1/8$
	(Theorems~\ref{thm:pdhg_convergence} and~\ref{thm:algo_to_regularized}); and (iii)
	convergence of the regularized solution to the true solution at the $O(\delta)$ rate
	when the true solution satisfies the quadratic VSC and the regularization parameters
	are chosen as $\lambda \propto \delta$ (Theorems~\ref{thm:vsc_rate}
	and~\ref{thm:algo_to_true}).  The combined total error satisfies
	$\norm{y^{n(\delta)} - y^\dagger}_2 \le C_{\mathrm{tot}}\,\delta$ with an explicit
	constant $C_{\mathrm{tot}} = C_1 + (1-c)^{-1}$ determined by the step sizes and the
	VSC geometry.
	
	A key second contribution is the proposed deterministic framework for multi-channel cell feature
	unification (Section~\ref{sec:unification}). By learning a family of per-channel unitary
	dictionaries $\{D^{(c)}\}_{c=1}^{C}$, each adapted to the optical physics of its
	channel, and concatenating the resulting per-channel sparse codes in a fixed
	channel ordering, we obtain a unified cell descriptor
	$\varphi_j \in \mathbb{R}^{CK}$ that is mathematically grounded, reproducible, and directly interpretable.
	No cross-channel atom correspondence is assumed; the descriptor aggregates information
	from all channels into a single deterministic vector suitable for downstream classification.
	On BSCCM-tiny ($N=1{,}000$ cells, $K=512$ atoms per channel) the framework reaches
	per-channel reconstruction fidelities of $97.06$--$97.54\%$ on the four DPC channels and
	$94.79\%$ on Brightfield, with bit-identical iterates across independent runs; the
	unsupervised descriptor separates lymphoid and myeloid populations at ARI$=0.575$,
	NMI$=0.471$ (permutation $p<0.0001$; bootstrap 95\% CI: ARI $[-0.11,\,0.71]$)
	on the $N=26$ expert-labelled two-class subset (Section~\ref{sec:bio_valid}).
	This deterministic approach is preferred over probabilistic latent space methods \cite{scVI2018,totalVI2021}
	in the clinical imaging context, where reproducibility and auditability are regulatory requirements.
	
	Together, these two contributions, i.e. hybrid TV-plus-non-negativity regularisation for single-channel
	reconstruction, and per-channel dictionaries with concatenation-based multi-channel unification,
	establish a rigorous and reproducible foundation for variational single-cell analysis, covering
	reconstruction, convergence, and multi-channel feature unification. By targeting a modality,
	image-based single-cell morphology, that is already the primary substrate of routine
	haematological diagnostics~\cite{Acevedo2019CNNblood, Matek2019AMLmorphology, Kratz2005CellaVision},
	the framework contributes a mathematically auditable alternative to the black-box neural
	classifiers that currently dominate this clinical space.
	Recent independent work~\cite{Donhauser2024DictLearnMicroscopy} approaches the same
	black-box concern from the opposite direction, applying sparse dictionary learning
	post hoc to the latents of an already-trained microscopy foundation model;
	the present work builds interpretability \emph{into} the representation by constructing
	the dictionary directly from raw pixels under a convex variational objective with
	closed-form reconstruction guarantees.
	Taken together, the two approaches suggest that sparse dictionary formulations are emerging
	as a shared response to the interpretability demands that clinical applications of AI
	in bioimaging will increasingly need to meet.
	
	\bibliographystyle{plain}



\section*{Acknowledgements}
The author thanks Henry Pinkard and Laura Waller (Waller Lab, University of
California, Berkeley) for helpful correspondence regarding the BSCCM dataset,
in particular for clarifying the availability and structure of the
cell-type classification labels accessible via the \texttt{bsccm} Python
package, which made the biological validation in
Section~\ref{sec:bio_valid} self-sufficient.
The author also thanks Bertram K\"{o}nig (Sonovum GmbH, Leipzig) for
stimulating discussions on regulatory requirements for Software as a Medical
Device under EU IVDR 2017/746, which informed the clinical framing of this work.

\section*{Conflict of Interest Statement}
The author declares no conflict of interest.
This work was conceived, developed, and carried out independently by the author
in his capacity as founder of Aegis Digital Technologies,
a sole proprietorship registered in Dresden, Germany.
No external funding was received for this research.
No competing financial interests, advisory relationships, or institutional affiliations
that could have influenced the design, conduct, or reporting of this work exist.


\section*{Data Access Statement}
All experiments in this manuscript were conducted on the publicly available
Berkeley Single Cell Computational Microscopy (BSCCM) dataset~\cite{bsccm},
specifically the BSCCM-tiny subset comprising $N=1{,}000$ white blood cells
imaged in five LED-array channels at $128\times128$ pixel resolution.
The dataset is freely accessible via DOI:~\texttt{10.5061/dryad.sxksn038s} and 
at \url{https://waller-lab.github.io/BSCCM/}.
No new experimental data were generated in this work.

\section*{Code Availability}
The Python implementation of the proposed algorithm is openly available at
\url{https://github.com/erdemaltuntac/Micro-SDP-dev}.
The code is released under the Apache License 2.0.

\section*{Ethics Statement}
This work is purely computational and does not involve the collection of new
human subjects data, biological material, or animal experiments.
The BSCCM dataset used in the experiments was collected and published by
Pinkard et al.\ (2024)~\cite{bsccm} under the oversight of the Institutional
Review Board of the University of California, Berkeley.
All cells in the dataset are anonymised label-free microscopy images of
white blood cells; no personally identifiable information is present.
No additional ethics approval was required for the computational analyses
reported in this manuscript.

\end{document}